\documentclass[10pt, a4paper, reqno, oneside]{amsart}

\usepackage{amsmath, amsfonts, amsthm, amssymb, mathtools, enumerate, multicol, scalefnt, relsize}
\usepackage[mathscr]{euscript}
\usepackage{mathbbol}
\usepackage[latin1]{inputenc}
\usepackage{graphicx}
\usepackage[all]{xy}
\usepackage{enumitem}
\usepackage{autobreak,lipsum}
\setlist[itemize]{noitemsep, topsep=1pt, leftmargin=20pt}
\usepackage{xfrac}
\usepackage{todonotes}

\usepackage{fullpage}

\usepackage[hypertexnames=false,
backref=page,
    pdfpagemode=UseNone,
    breaklinks=true,
    extension=pdf,
    colorlinks=true,
    linkcolor=blue,
    citecolor=blue,
    urlcolor=blue,
]{hyperref}

\newcommand\bcdot{\ensuremath{
  \mathchoice
   {\mskip\thinmuskip\lower0.2ex\hbox{\scalebox{1.6}{$\cdot$}}\mskip\thinmuskip}}
   {\mskip\thinmuskip\lower0.2ex\hbox{\scalebox{1.6}{$\cdot$}}\mskip\thinmuskip}
   {\lower0.3ex\hbox{\scalebox{1.2}{$\cdot$}}}
   {\lower0.3ex\hbox{\scalebox{1.2}{$\cdot$}}}
}
\theoremstyle{plain}
\newtheorem{theo}{Theorem}[section]

\theoremstyle{definition}

\newtheorem{example}[theo]{Example}
\newtheorem{definition}[theo]{Definition}

\theoremstyle{plain}
\newtheorem{lemma}[theo]{Lemma}
\newtheorem{theorem}[theo]{Theorem}
\newtheorem{corollary}[theo]{Corollary}
\newtheorem{proposition}[theo]{Proposition}

\theoremstyle{definition}

\newtheorem{remark}[theo]{Remark}

\theoremstyle{plain}
\newtheorem{thmint}{Theorem}

\newtheorem{propint}[thmint]{Proposition}

\theoremstyle{definition}
\newtheorem*{definition*}{Definition}

\DeclareSymbolFontAlphabet{\mathbb}{AMSb}
\DeclareSymbolFontAlphabet{\mathbbl}{bbold}

\makeatletter
\@namedef{subjclassname@2020}{\textup{2020} Mathematics Subject Classification}
\makeatother

\allowdisplaybreaks

\title[]{Toral symmetries of collapsed ancient solutions \\ to the homogeneous Ricci flow}

\author{Anusha M.\ Krishnan}
\address[Anusha M.\ Krishnan]{Department of Mathematics, Indian Institute of Technology Bombay \\ Powai, Mumbai-400076, India}
\email{anushamk@math.iitb.ac.in}

\author{Francesco Pediconi}
\address[Francesco Pediconi]{Dipartimento di Scienze Matematiche ``Giuseppe Luigi Lagrange'' \\ Politecnico di Torino, corso Duca degli Abruzzi 24, 10129 Torino, Italy}
\email{francesco.pediconi@polito.it}

\author{Sammy Sbiti}
\address[Sammy Sbiti]{Department of Mathematics, University of Pennsylvania, 209 South 33rd St, Philadelphia, PA, 19104-6395, USA}
\email{sammysbiti@gmail.com}

\thanks{The first-named author was funded by the Deutsche Forschungsgemeinschaft (DFG, German Research Foundation) under Germany's Excellence Strategy EXC 2044 -- 390685587, Mathematics M\"{u}nster: Dynamics--Geometry--Structure. The second-named author is member of GNSAGA of INdAM and has been supported by the PRIN 2022 project "Real and Complex Manifolds: Geometry and holomorphic dynamics" (code 2022AP8HZ9).}

\begin{document}

\begin{abstract}
Collapsed ancient solutions to the homogeneous Ricci flow on compact manifolds occur only on the total space of principal torus bundles. Under an algebraic assumption that guarantees flowing through diagonal metrics and a tameness assumption on the collapsing directions, we prove that such solutions have additional symmetries, i.e., they are invariant under the right action of their collapsing torus. As a byproduct of these additional torus symmetries, we prove that these solutions converge, backward in time, in the Gromov-Hausdorff topology to an Einstein metric on the base of a torus bundle.
\end{abstract}

\maketitle

\section{Introduction}

Ricci flow solutions that are defined for all negative times have a special significance and are called {\it ancient}. Ancient solutions appear as blow-up limits of finite-time singularities for the flow, and so they play a crucial role in the analysis of solutions near their extinction time. They have been classified and, as a result, shown to be rotationally symmetric on compact surfaces \cite{DHS12} and in higher dimensions under non-collapsing and positive curvature assumptions \cite{BDS21,BDNS23}.
Many other instances in the literature also appear where ancient solutions turn out to have more symmetries than initially assumed \cite{CaoSC09, BHS11, B14, BK21, DH21, ABDS22, Lai22, Sb22, BN23}.
Furthermore, by simplifying the evolution equation, symmetry assumptions play an important role in the construction of examples (see, e.g., \cite{King90, Ros95, FOZ93, Perel02, BKN12, Ta14, LuWa17, Lai22, BBLL23}). \smallskip

In this paper, we investigate whether this tendency for ancient solutions to be more symmetric holds true in the compact homogeneous setting, where any ancient solution is Type I \cite{BLS19} and the Ricci flow has special structural properties (see, e.g., \cite{Lau13, Buz14, Laf15, Boe15, BLS19, AC21, Sb22, GMPSS22}). In this case, if an ancient solution is non-collapsed, then the backward limit is unique and, in fact, a homogeneous Einstein metric on the same manifold \cite{BLS19}.  Thus, this picture is well-understood.  Here, we look at the collapsed case, where along any sequence of times approaching negative infinity, the metrics collapse with bounded curvature \cite{BLS19}. As expected from the foundational work on collapsed Riemannian manifolds in \cite{ChGr86, ChGr90, ChFuGr92}, these metrics asymptotically develop additional toral symmetries in a precise sense (see \cite{Ped19} and Section \ref{sect:limitbehavior}), and this motivates our investigation. As a further motivation, we show that the additional structure provided by these toral symmetries leads to the existence of backward limit solitons for collapsed ancient solutions. Here, this limit is, in fact, a homogeneous Einstein metric on a manifold of lower dimension. \smallskip

From now on, we denote by $M = \mathsf{G}/\mathsf{H}$ an almost-effective homogeneous space with $\mathsf{G}$, $\mathsf{H}$ compact and connected Lie groups, and by $N_{\mathsf{G}}(\mathsf{H})$ the normalizer of $\mathsf{H}$ in $\mathsf{G}$. By \cite[Corollary I.4.3]{Br72}, $N_{\mathsf{G}}(\mathsf{H})/\mathsf{H}$ is isomorphic to the {\it gauge group} of $\mathsf{G}$-equivariant diffeomorphisms of $M$. By \cite{BWZ04} and \cite[Remark 5.13]{BLS19}, $M$ admits a collapsed ancient solution only if it is the total space of a homogeneous torus bundle, i.e., only if $\dim(N_{\mathsf{G}}(\mathsf{H})/\mathsf{H})>0$. In fact, by adapting \cite{Ped19} to curvature-normalized solutions $\bar{g}(t) \coloneqq \frac1{|t|}g(t)$, we know more.

\begin{thmint}[c.f. \cite{Ped19}] \label{thm:main-existence.torus}
Let $g(t)$ be a collapsed, ancient solution to the homogeneous Ricci flow on a compact ma\-ni\-fold and $\xi = \{t^{(n)}\}$ a sequence of times with $t^{(n)} \to -\infty$. Then, up to passing to a subsequence, the collapsing directions of the rescaled metrics $\frac1{|t^{(n)}|}g(t^{(n)})$ converge to a limit distribution induced by the (possibly locally-defined) right action of a torus $\mathsf{T}_{\xi}$.
\end{thmint}

Here, $\mathsf{T}_{\xi}$ is a connected, abelian subgroup of $N_{\mathsf{G}}(\mathsf{H})/\mathsf{H}$, possibly non-closed. Moreover, the right action of $\mathsf{T}_{\xi}$ gets closer to becoming $g(t^{(n)})$-isometric, as $n \to +\infty$, in a precise sense (see Theorem \ref{thm:limittorus}). We remark that, while there must exist a {\it collapsing torus} $\mathsf{T}_{\xi} \subset N_\mathsf{G}(\mathsf{H})/\mathsf{H}$, the collapsing directions could potentially oscillate within the gauge group $N_\mathsf{G}(\mathsf{H})/\mathsf{H}$ as $t \to -\infty$. A priori, this allows for the possibility that different tori collapse along different sequences going to $-\infty$ (see Section \ref{sect:limitbehavior}). \smallskip

Examples of collapsed ancient solutions have been found on a case-by-case basis in \cite{CaoSC09, BKN12, Buz14, LuWa17, BLS19, Sb22}, and, more recently, in \cite{PedSb22} the second- and third-named authors proved a general existence result. As a matter of fact, in all known examples, the collapsing torus $\mathsf{T}$ is unique, i.e., it does not depend on the sequence of times, and it is compact. Moreover, the curvature-normalized metrics shrink the fibers of the principal torus bundle $\mathsf{T} \to \mathsf{G}/\mathsf{H} \to \mathsf{G}/\mathsf{H}\mathsf{T}$ and converge to an Einstein metric on the base $\mathsf{G}/\mathsf{H}\mathsf{T}$ in the Gromov-Hausdorff topology. Furthermore, all known examples have additional symmetries; in particular, they are invariant under the right action of $\mathsf{T}$ for all time.

It is not known whether the phenomenon of additional toral symmetries characterizes all collapsed ancient solutions on compact homogeneous spaces. We investigate this question when solutions evolve through diagonal metrics, i.e., their eigenspaces are time-independent. This happens, e.g., if the isotropy representation of $\mathsf{G}/\mathsf{H}$ does not contain any pair of equivalent irreducible submodules, an assumption that widely occurs in the literature. More generally, it is guaranteed by the existence of an {\it NR-decomposition} for the isotropy representation of $\mathsf{G}/\mathsf{H}$ (see Definition \ref{def:stronglyRd}). The main result of our paper is the following.

\begin{thmint} \label{thm:main-extra.symm}
Let $g(t)$ be a collapsed ancient solution to the homogeneous Ricci flow on a compact manifold such that $g(-1)$ is diagonal with respect to an NR-decomposition. If the collapsing torus $\mathsf{T}$ of this solution is unique, then $g(t)$ is invariant under the right action of $\mathsf{T}$ for all $t$.
\end{thmint}

Note that, if $M = \mathsf{G}/\mathsf{H}$ does not have any pair of equivalent irreducible submodules, the conclusion of Theorem \ref{thm:main-extra.symm} follows by applying Schur's Lemma. However, for such spaces, $\dim(N_{\mathsf{G}}(\mathsf{H})/\mathsf{H})\leq 1$, and so this class is rather restricted. This motivates the introduction of NR-decompositions that generalize work of \cite{Pay10, LauW13, K21}. In particular, Theorem \ref{thm:main-extra.symm} applies to all homogeneous spheres, for which the result was proved in an {\it ad hoc} manner in \cite{CaoSC09,Sb22}. Furthermore, Theorem \ref{thm:main-extra.symm} applies to a large class of examples including, e.g., the orthogonal groups $\mathsf{SO}(n)$ and their quotients $\mathsf{SO}(n) / \Pi_i\mathsf{SO}(p_i)$ with $\sum_i p_i \leq n-1$, the manifolds $\mathsf{Sp}(n) / \Pi_i\mathsf{Sp}(p_i)$ with $\sum_i p_i = n-1$, the Aloff-Wallach spaces and many others (see Section \ref{sect:example-sRd}). Finally, we point out that no curvature assumption is needed for Theorem \ref{thm:main-extra.symm}.

To prove Theorem \ref{thm:main-extra.symm}, we quantify the symmetries induced by the right action of the collapsing torus $\mathsf{T}$ in terms of the {\it submersion tensor} (see Section \ref{sect:mainproof}). This tensor $S$ is defined as the $g$-symmetric part of the infinitesimal action of $\mathsf{T}$ and vanishes if and only if the right $\mathsf{T}$-action on $M$ is $g$-isometric.  Results from \cite{Ped19}, coupled with our assumptions, imply that $S(\mathsf{T},g(t)) \to 0$ as $t \to -\infty$, which can be interpreted as the right action of $\mathsf{T}$ becoming closer to being $g(t)$-isometric as $t \to -\infty$.  Our main new technical result here is that, under the assumptions of the theorem, a monotonicity result holds for the norm $|S|$ of the submersion tensor along the solution $g(t)$, see Theorem \ref{thm:mainestimate}. These two facts then imply that $S$ is identically zero along the flow. \smallskip

Although it holds in all known examples, the uniqueness of the collapsing torus is a non-trivial question. In the following result, we describe a case in which the uniqueness is guaranteed.

\begin{propint} \label{prop:main-uniq.torus}
Let $g(t)$ be a collapsed ancient solution to the homogeneous Ricci flow on a compact manifold $M=\mathsf{G}/\mathsf{H}$ such that $g(-1)$ is diagonal with respect to an NR-decomposition. If $N_{\mathsf{G}}(\mathsf{H})/\mathsf{H}$ has rank $1$, then the collapsing torus $\mathsf{T}$ of $g(t)$ is unique.
\end{propint}

We remark that the hypotheses of Proposition \ref{prop:main-uniq.torus} are automatically satisfied in dimension $3$, where the conclusion has already been established in \cite{CaoSC09} via a case-by-case analysis. \smallskip

Finally, we investigate the limiting behavior of ancient solutions in the presence of right toral symmetries. By \cite[Theorem 3.1]{Na10} and \cite[Theorem 4.1]{CaoZh11}, if an ancient solution is non-collapsed, then there exists a sequence of times, going to $-\infty$, such that the corresponding blow-down sequence converges in the Cheeger-Gromov topology to a gradient shrinking Ricci soliton. Moreover, by \cite[Theorem 1.1]{PetWyl09}, all homogeneous gradient Ricci solitons are rigid, i.e., locally isometric to the product of an Einstein manifold with flat Euclidean space. However, when collapse happens along the flow, the results of \cite{Na10, CaoZh11} do not apply, and one cannot hope for a smooth limit in the classical Cheeger-Gromov topology. However, we prove the following.

\begin{thmint} \label{thm:main-coll.solitons}
Let $g(t)$ be a collapsed ancient solution to the homogeneous Ricci flow on a compact manifold $M$, such that its collapsing torus $\mathsf{T}$ is unique and maximal. If $g(t)$ is right $\mathsf{T}$-invariant, then the rescaled metrics $\frac1{|t|}g(t)$ converge in Gromov-Hausdorff topology to an Einstein metric on $M/\mathsf{T}$ as $t \to -\infty$.
\end{thmint}

Here, we do not assume the solution $g(t)$ to be diagonal. Additionally, we do not strictly require maximality of $\mathsf{T}$, which is just needed to guarantee that the horizontal distribution of the Riemannian submersion $(M, g(t)) \to M/\mathsf{T}$ is time-independent (see Theorem \ref{thm:limitEin-rigid}).

To prove Theorem \ref{thm:main-coll.solitons}, we introduce a new functional (see \eqref{eq:functional}) in the spirit of \cite{Lott10}, where immortal Ricci flow solutions were studied under local abelian symmetries. Our functional is defined on the space of metrics on $M$ that are both left $\mathsf{G}$-invariant and right $\mathsf{T}$-invariant. It is scale-invariant, monotone along the Ricci flow, and its critical points detect $\mathsf{G}$-invariant Einstein metrics on the quotient space $M/\mathsf{T}$. This allows us to prove the existence of a limit Einstein metric on $M/\mathsf{T}$. Uniqueness of the limit follows from a result of Lojasiewicz on analytic gradient flows, applied to an appropriate gauge-normalization of the Ricci flow introduced in \cite{PedSb22}.

To the best of our knowledge, Theorem \ref{thm:main-coll.solitons} is the first theoretical existence result for backward limit solitons of collapsed ancient solutions, which does not {\it presume} the existence of an Einstein metric on the base beforehand. Results concerning the existence of solitons for collapsed immortal solutions appear, e.g., in \cite{Lott07, Lott10, BLS18}. A different approach would be needed to drop the toral symmetry assumption in Theorem \ref{thm:main-coll.solitons}. This will be the subject of further studies. \medskip

The paper is organized as follows. Section \ref{sect:prel} gathers background material on compact homogeneous spaces and homogeneous ancient Ricci flow solutions. Section \ref{sect:limitbehavior} describes the collapsing tori of homogeneous ancient solutions. The proof of Theorem \ref{thm:main-coll.solitons} is presented in Section \ref{sect:proof-thmsolitons}. Section \ref{sect:stronglyRd} introduces NR-decompositions and their algebraic consequences, and it contains the proof of Proposition \ref{prop:main-uniq.torus}. Section \ref{sect:mainproof} is devoted to the proof of Theorem \ref{thm:main-extra.symm}. In Section \ref{sect:example-sRd}, we present a large class of homogeneous spaces where Theorem \ref{thm:main-extra.symm} and Theorem \ref{thm:main-coll.solitons} can be applied. Appendix \ref{subsect:roots} collects some basic facts about root spaces that get used in constructing examples in Section \ref{sect:example-sRd}. Finally, Appendix \ref{sect:proofsAGAG} proves Theorem \ref{thm:main-existence.torus}.
\medskip

\noindent {\itshape Acknowledgements.\ } The authors thank Christoph B{\"o}hm, Ramiro Lafuente, Jorge Lauret, and Wolfgang Ziller for their helpful comments and suggestions. The authors also thank the anonymous referees for their careful reading and helpful remarks.


\section{Preliminaries}
\label{sect:prel} \setcounter{equation} 0

\subsection{Homogeneous ancient solutions to the Ricci flow} \label{subsect:homRF} \hfill \par

Let $(M,g_o)$ be a compact Riemannian manifold. The Ricci flow is the geometric PDE 
\begin{equation} \label{eq:RF}
\partial_t g(t) = -2{\rm Ric}(g(t)) \,\, , \quad g(t_o) = g_o \,\, .
\end{equation}
Assume now that $\mathsf{G}$ is a compact Lie group acting isometrically on $(M,g_o)$. Since by \cite{Kot10} the isometry group remains unchanged under the flow, equation \eqref{eq:RF} restricts to a flow on the space $\EuScript{M}^{\mathsf{G}}$ of $\mathsf{G}$-invariant metrics. In this paper, we are interested in the case when the $\mathsf{G}$-action is transitive. Consequently, equation \eqref{eq:RF} becomes a dynamical system, and $\EuScript{M}^{\mathsf{G}}$ is finite-dimensional.

If $g_o \in \EuScript{M}^{\mathsf{G}}_1$, i.e. it has unit volume, then the volume-normalized Ricci flow $\tilde{g}(t)$ starting at $g_o$ satisfies the ODE
\begin{equation} \label{eq:volnormRF}
\tilde{g}'(t) = -2{\rm Ric}^0(\tilde{g}(t)) \,\, ,
\end{equation}
where ${\rm Ric}^0(g)$ denotes the traceless Ricci tensor of $g$. It is well known that the equation \eqref{eq:volnormRF} preserves $\EuScript{M}^{\mathsf{G}}_1$ and is equivalent to the Ricci flow, up to time-dependent rescaling and time reparametrization. Moreover, the volume-normalized Ricci flow \eqref{eq:volnormRF} coincides, up to a positive constant, with the $L^2$-gradient of the restricted scalar curvature functional on $\EuScript{M}^{\mathsf{G}}_1$ (see, e.g., \cite[Proposition 4.17]{Bes08}). This fact has remarkable consequences on the geometry of ancient solutions to the flow, i.e., those solutions $g(t)$ which are defined for any $t \in (-\infty,0)$. In the rest of the paper, we use $t_o=-1$ for formulating those conditions that may fail at the forward extinction time. We refer to \cite[Section 2.3]{PedSb22} for a more extended presentation of this topic. Here, we summarize the main facts. \smallskip

Let $g(t)$ be the solution to \eqref{eq:RF} with initial metric $g_o = g(-1) \in \EuScript{M}^{\mathsf{G}}_1$ and $\tilde{g}(t)$ the corresponding solution to \eqref{eq:volnormRF}. Notice that, by \cite[Theorem 3.2]{Sb22}, $g(t)$ being ancient is equivalent to $\tilde{g}(t)$ being ancient. By the maximum principle, if $g(t)$ is ancient, then ${\rm scal}(g(t))$ is positive and monotonically increasing (see, e.g., \cite[p.\ 102]{CLN06}) and so the same holds true for the scalar curvature of $\tilde{g}(t)$ as well. Moreover, by \cite{BLS19}, there exists a constant $c>0$, depending only on the dimension $m$ and on ${\rm scal}(g_o)$, such that
\begin{equation} \label{eq:estscal}
\big|{\rm Rm}(g(t))\big|_{g(t)} \leq c\cdot{\rm scal}(g(t)) \,\, , \quad |t|\,{\rm scal}(g(t)) \in [c^{-1},c] \quad \text{ for any $t \leq -1$ }\, .
\end{equation}
Indeed, by rescaling and time reparametrization, the solution can be assumed to satisfy the hypotheses of \cite{BLS19}. Then, the first inequality in \eqref{eq:estscal} is stated in \cite[eq.\ (20)]{BLS19}. The lower bound on $|t|\,{\rm scal}(g(t))$ follows by \cite[eq.\ (20)]{BLS19} and \cite[Corollary 2]{BLS19}, while the upper bound on $|t|\,{\rm scal}(g(t))$ follows by \cite[Corollary 2]{BLS19} and the Cauchy-Schwarz inequality. Up to changing the constant, it is possible to choose the same constant $c$ for all the bounds in \eqref{eq:estscal}. \smallskip

In particular, any ancient homogeneous solution is of Type I, and the behavior of its scalar curvature determines the geometry of the solution as $t \to -\infty$.

There are two possibilities for the limiting behavior of the scalar curvature under the volume-normalized flow, and, by \cite{BLS19}, they characterize the collapsing behavior in the homogeneous setting. Here, we recall that the solution $g(t)$ is said to be {\it non-collapsed} if the quantity
$$
{\rm inj}(g(t))\big(\big|{\rm Rm}(g(t))\big|_{g(t)}\big)^{\frac12}
$$
is bounded away from zero as $t \to -\infty$, where ${\rm inj}(g)$ denotes the injectivity radius of the metric $g$. Otherwise, $g(t)$ is said to be {\it collapsed}.

The first possibility is that ${\rm scal}(\tilde{g}(t)) \to \varepsilon > 0$ as $t \to -\infty$. In this case, the solution is non-collapsed and, according to \cite[Theorem 5.2]{BLS19}, converges to an Einstein metric on $M$ as $t \to -\infty$. Since the traceless Ricci tensor is the negative $L^2$-gradient of the scalar curvature functional, non-collapsed ancient solutions are known to exist whenever $M$ admits a {\it $\mathsf{G}$-unstable}, $\mathsf{G}$-invariant Einstein metric (see, e.g., \cite{BLS19,AC21}). The second possibility is that ${\rm scal}(\tilde{g}(t)) \to 0$ as $t \to -\infty$. In this case, the solution is collapsed and, based on \eqref{eq:estscal} and the fact that $M$ is not diffeomorphic to a torus, it does not admit any convergent subsequence. In this paper, we are interested in the latter case.

\subsection{Compact homogeneous spaces} \label{sect:prelhomsp} \hfill \par

Let $M^m = \mathsf{G}/\mathsf{H}$ be an almost-effective homogeneous space of dimension $m$, with $\mathsf{G}$ and $\mathsf{H}$ compact, connected Lie groups. We assume that $M$ is not diffeomorphic to a torus. Fix an ${\rm Ad}(\mathsf{G})$-invariant Euclidean inner product $Q$ on the Lie algebra $\mathfrak{g} \coloneqq {\rm Lie}(\mathsf{G})$ and denote by $\mathfrak{m}$ the $Q$-orthogonal complement of $\mathfrak{h} \coloneqq {\rm Lie}(\mathsf{H})$ in $\mathfrak{g}$. We identify $\mathfrak{m} \simeq T_{e\mathsf{H}}M$ by means of the induced infinitesimal action
$$
X \mapsto X^*_{e\mathsf{H}} \coloneqq \tfrac{\rm d}{{\rm d}s}\exp(sX)\mathsf{H} \big|_{s = 0} \,\, , \quad X \in \mathfrak{m}
$$
and, consequently, any $\mathsf{G}$-invariant tensor field on $M$ with the corresponding ${\rm Ad}(\mathsf{H})$-invariant tensor on the vector space $\mathfrak{m}$ by the evaluation map at $e\mathsf{H} \in M$. This induces an identification $\EuScript{M}^{\mathsf{G}} \simeq S^2_+(\mathfrak{m}^*)^{{\rm Ad}(\mathsf{H})}$ of the space of $\mathsf{G}$-invariant Riemannian metrics on $M$ with the space of symmetric, positive-definite, ${\rm Ad}(\mathsf{H})$-invariant bilinear forms on $\mathfrak{m}$. We denote by $\EuScript{M}^{\mathsf{G}}_1$ the subset of unit volume metrics. \smallskip

We denote by $\mathfrak{m}_0$ the trivial ${\rm Ad}(\mathsf{H})$-submodule of $\mathfrak{m}$, i.e.,
\begin{equation} \label{eq:trivialmodule}
\mathfrak{m}_0 \coloneqq \{X \in \mathfrak{m} : [\mathfrak{h},X] = 0\} \,\, .
\end{equation}
By \cite[Lemma 4.27]{Boe04}, $\mathfrak{m}_0$ is the Lie algebra of a compact, connected complement of $\mathsf{H}$ inside the normalizer $N_{\mathsf{G}}(\mathsf{H})$. In the following, we identify this compact, connected complement with the identity component $(N_{\mathsf{G}}(\mathsf{H})/\mathsf{H})_0$. By \cite[Corollary I.4.3]{Br72}, $N_{\mathsf{G}}(\mathsf{H})/\mathsf{H}$ acts by right multiplication on $M = \mathsf{G}/\mathsf{H}$ and it is isomorphic to the {\it gauge group} of $\mathsf{G}$-equivariant diffeomorphisms of $M$. Therefore, any abelian subalgebra $\mathfrak{t} \subset \mathfrak{m}_0$ gives rise to a homogeneous torus bundle
\begin{equation} \label{eq:generaltorusbundle}
\mathsf{T} \to \mathsf{G}/\mathsf{H} \to \mathsf{G}/\mathsf{H}\mathsf{T} \,\, ,
\end{equation}
where $\mathsf{T} \subset N_{\mathsf{G}}(\mathsf{H})/\mathsf{H}$ is the connected Lie subgroup of $\mathsf{G}$ satisfying ${\rm Lie}(\mathsf{T}) = \mathfrak{t}$. Note that $\mathsf{T}$ can be non-closed in $\mathsf{G}$. In that case, one can still define $\mathsf{G}/\mathsf{H}\mathsf{T}$ as a {\it local quotient} (see \cite[Proposition 6.1]{Ped20}) and \eqref{eq:generaltorusbundle} becomes a locally homogeneous torus bundle. Since we are interested in collapsed ancient solutions to the homogeneous Ricci flow, we are in the case where $\dim(\mathfrak{m}_0) \geq 1$ (see Section \ref{sect:limitbehavior}). We also denote by $\mathfrak{m}_0^{\perp}$ the $Q$-orthogonal complement of the trivial submodule $\mathfrak{m}_0$ inside $\mathfrak{m}$.
\smallskip

We recall that the isotropy representation of $\mathsf{G}/\mathsf{H}$ is equivalent to the adjoint representation of $\mathsf{H}$ on $\mathfrak{m}$. Since $\mathfrak{m}$ is finite dimensional and $\mathsf{H}$ is compact, this representation can be split into a sum of $Q$-orthogonal, irreducible, ${\rm Ad}(\mathsf{H})$-submodules
\begin{equation} \label{eq:AdH-dec}
\mathfrak{m} = \mathfrak{m}_1 + \mathfrak{m}_2 + {\dots} + \mathfrak{m}_{\ell} \,\, .
\end{equation}
Two ${\rm Ad}(\mathsf{H})$-submodules $\mathfrak{m}_i, \mathfrak{m}_j \subset \mathfrak{m}$ are said to be {\it equivalent} if there exists
\begin{equation} \label{eq:intert}
\text{ $L: \mathfrak{m}_i \to \mathfrak{m}_j$ linear map such that $L \circ {\rm Ad}(h) = {\rm Ad}(h) \circ L$ for any $h \in \mathsf{H}$.}
\end{equation}
In that case, we write $\mathfrak{m}_i \simeq \mathfrak{m}_j$, and any map as in \eqref{eq:intert} is called an {\it ${\rm Ad}(\mathsf{H})$-intertwining map}. If $\mathfrak{m}$ is {\it multiplicity-free} as an ${\rm Ad}(\mathsf{H})$-representation, i.e., its irreducible submodules are pairwise inequivalent, then the decomposition \eqref{eq:AdH-dec} is unique up to order.

In Section \ref{sect:limitbehavior} and Appendix \ref{sect:proofsAGAG} we will need to allow for the possibility of the decomposition $\varphi$ changing, while elsewhere in the paper the choice of a specific decomposition plays a central role. Hence, for the sake of notation, we denote by $\EuScript{F}^{\mathsf{G}}$ the set of ordered, $Q$-orthogonal, ${\rm Ad}(\mathsf{H})$-invariant, irreducible decompositions of $\mathfrak{m}$. By \cite[Lemma 4.19]{Boe04}, $\EuScript{F}^{\mathsf{G}}$ inherits the structure of a compact homogeneous space. Moreover, the number $\ell$ of irreducible submodules and the dimensions $d_i \coloneqq {\rm dim}(\mathfrak{m}_i)$ do not depend on the specific choice of $\varphi \in \EuScript{F}^{\mathsf{G}}$.

\begin{remark} \label{rem:gauge_action}
The conjugation action of the gauge group $N_{\mathsf{G}}(\mathsf{H})/\mathsf{H}$ yields an infinitesimal action on $\mathfrak{m}$ that preserves the splitting $\mathfrak{m} = \mathfrak{m}_0 + \mathfrak{m}_0^{\perp}$. Its restriction to $\mathfrak{m}_0$ coincides with its adjoint action as a Lie group. Moreover, by \eqref{eq:trivialmodule} and \eqref{eq:intert}, $N_{\mathsf{G}}(\mathsf{H})/\mathsf{H}$ acts on $\mathfrak{m}$ by ${\rm Ad}(\mathsf{H})$-intertwining maps.
\end{remark}

Fix a decomposition $\varphi = (\mathfrak{m}_1, {\dots}, \mathfrak{m}_{\ell}) \in \EuScript{F}^{\mathsf{G}}$. An equivalence class of modules in $\varphi$ under the relation induced by ${\rm Ad}(\mathsf{H})$-intertwining maps is called an {\it isotypical class}. Notice that the number of isotypical classes does not depend on the specific choice of $\varphi$.
Up to reordering, we can assume that
\begin{equation} \label{eq:orderdec}
\mathfrak{m}_0 = \sum_{p=1}^{\ell_0} \mathfrak{m}_p \,\, , \quad \text{ where $d_p = 1$ for any $1 \leq p \leq \ell_0$ .}
\end{equation}
A basis $\{e_{\alpha}\}$ for $\mathfrak{m}$ is said to be {\it $\varphi$-adapted} if it respects the decomposition $\varphi$. Let
\begin{equation} \label{eq:[ijk]}
[ijk]_{\varphi} \coloneqq \sum_{e_{\alpha} \in \mathfrak{m}_i} \sum_{e_{\beta} \in \mathfrak{m}_j} \sum_{e_{\gamma} \in \mathfrak{m}_k} Q([e_{\alpha},e_{\beta}],e_{\gamma})^2 \quad \text{for any $1 \leq i, j, k \leq \ell$\,.}
\end{equation}
Notice that \eqref{eq:[ijk]} does not depend on the choice of the $\varphi$-adapted, $Q$-orthonormal basis. Since ${\rm ad}(X)$ is $Q$-skew-symmetric for any $X \in \mathfrak{g}$, the coefficients $[ijk]_{\varphi}$ are symmetric in all three entries. Furthermore $[ijk]_{\varphi} \geq 0$, with $[ijk]_{\varphi} = 0$ if and only if $[\mathfrak{m}_i, \mathfrak{m}_j]$ and $\mathfrak{m}_k$ are $Q$-orthogonal. Moreover, the correspondence $\varphi \mapsto [ijk]_{\varphi}$ is a continuous function on $\EuScript{F}^{\mathsf{G}}$ (see \cite[Sec 4.3]{Boe04}). We also define the coefficients $b_1,{\dots},b_{\ell} \in \mathbb{R}$ by
\begin{equation} \label{def:b_i}
(-\mathcal{B}_{\mathfrak{g}})|_{\mathfrak{m}_i \otimes \mathfrak{m}_i} = b_i Q|_{\mathfrak{m}_i \otimes \mathfrak{m}_i} \,\, ,
\end{equation}
where $\mathcal{B}_{\mathfrak{g}}$ is the Cartan-Killing form of $\mathfrak{g}$. Since $\mathsf{G}$ is compact, it follows that $b_i \geq 0$ and $b_i=0$ if and only if $\mathfrak{m}_i \subset \mathfrak{z}(\mathfrak{g})$, where $\mathfrak{z}(\mathfrak{g})$ denotes the center of $\mathfrak{g}$. \smallskip

For any $g \in \EuScript{M}^{\mathsf{G}}$ there is a decomposition $\varphi = (\mathfrak{m}_1, {\dots}, \mathfrak{m}_{\ell})\in \EuScript{F}^{\mathsf{G}}$ with respect to which $g$ is diagonal, i.e.,
\begin{equation} \label{eq:g-diag}
g = x_1\,Q|_{\mathfrak{m}_1 \otimes \mathfrak{m}_1} + {\dots} + x_{\ell}\,Q|_{\mathfrak{m}_{\ell} \otimes \mathfrak{m}_{\ell}} \,\, .
\end{equation}
Notice that, in general, this condition does not uniquely determine the decomposition $\varphi$. If the adjoint representation is multiplicity free, then any $\mathsf{G}$-invariant metric on $M$ is diagonal with respect to the (essentially) unique decomposition of $\mathfrak{m}$. If $g$ is diagonal with respect to $\varphi$ ({\it $\varphi$-diagonal} for short) and has eigenvalues $x_1, {\dots}, x_{\ell}$ as in \eqref{eq:g-diag}, then its Ricci tensor satisfies
$$
{\rm Ric}(g)|_{\mathfrak{m}_i \otimes \mathfrak{m}_i} = x_i\, {\rm ric}_i(g)\,Q|_{\mathfrak{m}_i \otimes \mathfrak{m}_i} \quad \text{ for any $1 \leq i \leq \ell$ } \,\, ,
$$
with
\begin{equation} \label{eq:ric}
{\rm ric}_i(g) \coloneqq \frac{b_i}{2x_i} -\frac1{2d_i}\sum_{1 \leq j,k \leq \ell}[ijk]_{\varphi}\frac{x_k}{x_ix_j}+\frac1{4d_i}\sum_{1 \leq j,k \leq \ell}[ijk]_{\varphi}\frac{x_i}{x_jx_k} \,\, .
\end{equation}
Notice that, although the metric $g$ is $\varphi$-diagonal, in general ${\rm Ric}(g)$ has off-diagonal terms. Finally, the scalar curvature of $g$ takes the form
\begin{equation} \label{eq:scal}
{\rm scal}(g) = \frac12\sum_{1 \leq i \leq \ell}\frac{d_ib_i}{x_i} -\frac14\sum_{1 \leq i,j,k \leq \ell}[ijk]_{\varphi}\frac{x_i}{x_j x_k} \,\, .
\end{equation}
In the rest of the paper, when a distinguished decomposition $\varphi$ is fixed and there is no ambiguity, we will write $[ijk]$ instead of $[ijk]_{\varphi}$.

\section{Collapsing tori of homogeneous ancient solutions} \label{sect:limitbehavior}  \setcounter{equation} 0

As mentioned in the Introduction, a compact homogeneous space admits a collapsed ancient solution only if it is the total space of a principal torus bundle (\cite{BWZ04} and \cite[Remark 5.13]{BLS19}). In this section, we describe the collapsing behavior of these solutions along the toral fibers in more detail. \smallskip

Let $g(t)$ be a collapsed ancient solution to the homogeneous Ricci flow on $M = \mathsf{G}/\mathsf{H}$. By the discussion in Section \ref{subsect:homRF}, the scalar curvature of the corresponding volume-normalized solution $\tilde{g}(t)$ satisfies ${\rm scal}(\tilde{g}(t)) \to 0$ as $t \to -\infty$. Then, since the first inequality of \eqref{eq:estscal} is scale-invariant, we obtain
$$
\lim_{t \to -\infty} \big|{\rm Rm}(\tilde{g}(t))\big|_{\tilde{g}(t)} = 0 \,\, ,
$$
which implies that $\tilde{g}(t)$ cannot admit any convergent subsequence of metrics in $\EuScript{M}^{\mathsf{G}}_1$. More precisely, fix a sequence of times $\xi = \{\tau^{(n)}\} \subset (0,+\infty)$ such that $\tau^{(n)} \to +\infty$ as $n \to +\infty$. Then, $\big\{\tilde{g}(-\tau^{(n)})\big\} \subset \EuScript{M}^{\mathsf{G}}_1$ is a divergent sequence of unit volume $\mathsf{G}$-invariant metrics with bounded curvature, and so \cite[Theorem 4.3]{Ped19} applies. In fact, a similar characterization also holds for the curvature-normalized metrics, from which Theorem \ref{thm:main-existence.torus} follows immediately.

\begin{theorem} \label{thm:limittorus}
Let $g(t)$ be a collapsed, ancient solution to the homogeneous Ricci flow on $M$ and let $\xi = \{\tau^{(n)}\}$ be a sequence such that $\tau^{(n)} \to +\infty$. Then, up to passing to a subsequence, the following properties hold. \begin{itemize}
\item[i)] There exists a sequence of decompositions $\varphi^{(n)} = (\mathfrak{m}_1^{(n)},{\dots},\mathfrak{m}_{\ell}^{(n)}) \in \EuScript{F}^{\mathsf{G}}$, with
$$
g(-\tau^{(n)}) = x_1^{(n)} Q\big|_{\mathfrak{m}_1^{(n)} \otimes \mathfrak{m}_1^{(n)}} + {\dots} + x_{\ell}^{(n)} Q\big|_{\mathfrak{m}_{\ell}^{(n)} \otimes \mathfrak{m}_{\ell}^{(n)}} \,\, ,
$$
such that $\{\varphi^{(n)}\}$ converges to a limit decomposition $\varphi^{(\infty)} = (\mathfrak{m}_1^{(\infty)},{\dots},\mathfrak{m}_{\ell}^{(\infty)})$ as $n \to +\infty$.
\item[ii)] Up to reordering, there exist $1 \leq s_{\xi} \leq \ell$ and $\delta>0$ such that
\begin{equation} \label{eq:asymshrinking} \begin{aligned}
\lim_{n \to +\infty} \frac{x_p^{(n)}}{\tau^{(n)}} = 0 \quad &\text{ for any $1 \leq p \leq s_{\xi}$} \,\, , \\
x_i^{(n)} \geq \delta\, \tau^{(n)} \quad &\text{ for any $s_{\xi} < i \leq \ell$} \, , \,\, \text{ for any $n \in \mathbb{N}$} \,\, .
\end{aligned} \end{equation}
\item[iii)] The sum $\mathfrak{t}_{\xi} \coloneqq \mathfrak{m}_1^{(\infty)}+{\dots}+\mathfrak{m}_{s_{\xi}}^{(\infty)}$ is an abelian subalgebra of $\mathfrak{m}_0$.
\item[iv)] For any $1 \leq p \leq s_{\xi}$ and for any $1 \leq i \leq j \leq \ell$,
\begin{align}
[pij]_{\varphi^{(\infty)}} = 0 \quad &\Longrightarrow \quad \lim_{n \to +\infty}[pij]_{\varphi^{(n)}}\frac{x_j^{(n)}}{x_i^{(n)}} = 0 \,\, , \label{eq:AGAG1} \\
[pij]_{\varphi^{(\infty)}} > 0 \quad &\Longrightarrow \quad \lim_{n \to +\infty}\frac{x_j^{(n)}}{x_i^{(n)}} = 1 \,\, . \label{eq:AGAG2}
\end{align}
\item[v)] Up to reordering, there exists $C > 0$ such that 
$$
x_{s_{\xi}+1}^{(n)} \leq C\, \tau^{(n)} \,\, \text{ for any $n \in \mathbb{N}$} \, .
$$
\end{itemize}
\end{theorem}

Since the proof of Theorem \ref{thm:limittorus} is technical and is an adaptation of the proofs of \cite[Theorem 4.1, Theorem 4.3]{Ped19} to this context, we provide it in Appendix \ref{sect:proofsAGAG}. \smallskip

Notice that the subalgebra $\mathfrak{t}_{\xi}$ obtained in Theorem \ref{thm:limittorus} could depend on the sequence of times $\xi$. Within this context, we give the following.

\begin{definition}
Let $g(t)$ be a collapsed, ancient solution to the $\mathsf{G}$-homogeneous Ricci flow. We say that a connected, abelian subgroup $\mathsf{T} \subset N_{\mathsf{G}}(\mathsf{H})/\mathsf{H}$ is a {\it collapsing torus of $g(t)$} if there exists a sequence $\xi = \{\tau^{(n)}\}$, with $\tau^{(n)} \to +\infty$, such that ${\rm Lie}(\mathsf{T}) = \mathfrak{t}_{\xi}$ as in Theorem \ref{thm:limittorus}.
\end{definition}

Notice that any collapsing torus $\mathsf{T}_{\xi}$ of $g(t)$ acts on the right on $M$ and gives rise to a (possibly locally) homogeneous torus bundle
\begin{equation} \label{eq:torusbundlexi}
\mathsf{T}_{\xi} \to \mathsf{G}/\mathsf{H} \to \mathsf{G}/\mathsf{H}\mathsf{T}_{\xi} \,\, .
\end{equation}
Moreover, condition \eqref{eq:AGAG2} can be interpreted as the (possibly locally-defined) right action of $\mathsf{T}_{\xi}$ getting closer to being isometric along the sequence of times $\xi$ as $n \to +\infty$. Indeed, condition $[pij]_{\varphi^{(\infty)}} > 0$ implies that the right action of $\mathsf{T}_{\xi}$ intertwines the submodules $\mathfrak{m}_i^{(\infty)}$ and $\mathfrak{m}_j^{(\infty)}$, which are limits of $\mathfrak{m}_i^{(n)}$ and $\mathfrak{m}_j^{(n)}$, while the right-hand side of \eqref{eq:AGAG2} implies that the metric eigenvalues $x_i^{(n)}$ and $x_j^{(n)}$ are getting closer to each other as $n \to +\infty$. \smallskip

To draw the geometric conclusions that we are interested in, we need collapsing tori to satisfy additional properties.

\begin{definition} \label{def:sstorus}
Let $g(t)$ be a collapsed, ancient solution to the $\mathsf{G}$-homogeneous Ricci flow. We say that a collapsing torus $\mathsf{T}$ of $g(t)$, with $\mathfrak{t} = {\rm Lie}(\mathsf{T})$, is {\it rigid} if the following conditions are satisfied: 
\begin{itemize}
\item[$i)$] for any $V \in \mathfrak{t}$, we have $\lim_{t \to -\infty}|t|^{-\frac12}|V|_{g(t)} = 0$;
\item[$ii)$] $\mathfrak{t}$ and its $Q$-orthogonal complement $\mathfrak{t}^{\perp}$ inside $\mathfrak{m}$ are $g(t)$-orthogonal for any $t < 0$; 
\item[$iii)$] there exists $\delta>0$ such that $|X|_{g(t)} \geq \delta |X|_{g(-1)}$ for any $X \in \mathfrak{t}^{\perp}$ and for any $t \leq -1$.
\end{itemize}
\end{definition}

Notice that if a collapsing torus $\mathsf{T}$ of $g(t)$ is rigid, then $\mathsf{T}$ is also unique. Interestingly, the collapsing torus is rigid in all known constructions in the literature. We remark that it is unknown whether this is generally true. However, we can prove it under the assumptions of Proposition \ref{prop:main-uniq.torus} (see Remark \ref{rem:unique-rigid}).

\section{Limit solitons for homogeneous ancient solutions}
\label{sect:proof-thmsolitons} \setcounter{equation} 0

For non-collapsed ancient Ricci flow solutions, \cite{Na10, CaoZh11} give us a sequence of times, going to $-\infty$, along which the metric converges to a gradient shrinking Ricci soliton. Moreover, in the homogeneous case, the limit is an Einstein metric on the same manifold, obtained as a full backward limit (see \cite{BLS19}).

In the collapsed setting, one cannot hope for a limit soliton metric on the same space, but the behavior of known examples, along with the results of Section \ref{sect:limitbehavior}, suggest seeking a limiting Einstein metric on a quotient by a collapsing torus. Here, we verify this claim for flows that satisfy additional structural assumptions. The main result of this section is the following.

\begin{theorem} \label{thm:limitEin-rigid}
Let $g(t)$ a collapsed ancient solution to the homogeneous Ricci flow on $M$. If its collapsing torus $\mathsf{T}$ is rigid and $g(t)$ is right $\mathsf{T}$-invariant, then $\big(M,\tfrac1{|t|}g(t)\big)$ converges in Gromov-Hausdorff topology to an Einstein metric on $M/\mathsf{T}$ as $t \to -\infty$.
\end{theorem}

As a corollary, we obtain Theorem \ref{thm:main-coll.solitons}.

\begin{proof}[Proof of Theorem \ref{thm:main-coll.solitons}]
We need to show that, under the assumption of $\mathsf{T}$ being unique and maximal, if $g(t)$ is right $\mathsf{T}$-invariant, then $\mathsf{T}$ is also rigid (see Definition \ref{def:sstorus}). Let $\mathfrak{t} = {\rm Lie}(\mathsf{T})$ and consider the $Q$-orthogonal splitting
$$
\mathfrak{m} = \mathfrak{t} + \mathfrak{t}^{\perp} \,\, .
$$
By maximality of $\mathfrak{t}$, there exists no trivial ${\rm Ad}(\mathsf{H}\mathsf{T})$-submodule in $\mathfrak{t}^{\perp}$. Therefore, since $g(t)$ is right $\mathsf{T}$-invariant for any $t < 0$, it follows from Schur's Lemma that $\mathfrak{t}$ and its $Q$-orthogonal complement $\mathfrak{t}^{\perp}$ are $g(t)$-orthogonal for any $t < 0$. By maximality of $\mathfrak{t}$ and Theorem \ref{thm:limittorus}, there exists $\delta>0$ such that $|X|_{g(t)} \geq \delta |X|_{g(-1)}$ for any $X \in \mathfrak{t}^{\perp}$ and for any $t \leq -1$. Finally, since the splitting $\mathfrak{m} = \mathfrak{t} + \mathfrak{t}^{\perp}$ is $g(t)$-orthogonal for any $t < 0$ and the collapsing torus $\mathsf{T}$ is unique, it follows that $\lim_{t \to -\infty}|t|^{-\frac12}|V|_{g(t)} = 0$ for any $V \in \mathfrak{t}$, because every sequence of times $t^{(n)} \to -\infty$ admits a subsequence along which the limit is zero (see the first equation in \eqref{eq:asymshrinking}).
\end{proof}

Theorem \ref{thm:limitEin-rigid} illustrates the important role of right toral symmetries and provides further motivation for Theorem \ref{thm:main-extra.symm}. In the presence of right $\mathsf{T}$-invariance, the metric $g$ on $M$ has extra structure, namely it is a Riemannian submersion onto a homogeneous metric on $\mathsf{G}/\mathsf{H}\mathsf{T}$. Following \cite{ONeill66}, this simplifies the expression for ${\rm Ric}(g)$. We begin by introducing the notation required for this.

\subsection{Riemannian submersions on homogeneous torus bundles}
\label{sect:submetrics} \hfill \par

Let $M^m = \mathsf{G}/\mathsf{H}$ be as in Section \ref{sect:prel} and fix a torus $\mathsf{T} \subset N_{\mathsf{G}}(\mathsf{H})/\mathsf{H}$. This gives rise to a (locally) homogeneous torus bundle
\begin{equation} \label{eq:torusbundle}
\mathsf{T} \to M = \mathsf{G}/\mathsf{H} \to B \coloneqq \mathsf{G}/\mathsf{H}\mathsf{T}
\end{equation}
and the corresponding $Q$-orthogonal splitting at the Lie algebra level:
\begin{equation} \label{eq:decomp}
\mathfrak{g} = \mathfrak{h} + \overbrace{\mathfrak{t} +\mathfrak{b}}^{\mathfrak{m}} \,\, , \quad \text{ with } \mathfrak{b} \simeq T_{e\mathsf{H}\mathsf{T}}B \,\, .
\end{equation}
We say that a left $\mathsf{G}$-invariant metric $g$ on $M$ is of {\it submersion type} with respect to the bundle \eqref{eq:torusbundle} if it preserves the decomposition \eqref{eq:decomp} and its restriction to the subspace $\mathfrak{b}$ is ${\rm Ad}(\mathsf{H}\mathsf{T})$-invariant.  We denote the set of all such metrics by $\EuScript{M}^{\mathsf{G},\mathsf{T}}$.  Any $g \in \EuScript{M}^{\mathsf{G},\mathsf{T}}$ can be written as
\begin{align*}
 g = g_{\mathfrak{t}} + g_{\mathfrak{b}},
\end{align*}
where $g_{\mathfrak{t}}$ is an inner product on $\mathfrak{t}$ and $g_{\mathfrak{b}}$ is an ${\rm Ad}(\mathsf{H}\mathsf{T})$-invariant inner product on $\mathfrak{b}$. Notice that $(M, g) \to (B, g_{\mathfrak{b}})$ is a Riemannian submersion with totally geodesic fibers (see, e.g., \cite[Lemma 3.5]{Ped19}). The failure of the horizontal distribution $\mathfrak{b}$ to be integrable is measured by the O'Neill tensor $A$ defined in \cite[Section 2]{ONeill66}, which is characterized by the following (see \cite[Chapter 9.C]{Bes08}):
\begin{equation} \label{eq:defA}
A_XY = \tfrac12 [X,Y]_{\mathfrak{t}} \,\, , \quad A_UX = A_UV = 0 \,\, , \quad g(A_XU,V) = 0 \,\, , \quad g(A_XU,Y) = -g(A_XY,U) \,\, ,
\end{equation}
where $X,Y \in \mathfrak{b}$, $U,V \in \mathfrak{t}$ and $[\cdot,\cdot]_{\mathfrak{t}}$ denotes the orthogonal projection of the bracket onto $\mathfrak{t}$. Note that if $A\equiv 0$, then $g$ is (at least locally) a Riemannian product $\mathsf{T} \times B$ (see \cite[Remark 9.26]{Bes08}). One can express the quantity $|A|_g^2$ in terms of the eigenvalues $x_1,{\dots},x_{\ell}$ of $g \in \EuScript{M}^{\mathsf{G},\mathsf{T}}$ as
\begin{equation} \label{eq:formula|A|}
|A|_g^2 = \sum_{k=1}^s\sum_{i,j=s+1}^{\ell} [ijk]_{\varphi} \frac{x_k}{x_i x_j} \,\, ,
\end{equation}
where $\varphi = (\mathfrak{m}_1,{\dots},\mathfrak{m}_{\ell})$ is a decomposition with respect to which $g$ is diagonal and the integer $1 \leq s \leq \ell$ is such that $\mathfrak{t} = \mathfrak{m}_1 + {\dots} +\mathfrak{m}_{s}$, $\mathfrak{b} = \mathfrak{m}_{s+1} + {\dots} +\mathfrak{m}_{\ell}$.

By \cite[Proposition 9.36]{Bes08}, for a metric $g \in \EuScript{M}^{\mathsf{G},\mathsf{T}}$, the vertical and the horizontal components of its Ricci tensor are given by
\begin{equation} \label{eq:Besseformulas}
{\rm Ric}(g)|_{\mathfrak{t} \otimes \mathfrak{t}} = \hat{A}_g \,\, , \qquad {\rm Ric}(g)|_{\mathfrak{b} \otimes \mathfrak{b}} = {\rm Ric}_B(g_{\mathfrak{b}}) - 2\check{A}_g \,\, ,
\end{equation}
where the tensors $\hat{A}_g \in S^2(\mathfrak{t}^*)$ and $\check{A}_g \in S^2(\mathfrak{b}^*)^{{\rm Ad}(\mathsf{H}\mathsf{T})}$ are given by
$$
\hat{A}_g(U, V) \coloneqq \sum_i g(A_{X_i}U, A_{X_i}V) \,\, , \qquad \check{A}_g (X, Y) \coloneqq \sum_i g(A_X X_i, A_Y X_i) = \sum_j g(A_X U_j, A_Y U_j)
$$
(compare with \cite[Formula (9.33a), Formula (9.33c)]{Bes08}). Here, $\{X_i\}$ (resp.\ $\{U_i\}$) denotes a $g$-orthonormal basis of $\mathfrak{b}$ (resp.\ $\mathfrak{t}$). Observe that, whereas the tensor $A$ depends only on the principal connection on the fiber bundle \eqref{eq:torusbundle}, the tensors $\hat{A}_g$ and $\check{A}_g$ depend on the metric $g$. For later use, we provide the following estimates.

\begin{proposition}
There exists $C=C(m)>1$, depending only on the dimension, such that
\begin{equation} \label{eq:estA}
C^{-1}|A|_g^2 \leq |\hat{A}_g|_g \leq C |A|_g^2 \,\, , \quad C^{-1}|A|_g^2 \leq |\check{A}_g|_g \leq C |A|_g^2
\end{equation}
for any $g \in \EuScript{M}^{\mathsf{G},\mathsf{T}}$.
\end{proposition}

\begin{proof}
We will give the proof of the inequalities for $|\hat{A}_g|_g$, since the proof for $|\check{A}_g|_g$ is very similar. For convenience, we set $A_{pq}^r \coloneqq g(A_{X_p}X_q, U_r)$. Since
\begin{align*}
|A|_g^2 &= \sum_{p,q,r}g(A_{X_p}X_q, U_r)^2 = \sum_{p,q,r} (A_{pq}^r)^2 \,\, , \\
|\hat{A}_g|_g^2 &= \sum_{i,j} \hat{A}_g(U_i, U_j)^2 = \sum_{i,j} \bigg( \sum_{k,l} A_{kl}^i A_{kl}^j \bigg)^2 \,\, ,
\end{align*}
the Cauchy--Schwarz inequality implies that
$$
|\hat{A}_g|_g^2 \leq \sum_{i,j} \bigg(\bigg(\sum_{k,l}(A_{kl}^i)^2 \bigg) \bigg(\sum_{k,l}(A_{kl}^j)^2 \bigg)\bigg) \leq m^2 |A|_g^4 \,\, .
$$
On the other hand, by the equivalence of norms on a finite-dimensional space, we obtain
$$
|\hat{A}_g|_g^2 \geq C_1(m) \bigg(\sum_{i,j} \bigg| \sum_{k,l}A_{kl}^iA_{kl}^j \bigg| \bigg)^2 \geq C_1(m) \bigg( \sum_{i,k,l} \big|A_{kl}^i\big| \cdot \big|A_{kl}^i\big| \bigg)^2 = C_1(m) |A|_g^4
$$
and this concludes the proof.
\end{proof}

Finally, the Ricci tensor generally has off-diagonal entries, namely ${\rm Ric}(g)|_{\mathfrak{t} \otimes \mathfrak{b}} \neq 0$. However, under the assumptions of Theorem \ref{thm:limitEin-rigid}, this cannot happen. 

\subsection{Ricci flow solutions on homogeneous torus bundles}
\label{sect:submetricsRF} \hfill \par

In this subsection, we introduce a functional $F$ with monotonicity properties along collapsed ancient Ricci flow solutions on compact homogeneous spaces. More precisely:
\begin{equation} \label{eq:functional}
F: \EuScript{M}^{\mathsf{G},\mathsf{T}} \to \mathbb{R} \,\, , \quad F(g) \coloneqq \big({\rm scal}(g)+2|A|_{g}^2\big) {\rm vol}_B(g_{\mathfrak{b}})^{\frac2k} \,\, ,
\end{equation}
where ${\rm vol}_B(g_{\mathfrak{b}})$ is the volume of $(B, g_\mathfrak{b})$ and $k = \dim(B)$. Notice that $F$ is scale-invariant. \smallskip

The following theorem establishes a monotonicity property of $F$. To state it more conveniently, we consider a solution $g(\tau)$ to the {\it backward} Ricci flow defined for any $\tau > 0$, i.e.,
\begin{equation} \label{eq:backhomRF}
g'(\tau) = +2 {\rm Ric}(g(\tau)) \,\, , \quad \tau \in (0,+\infty) \,\, .
\end{equation}
Notice that this is equivalent to considering an ancient solution for the forward flow, but we prefer to work with positive times for the sake of readability. 

\begin{theorem} \label{thm:functionalF}
Let $g(\tau)$ be a collapsed homogeneous solution to the backward Ricci flow on $M$ existing for all $\tau > 0$ and assume that its collapsing torus $\mathsf{T}$ is rigid and that $g(\tau)$ is right $\mathsf{T}$-invariant for any $\tau > 0$. Then, the first variation of the functional $F$ defined in \eqref{eq:functional} is given by
\begin{multline} \label{eq:derF}
\frac{\rm d}{{\rm d}\tau}F(g(\tau)) = -2{\rm vol}_B(g_{\mathfrak{b}})^{\frac2k} \Big( |{\rm Ric}_B(g_{\mathfrak{b}})|_{g_{\mathfrak{b}}}^2-\tfrac1k{\rm scal}_B(g_{\mathfrak{b}})^2 \\
+\tfrac1k{\rm scal}_B(g_{\mathfrak{b}})|A|_{g}^2 -|\hat{A}_g|_g^2 -4|\check{A}_g|_g^2 +\tfrac2k|A|_g^4\Big) \,\, .
\end{multline}
As a consequence, there exists $T>0$ such that
\begin{equation} \label{eq:Fmonotone}
\frac{\rm d}{{\rm d}\tau}F(g(\tau)) \leq 0 \quad \text{ for any $\tau \geq T$.}
\end{equation}
In particular, $F$ is bounded from above.
\end{theorem}

We need two preparatory lemmas to prove Theorem \ref{thm:functionalF}. For the rest of this section, assume that the collapsing torus $\mathsf{T}$ of $g(\tau)$ is rigid and that $g(\tau)$ is right $\mathsf{T}$-invariant for any $\tau > 0$. Then, the solution $g(\tau)$ evolves through submersion metrics with respect to the (locally) homogeneous torus bundle \eqref{eq:torusbundle}, i.e.,
\begin{equation} \label{eq:subm-sol}
g(\tau) = g_{\mathfrak{t}}(\tau) + g_{\mathfrak{b}}(\tau) \,\, , \quad \tau \in (0,+\infty) \,\, .
\end{equation}
Since $g(\tau)$ is of the form \eqref{eq:subm-sol} for any $\tau > 0$, it follows that its Ricci tensor is block diagonal, namely, ${\rm Ric}(g(\tau))(U, X) = 0$ for any $U \in \mathfrak{t}$, $X \in \mathfrak{b}$.

\begin{lemma}
Along a solution as in \eqref{eq:subm-sol}, the first variation of the quantity $|A|_g^2$ is given by
\begin{equation} \label{eq:derA}
\frac{\rm d}{{\rm d}\tau}|A|_{g(\tau)}^2 = 2|\hat{A}_g|_g^2 +8|\check{A}_g|_g^2 -4g({\rm Ric}_B(g_{\mathfrak{b}}),\check{A}_g) \,\, .
\end{equation}
\end{lemma}

\begin{proof}
Let $\{X_i\}$ be a $g(1)$-orthonormal basis for $\mathfrak{b}$ and denote by $X_i(\tau)$ the vector fields evolving through
$$
\frac{\rm d}{{\rm d}\tau}X_i(\tau) = -{\rm Ric}(g(\tau))^{\sharp}(X_i(\tau)) \,\, , \quad X_i(1) = X_i \,\, .
$$
Then, $\{X_i(\tau)\}$ is $g(\tau)$-orthonormal. Notice that $A$ itself does not depend on $g$, so that
$$\begin{aligned}
\frac{\rm d}{{\rm d}\tau}|A|_{g(\tau)}^2 &= \frac{\rm d}{{\rm d}\tau} \left\{\sum_{i,j} g(A_{X_i}X_j,A_{X_i}X_j)\right\} \\
&= \sum_{i,j} 2{\rm Ric}(g(\tau))(A_{X_i}X_j,A_{X_i}X_j) -2g(A_{({\rm Ric}(g(\tau))^{\sharp}X_i)}X_j,A_{X_i}X_j) \\
&\qquad\qquad -2g(A_{X_i}({\rm Ric}(g(\tau))^{\sharp}X_j),A_{X_i}X_j) \\
&= \sum_{i,j} 2{\rm Ric}(g(\tau))(A_{X_i}X_j,A_{X_i}X_j) -4g(A_{X_i}({\rm Ric}(g(\tau))^{\sharp}X_j),A_{X_i}X_j) \,\, .
\end{aligned}$$
Notice that the fibers of $M \to B$ are abelian and totally geodesic. At any fixed time $\tau > 0$, we can consider a $g(\tau)$-orthonormal basis $\{U_p\}$ for $\mathfrak{t}$. Therefore, by \cite[9.21 and Proposition 9.36]{Bes08}, we have
$$\begin{aligned}
2\sum_{i,j} {\rm Ric}(g(\tau))(A_{X_i}X_j,A_{X_i}X_j) &= 2\sum_{i,j} \hat{A}_{g(\tau)}(A_{X_i}X_j,A_{X_i}X_j) \\
&= 2\sum_{i,j,p,q} \hat{A}_{g(\tau)}(U_p,U_q)g(\tau)(A_{X_i}U_p,X_j)g(\tau)(A_{X_i}U_q,X_j) \\
&= 2\sum_{i,j,p,q} \hat{A}_{g(\tau)}(U_p,U_q)^2 \\
&= 2|\hat{A}_{g(\tau)}|_{g(\tau)}^2
\end{aligned}$$
and
$$\begin{aligned}
-4\sum_{i,j}g(A_{X_i}({\rm Ric}(g(\tau))^{\sharp}X_j),A_{X_i}X_j) &= 4\sum_{i,j}{\rm Ric}(g(\tau))(A_{X_i}A_{X_i}X_j,X_j) \\
&= -4\sum_{i,j,k}{\rm Ric}(g(\tau))(X_j,X_k)g(\tau)(A_{X_i}X_k,A_{X_i}X_j) \\
&= -4\sum_{i,j,k}{\rm Ric}(g(\tau))(X_j,X_k)\check{A}_{g(\tau)}(X_j,X_k) \\
&= -4g(\tau)({\rm Ric}(g(\tau)),\check{A}_{g(\tau)}) \,\, .
\end{aligned}$$
Finally, by \eqref{eq:Besseformulas} we obtain \eqref{eq:derA}.
\end{proof}

The following lemma will be used in the proof of Theorem \ref{thm:functionalF}.

\begin{lemma}
Under the assumptions of Theorem \ref{thm:functionalF}, we have
\begin{equation} \label{eq:Ato0}
\lim_{\tau \to +\infty} \tau|A|_{g(\tau)}^2 = 0 \,\, .
\end{equation}
\end{lemma}

\begin{proof}
Since the collapsing torus $\mathsf{T}$ is rigid, it follows that
$$\begin{gathered}
\frac{x_k(\tau)}{\tau} \to 0 \quad \text{as $\tau \to +\infty$, for every $1 \leq k \leq s$} \,\, , \\
\frac{x_i(\tau)}{\tau} \geq \delta \quad \text{for any $\tau \geq 1$, for every $s < i \leq \ell$} \,\, .
\end{gathered}$$
Therefore, by \eqref{eq:formula|A|} we obtain
$$
\lim_{\tau \to +\infty} |A|_{\frac1{\tau}g(\tau)} = 0 \,\, ,
$$
which completes the proof.
\end{proof}

Finally, we are ready to prove Theorem \ref{thm:functionalF}.

\begin{proof}[Proof of Theorem \ref{thm:functionalF}]
By \cite[Theorem 1.174]{Bes08} and homogeneity, we have
$$
\frac{\rm d}{{\rm d}\tau} {\rm scal}(g(\tau)) = -2|{\rm Ric}(g)|_g^2
$$
and, by \cite[Proposition 1.186, Proposition 9.36]{Bes08},
$$
\frac{\rm d}{{\rm d}\tau} {\rm vol}_B(g_{\mathfrak{b}}(\tau)) = {\rm Tr}({\rm Ric}(g)|_{\mathfrak{b}}){\rm vol}_B(g_{\mathfrak{b}}) = ({\rm scal}_B(g_{\mathfrak{b}})-2|A|_g^2){\rm vol}_B(g_{\mathfrak{b}}) \,\, .
$$
Therefore, by using \eqref{eq:derA}, we have
$$\begin{aligned}
\frac{\rm d}{{\rm d}\tau}F(g(\tau)) &= 2\Big(-|{\rm Ric}(g)|_g^2 +2|\hat{A}_g|_g^2 +8|\check{A}_g|_g^2 -4g({\rm Ric}_B(g_{\mathfrak{b}}),\check{A}_g)\Big){\rm vol}_B(g_{\mathfrak{b}})^{\frac2k} +\\
&\quad\quad +\tfrac2k\big({\rm scal}(g)+2|A|_{g}^2\big)({\rm scal}_B(g_{\mathfrak{b}})-2|A|_g^2) {\rm vol}_B(g_{\mathfrak{b}})^{\frac2k} \\
&= -2{\rm vol}_B(g_{\mathfrak{b}})^{\frac2k} \Big(|{\rm Ric}(g)|_g^2 -|\hat{A}_g|_g^2 -4|\check{A}_g|_g^2 +4g({\rm Ric}_B(g_{\mathfrak{b}}),\check{A}_g) -\tfrac1k{\rm scal}_B(g_{\mathfrak{b}})^2 \\
&\quad\quad +\tfrac1k{\rm scal}_B(g_{\mathfrak{b}})^2|A|_g^2 -|\hat{A}_g|_g^2 -4|\check{A}_g|_g^2 +\tfrac2k|A|_g^4\Big) \,\, .
\end{aligned}$$
Since
$$
|{\rm Ric}(g)|_g^2 = |\hat{A}_g|_g^2 +|{\rm Ric}_B(g_{\mathfrak{b}})|_{g_{\mathfrak{b}}}^2 +4|\check{A}_g|_g^2 -4g({\rm Ric}_B(g_{\mathfrak{b}}),\check{A}_g) \,\, ,
$$
we obtain \eqref{eq:derF}. Moreover, by the Cauchy--Schwarz inequality, we have
$$
|{\rm Ric}_B(g_{\mathfrak{b}}(\tau))|_{g_{\mathfrak{b}}(\tau)}^2-\tfrac1k{\rm scal}_B(g_{\mathfrak{b}}(\tau))^2 \geq 0
$$
with equality if and only if ${\rm Ric}^0_B(g_{\mathfrak{b}}(\tau)) = 0$. By \eqref{eq:estA}, the remaining term can be estimated as
$$
\tfrac1k{\rm scal}_B(g_{\mathfrak{b}}(\tau))|A|_{g(\tau)}^2 -|\hat{A}_{g(\tau)}|_{g(\tau)}^2 -4|\check{A}_{g(\tau)}|_{g(\tau)}^2 +\tfrac2k|A|_{g(\tau)}^4 \geq \tfrac1{k}|A|_{g(\tau)}^2\big({\rm scal}_B(g_{\mathfrak{b}}(\tau)) -C|A|_{g(\tau)}^2\big)
$$
for some $C>0$. Therefore, by \eqref{eq:Ato0}, there exists $T>0$ such that
$$
{\rm scal}_B(g_{\mathfrak{b}}(\tau)) \geq C|A|_{g(\tau)}^2 \quad \text{for any $\tau \geq T$}
$$
and this completes the proof of \eqref{eq:Fmonotone}.
\end{proof}

\subsection{Existence of limit Einstein metrics}
\label{sect:prooflimitEin-rigid} \hfill \par

We use the estimates in the previous subsections to prove Theorem \ref{thm:limitEin-rigid}, namely, the existence of a limit Einstein metric on the base of the principal bundle \eqref{eq:torusbundle}.

\begin{proof}[Proof of Theorem \ref{thm:limitEin-rigid}]
Let $g(\tau)$ be a collapsed solution to the backward Ricci flow existing for all $\tau > 0$ and $\bar{g}(\tau) = \frac{1}{\tau}g(\tau)$ the corresponding curvature-normalized metrics. Assume that the collapsing torus $\mathsf{T}$ of $g(\tau)$ is rigid, and that $g(\tau)$ is right $\mathsf{T}$-invariant. Then, $\bar{g}(\tau)$ splits as in \eqref{eq:subm-sol}, with
$$
\bar{g}_{\mathfrak{t}}(\tau) \to 0 \,\, \text{ as $\tau \to +\infty$} \,\, , \qquad \bar{g}_{\mathfrak{b}}(\tau) \geq \delta\, \bar{g}_{\mathfrak{b}}(1) \,\, \text{ for any $\tau \geq 1$}.
$$
By \eqref{eq:Ato0} and \cite[Corollary 9.37]{Bes08}, ${\rm scal}(\bar{g}_{\mathfrak{b}}(\tau))$ converges to a positive constant as $\tau \to +\infty$. Since $F(\bar{g}(\tau))$ is bounded, it follows that ${\rm vol}_B(\bar{g}_{\mathfrak{b}}(\tau))$ is bounded as well. Therefore, the 1-parameter family $\bar{g}_{\mathfrak{b}}(\tau)$ lives in a compact set of $\mathsf{G}$-invariant metrics on $B$. We now prove the following two claims. \medskip

\noindent{\it Claim 1: There exists $\tau_i \to +\infty$ such that $\bar{g}_{\mathfrak{b}}(\tau_i) \to \bar{g}_{\mathfrak{b}}^{\infty}$ as $i \to +\infty$, with ${\rm Ric}^0_B(\bar{g}_{\mathfrak{b}}^{\infty})= 0$.}

To prove Claim 1, we observe that, by Theorem \ref{thm:functionalF},
\begin{multline*}
\tau\frac{\rm d}{{\rm d}\tau}F(g(\tau)) = -2{\rm vol}_B(\bar{g}_{\mathfrak{b}})^{\frac2k} \tau{\rm scal}(g) \Big( |{\rm Ric}_B(\bar{g}_{\mathfrak{b}})|_{\bar{g}_{\mathfrak{b}}}^2 -\tfrac1k{\rm scal}_B(\bar{g}_{\mathfrak{b}})^2 \\
+\tfrac1k{\rm scal}_B(\bar{g}_{\mathfrak{b}})|A|_{\bar{g}_{\mathfrak{b}}}^2 -|\hat{A}_{\bar{g}}|_{\bar{g}}^2 -4|\check{A}_{\bar{g}}|_{\bar{g}}^2 +\tfrac2k|A|_{\bar{g}}^4\Big) \,\, .
\end{multline*}
We claim that there exists $\tau_i \to +\infty$ such that $\tau_i\frac{\rm d}{{\rm d}\tau}F(g(\tau_i)) \to 0$. Indeed, suppose that this is not the case. Then, by \eqref{eq:Fmonotone}, there exist $T>0$ and $C>0$ such that
\begin{equation} \label{eq:absurdF'}
\tau \frac{\rm d}{{\rm d}\tau}F(g(\tau)) \leq -C \quad \text{ for any $\tau \geq T$} \,\, .
\end{equation}
Then, define the function
$$
\psi: [T,+\infty) \to \mathbb{R} \,\, , \quad \psi(\tau) \coloneqq F(g(T)) -C\log\left(\frac{\tau}{T}\right)
$$
and notice that $\psi(T) = F(g(T))$, and $\lim_{\tau \to +\infty} \psi(\tau) = - \infty$. Moreover, by \eqref{eq:absurdF'}, it follows that $\frac{\rm d}{{\rm d}\tau}F(g(\tau)) \leq \psi'(\tau)$ for all $\tau \geq T$, from which we obtain $\lim_{\tau \to +\infty} F(g(\tau)) = -\infty$. This is a contradiction since $F(g(\tau))$ is non-negative.

Consider $\{\tau_i\}$ as above. Since the quantities ${\rm vol}_B(\bar{g}_{\mathfrak{b}})^{\frac2k}$, $\tau{\rm scal}(g)$, and ${\rm scal}_B(\bar{g}_{\mathfrak{b}})$ are bounded from above and bounded away from zero along the flow, and $|A|_{\bar{g}}\to 0$ by \eqref{eq:Ato0}, it follows that
$$
|{\rm Ric}_B(\bar{g}_{\mathfrak{b}}(\tau_i))|_{\bar{g}_{\mathfrak{b}}(\tau_i)}^2-\tfrac1k{\rm scal}_B(\bar{g}_{\mathfrak{b}}(\tau_i))^2 \to 0 \,\, .
$$
Up to passing to a subsequence, we can assume that $\bar{g}_{\mathfrak{b}}(\tau_i) \to \bar{g}_{\mathfrak{b}}^{\infty}$ as $i \to +\infty$. Then, Claim 1 follows by applying the Cauchy--Schwarz inequality. \medskip

\noindent{\it Claim 2: We have $\bar{g}(\tau) \to 0 \oplus \bar{g}_{\mathfrak{b}}^{\infty}$ as $\tau \to +\infty$.}

By the previous step, we have convergence to the degenerate metric $0 \oplus \bar{g}_{\mathfrak{b}}^{\infty}$ along a sequence of times. Here, we need to show full convergence along the flow. We use the theory the second- and third-named authors developed in \cite{PedSb22} to prove this fact.

By \cite[Proposition 3.1]{PedSb22}, the Ricci curvature can be analytically extended to the space
$$
S^2(\mathfrak{t}^*) \oplus S^2_+(\mathfrak{b}^*)^{{\rm Ad}(\mathsf{H}\mathsf{T})}
$$
of submersion metrics on $M = \mathsf{G}/\mathsf{H}$ that are possibly not positive-definite on the toral fibers of \eqref{eq:torusbundle}. Moreover, by \cite[Formula (9.37)]{Bes08}, the scalar curvature of a submersion metric $g$ is given by
$$
{\rm scal}(g) = {\rm scal}_B(g_{\mathfrak{b}}) -|A|_g^2
$$
and so, by \eqref{eq:formula|A|}, it follows that the scalar curvature can be analytically extended to the space $S^2(\mathfrak{t}^*) \oplus S^2_+(\mathfrak{b}^*)^{{\rm Ad}(\mathsf{H}\mathsf{T})}$ as well. By \cite[Section 3.2]{PedSb22}, there exists an inner product $\langle\!\langle \cdot, \cdot \rangle\!\rangle^{\bar{g}_{\mathfrak{b}}^{\infty}}$ on $S^2(\mathfrak{m}^*)^{{\rm Ad}(\mathsf{H})}$, depending on the Einstein metric $\bar{g}_{\mathfrak{b}}^{\infty}$ on the base, such that the following claims hold. First, up to time reparametrization, the projection of $g(\tau)$ on the unit sphere
$$
\Sigma \coloneqq \big\{g \in S^2(\mathfrak{t}^*) \oplus S^2_+(\mathfrak{b}^*)^{{\rm Ad}(\mathsf{H}\mathsf{T})} : \langle\!\langle g, g \rangle\!\rangle^{\bar{g}_{\mathfrak{b}}^{\infty}} = 1\big\}
$$
is a solution to the gradient flow of the restricted scalar curvature functional
\begin{equation} \label{eq:restrscal}
g \in \Sigma \mapsto {\rm scal}(g) \,\, .
\end{equation}
Second, the degenerate metric $0 \oplus \bar{g}_{\mathfrak{b}}^{\infty}$ is a critical point for the functional \eqref{eq:restrscal}. Therefore, by \cite[Th{\'e}or{\`e}me (1)]{Loj82}, the convergence is uniform along the flow trajectories, and this concludes the proof of Claim 2. \smallskip

Finally, arguing as in \cite[Proposition 4.2]{PedSb22}, it follows that $(\mathsf{G}/\mathsf{H},\bar{g}(\tau))$ converges to $(\mathsf{G}/\mathsf{H}\mathsf{T},\bar{g}_{\mathfrak{b}}^{\infty})$ in the Gromov-Hausdorff topology as $\tau \to +\infty$.
\end{proof}

\section{NR-decompositions for the isotropy representation}
\label{sect:stronglyRd} \setcounter{equation} 0

This section describes our main algebraic assumption in Theorem \ref{thm:main-extra.symm}, that is, the notion of a {\it NR-decomposition} for the isotropy representation. 

\subsection{Normalizer-adapted decompositions} \label{subsect:norm-diag} \hfill \par

We introduce a class of decompositions, for which the associated structure constants will turn out to satisfy specific symmetry properties (see Lemma \ref{lem:rowreduction}). This will be an important tool to estimate the ricci eigenvalues in Section \ref{sect:mainproof}. \smallskip

Let $\varphi = (\mathfrak{m}_1,{\dots},\mathfrak{m}_{\ell}) \in \EuScript{F}^{\mathsf{G}}$ be a decomposition for $\mathfrak{m}$ and assume that it is ordered such that \eqref{eq:orderdec} holds. For the sake of notation, we will also fix generators $V_p \in \mathfrak{m}_p$ with $Q(V_p,V_p) = 1$ for any $1 \leq p \leq \ell_0$. We introduce the following definition.

\begin{definition} \label{def:norm-diag}
A decomposition $\varphi = (\mathfrak{m}_1,{\dots},\mathfrak{m}_{\ell}) \in \EuScript{F}^{\mathsf{G}}$ is said to be {\it normalizer-adapted} if the following property holds: for any $1 \leq p \leq \ell_0$ and for any $1 \leq i \leq \ell$, if $[\mathfrak{m}_p, \mathfrak{m}_i] \neq \{0\}$, then there exists $1 \leq j \leq \ell$ such that $[\mathfrak{m}_p, \mathfrak{m}_i] \cap \mathfrak{m}_j \neq \{0\}$.
\end{definition}

Notice that the nomenclature above is due to the fact that the action of the normalizer intertwines the modules $\mathfrak{m}_i \in \varphi$ in a precise sense described by the following lemma.

\begin{lemma} \label{lem:pij}
Let $\varphi = (\mathfrak{m}_1,{\dots},\mathfrak{m}_{\ell})$ be normalizer-adapted. Fix $1 \leq p \leq \ell_0$ and $1 \leq i \leq \ell$. If there exists $1 \leq j \leq \ell$ such that $[pij] >0$, then $[pik] = 0$ for any $k \in \{1,{\dots},\ell\} \setminus \{j\}$. As a consequence, $\mathfrak{m}_i \simeq \mathfrak{m}_j$ and ${\rm ad}(V_p)(\mathfrak{m}_i) = \mathfrak{m}_j$.
\end{lemma}

\begin{proof}
Fix $1 \leq p \leq \ell_0$, $1 \leq i \leq \ell$ and assume that there exists $1 \leq j \leq \ell$ such that $[pij] >0$. By means of Remark \ref{rem:gauge_action}, the image $\widetilde{\mathfrak{m}}_i \coloneqq {\rm ad}(V_p)(\mathfrak{m}_i)$ is ${\rm Ad}(\mathsf{H})$-equivalent to $\mathfrak{m}_i$ and thus it is contained in the linear span of the isotypical class of $\mathfrak{m}_i$. By assumption, there exists $1 \leq j' \leq \ell$ such that $\widetilde{\mathfrak{m}}_i \cap \mathfrak{m}_{j'} \neq \{0\}$ and hence $\widetilde{\mathfrak{m}}_i = \mathfrak{m}_{j'}$. By hypothesis, $j' = j$ and the claim follows.
\end{proof}

For any $1 \leq p \leq \ell_0$, we can therefore consider the splitting
\begin{equation} \label{eq:I+I0} \begin{gathered}
\{1 , \dots , \ell \} = I^+_p \cup I^0_p \,\, , \quad \text{ with } \\
I^0_p \coloneqq \{1 \leq i \leq \ell : \text{${\rm ad}(V_p)|_{\mathfrak{m}_i}$ is trivial }\} \,\, , \quad I^+_p \coloneqq \{1 , \dots , \ell \} \setminus I^0_p
\end{gathered} \end{equation}
and the map
$$
\phi_p : \{1 , \dots , \ell \} \to \{1 , \dots , \ell \} \,\, , \quad
\phi_p(i) \coloneqq \left\{\begin{array}{ll}
\text{the unique $1 \leq j \leq \ell$ that satisfies $[pij]>0$} \, , &\!\!\! \text{if $i \in I_p^+$} \\
i \, , &\!\!\! \text{if $i \in I_p^0$}
\end{array}\right. \,\, .
$$
By Lemma \ref{lem:pij}, for any $i \in I^0_p$ there exists a unique element $j \in I^0_p$ such that $[pij] > 0$, and hence the map $\phi_p$ is well defined. Moreover, since $[pij] = [pji]$ (see \eqref{eq:[ijk]}), it follows that $\phi_p$ is bijective with $\phi_p^{-1} = \phi_p$. For the sake of notation, we consider the splitting
\begin{equation} \label{eq:II'I''} \begin{gathered}
I^+_p = I_p \cup I'_p \cup I''_p \,\, , \quad \text{ with } \\
I_p \coloneqq \{i \in I^+_p : i < \phi_p(i)\} \,\, , \quad I'_p \coloneqq \{i \in I^+_p : i = \phi_p(i)\} \,\, , \quad I''_p \coloneqq \{i \in I^+_p : i > \phi_p(i)\} \,\, .
\end{gathered} \end{equation}
Let us also emphasize that, for any $j \in I'_p$, ${\rm ad}(V_p)|_{\mathfrak{m}_j}$ is a linear $Q$-skew-symmetric isomorphism and therefore the dimension $d_j$ is even. Notice also that, by Schur's Lemma, we can find a $\varphi$-adapted, $Q$-orthonormal basis $\EuScript{B}$ such that the following hold:
\begin{itemize}
\item[$i)$] for any $i \in I_p$, if $\{e_{\alpha}\} = \EuScript{B} \cap \mathfrak{m}_i$ and $\{e_{\bar{\alpha}}\} = \EuScript{B} \cap \mathfrak{m}_{\phi_p(i)}$, then
\begin{equation} \label{eq:diagadN}
[V_p,e_{\alpha}] = \mu_{p,i}\, e_{\bar{\alpha}} \,\, , \quad [V_p,e_{\bar{\alpha}}] = -\mu_{p,i}\, e_{\alpha}
\end{equation}
for any $1 \leq \alpha \leq d_i$, for some coefficients $\mu_{p,i} \in \mathbb{R} \setminus \{0\}$;
\item[$ii)$] for any $j \in I'_p$, if $\{e_{\alpha},e_{\bar{\alpha}}\} = \EuScript{B} \cap \mathfrak{m}_j$, then
$$
[V_p,e_{\alpha}] = \mu_{p,j}\, e_{\bar{\alpha}} \,\, , \quad [V_p,e_{\bar{\alpha}}] = -\mu_{p,j}\, e_{\alpha}
$$
for any $1 \leq \alpha \leq \frac{d_j}2$, for some coefficients $\mu_{p,j} \in \mathbb{R} \setminus \{0\}$.
\end{itemize}
A direct computation shows that the coefficients $\mu_{p,i}$ can be explicitly determined, up to a sign, by the relation
\begin{equation} \label{eq:phimu}
[pi\phi_p(i)] = d_i (\mu_{p,i})^2 \quad \text{ for any } i \in I^+_p \,\, .
\end{equation}
As a matter of notation, we set $\mu_{p,i} \coloneqq 0$ for any $i \in I^0_p$, so that \eqref{eq:phimu} holds true for any $1 \leq i \leq \ell$. \smallskip

For later use, we observe that for any $1 \leq p, q \leq \ell_0$, by the Jacobi identity we have
\begin{equation} \label{eq:commphip}
[V_p,V_q] = 0 \quad\Longrightarrow\quad \phi_p \circ \phi_q = \phi_q \circ \phi_p \,\, .
\end{equation}
We also prove the following two results about the structure constants of $\varphi$, which will play a crucial role in the proof of our main result.

\begin{lemma}
Let $\varphi = (\mathfrak{m}_1,{\dots},\mathfrak{m}_{\ell})$ be normalizer-adapted. For any $1 \leq p \leq \ell_0$ and $1 \leq i \leq \ell$, we have
\begin{gather}
d_i = d_{\phi_p(i)} \,\, , \label{eq:coef1} \\
b_i = b_{\phi_p(i)} \,\, . \label{eq:coef2}
\end{gather}
\end{lemma}

\begin{proof}
It is sufficient to prove the lemma for $i \in I_p$. Equation \eqref{eq:coef1} is true since $\mathfrak{m}_i$ and $\mathfrak{m}_{\phi_p(i)}$ are ${\rm Ad}(\mathsf{H})$-equivalent and so, in particular, are isomorphic as linear spaces. Moreover, since $\mathfrak{m}_i$ and $\mathfrak{m}_{\phi_p(i)}$ are intertwined by the adjoint action of a vector $V_p \in \mathfrak{m}_0$, it follows that both $\mathfrak{m}_i$ and $\mathfrak{m}_{\phi_p(i)}$ are contained in a single simple ideal of $\mathfrak{g}$. Therefore, since $Q$ is ${\rm Ad}(\mathsf{G})$-invariant, \eqref{eq:coef2} follows.
\end{proof}

\begin{lemma} \label{lem:rowreduction}
Let $\varphi = (\mathfrak{m}_1,{\dots},\mathfrak{m}_{\ell})$ be normalizer-adapted. Fix $1 \leq p \leq \ell_0$ and $i \in I_p$. \begin{itemize}
\item[$a)$] If $j \in I^+_p$ and $1 \leq k \leq \ell$, then
\begin{equation} \label{eq:strconst1}
[\phi_p(i)\phi_p(j)k] = [ijk] \,\, .
\end{equation}
\item[$b)$] If $j,k \in I^0_p$, then 
\begin{equation} \label{eq:strconst2}
[\phi_p(i)jk] = [ijk] \,\, .
\end{equation}
\item[$c)$] If $1 \leq q \leq \ell_0$ with $[V_p,V_q]=0$, then 
\begin{equation} \label{eq:strconst3}
[q\phi_p(i)\phi_q(\phi_p(i))] = [qi\phi_q(i)] \,\, .
\end{equation}
\end{itemize}
\end{lemma}

\begin{proof}
Fix $1 \leq p \leq \ell_0$, $i \in I_p$ and, for the sake of shortness, set $\overline{r} \coloneqq \phi_p(r)$ for any $1 \leq r \leq \ell$. Fix a $\varphi$-adapted, $Q$-orthonormal basis $\EuScript{B}$ for $\mathfrak{m}$ and assume that it diagonalizes ${\rm ad}(V_p)$ as in \eqref{eq:diagadN}. Then, for any $1 \leq j, k \leq \ell$ we have
\begin{equation} \label{eq:Jacobi} \begin{gathered}
\mu_{p,k} Q([e_{\bar{\alpha}},e_{\beta}],e_{\bar{\gamma}}) +\mu_{p,j} Q([e_{\bar{\alpha}},e_{\bar{\beta}}],e_{\gamma}) -\mu_{p,i}Q([e_{\alpha},e_{\beta}],e_{\gamma}) = 0 \,\, , \\
\mu_{p,k}Q([e_{\alpha},e_{\bar{\beta}}],e_{\bar{\gamma}}) +\mu_{p,i} Q([e_{\bar{\alpha}},e_{\bar{\beta}}],e_{\gamma}) -\mu_{p,j}Q([e_{\alpha},e_{\beta}],e_{\gamma}) = 0
\end{gathered} \end{equation}
for any $e_{\alpha} \in \EuScript{B} \cap \mathfrak{m}_i$, $e_{\beta} \in \EuScript{B} \cap \mathfrak{m}_j$, $e_{\gamma} \in \EuScript{B} \cap \mathfrak{m}_k$. Indeed, by \eqref{eq:diagadN}, the ${\rm Ad}(\mathsf{G})$-invariance of $Q$ and the Jacobi Identity, it follows that
\begin{multline*}
\mu_{p,k} Q([e_{\bar{\alpha}},e_{\beta}],e_{\bar{\gamma}}) +\mu_{p,j} Q([e_{\bar{\alpha}},e_{\bar{\beta}}],e_{\gamma}) -\mu_{p,i}Q([e_{\alpha},e_{\beta}],e_{\gamma}) = \\
= -Q([V_p,[e_{\bar{\alpha}},e_{\beta}]],e_{\gamma}) -Q([e_{\bar{\alpha}},[e_{\beta},V_p]],e_{\gamma}) -Q([e_{\beta},[V_p,e_{\bar{\alpha}}]],e_{\gamma}) = 0 \,\, .
\end{multline*}
The second equation in \eqref{eq:Jacobi} is obtained similarly.

By taking the squares of \eqref{eq:Jacobi}, we have
$$\begin{aligned}
\mu_{p,k}^2 Q([e_{\bar{\alpha}},e_{\beta}],e_{\bar{\gamma}})^2 &= \mu_{p,j}^2 Q([e_{\bar{\alpha}},e_{\bar{\beta}}],e_{\gamma})^2 +\mu_{p,i}^2Q([e_{\alpha},e_{\beta}],e_{\gamma})^2 \\
&\qquad\qquad\qquad -2\mu_{p,i}\mu_{p,j} Q([e_{\bar{\alpha}},e_{\bar{\beta}}],e_{\gamma}) Q([e_{\alpha},e_{\beta}],e_{\gamma}) \,\, ,\\
\mu_{p,k}^2 Q([e_{\alpha},e_{\bar{\beta}}],e_{\bar{\gamma}}))^2 &= \mu_{p,i}^2 Q([e_{\bar{\alpha}},e_{\bar{\beta}}],e_{\gamma})^2 +\mu_{p,j}^2Q([e_{\alpha},e_{\beta}],e_{\gamma})^2 \\
&\qquad\qquad\qquad -2\mu_{p,i} \mu_{p,j}Q([e_{\bar{\alpha}},e_{\bar{\beta}}],e_{\gamma})Q([e_{\alpha},e_{\beta}],e_{\gamma}),
\end{aligned}$$
and so
\begin{multline} \label{eq:Jacobi2}
\mu_{p,k}^2 Q([e_{\bar{\alpha}},e_{\beta}],e_{\bar{\gamma}})^2 -\mu_{p,j}^2 Q([e_{\bar{\alpha}},e_{\bar{\beta}}],e_{\gamma})^2 -\mu_{p,i}^2Q([e_{\alpha},e_{\beta}],e_{\gamma})^2 = \\
\mu_{p,k}^2 Q([e_{\alpha},e_{\bar{\beta}}],e_{\bar{\gamma}}))^2 -\mu_{p,i}^2 Q([e_{\bar{\alpha}},e_{\bar{\beta}}],e_{\gamma})^2 -\mu_{p,j}^2Q([e_{\alpha},e_{\beta}],e_{\gamma})^2 \,\, .
\end{multline}
Taking sums in \eqref{eq:Jacobi2}, we get
\begin{equation} \label{eq:proofcoef}
\mu_{p,k}^2 [\bar{i}j\bar{k}] -\mu_{p,j}^2 [\bar{i}\bar{j}k] -\mu_{p,i}^2 [ijk] = \mu_{p,k}^2 [i\bar{j}\bar{k}] -\mu_{p,i}^2 [\bar{i}\bar{j}k] -\mu_{p,j}^2 [ijk] \,\, .
\end{equation}
By applying cyclic index permutations and the index transformations
$$
(j,k) \mapsto (\bar{j},\bar{k}) \,\, , \quad (i,j) \mapsto (\bar{i},\bar{j}) \,\, , \quad (i,k) \mapsto (\bar{i},\bar{k})
$$
to \eqref{eq:proofcoef}, we obtain the following linear system:
\begin{equation} \label{eq:linsyst1}
\left\{\!\!\begin{array}{rrrrr}
\mu_{p,i}^2 [\bar{i}\bar{j}k] &\!\!\!\!
+(\mu_{p,j}^2 -\mu_{p,k}^2) [i\bar{j}\bar{k}] &\!\!\!\!
-\mu_{p,i}^2 [\bar{i}j\bar{k}] &\!\!\!\!
-(\mu_{p,j}^2 -\mu_{p,k}^2) [ijk] &\!\!\!\!
= 0 \\
\mu_{p,j}^2 [\bar{i}\bar{j}k] &\!\!\!\!
-\mu_{p,j}^2 [i\bar{j}\bar{k}] &\!\!\!\!
+(\mu_{p,i}^2 -\mu_{p,k}^2) [\bar{i}j\bar{k}] &\!\!\!\!
-(\mu_{p,i}^2 -\mu_{p,k}^2) [ijk] &\!\!\!\!
= 0 \\
\mu_{p,k}^2 [\bar{i}\bar{j}k] &\!\!\!\!
-(\mu_{p,i}^2 -\mu_{p,j}^2) [i\bar{j}\bar{k}] &\!\!\!\!
+(\mu_{p,i}^2 -\mu_{p,j}^2) [\bar{i}j\bar{k}] &\!\!\!\!
-\mu_{p,k}^2 [ijk] &\!\!\!\!
= 0 \\
(\mu_{p,i}^2 -\mu_{p,j}^2) [\bar{i}\bar{j}k] &\!\!\!\!
-\mu_{p,k}^2 [i\bar{j}\bar{k}] &\!\!\!\!
+\mu_{p,k}^2 [\bar{i}j\bar{k}] &\!\!\!\!
-(\mu_{p,i}^2 -\mu_{p,j}^2) [ijk] &\!\!\!\!
= 0 \\
(\mu_{p,j}^2 -\mu_{p,k}^2) [\bar{i}\bar{j}k] &\!\!\!\!
+\mu_{p,i}^2 [i\bar{j}\bar{k}] &\!\!\!\!
-(\mu_{p,j}^2 -\mu_{p,k}^2) [\bar{i}j\bar{k}] &\!\!\!\!
-\mu_{p,i}^2 [ijk] &\!\!\!\!
= 0 \\
(\mu_{p,i}^2 -\mu_{p,k}^2) [\bar{i}\bar{j}k] &\!\!\!\!
-(\mu_{p,i}^2 -\mu_{p,k}^2) [i\bar{j}\bar{k}] &\!\!\!\!
+\mu_{p,j}^2 [\bar{i}j\bar{k}] &\!\!\!\!
-\mu_{p,j}^2 [ijk] &\!\!\!\!
= 0
\end{array} \right. \,\, . \end{equation}
Let us write the the system \eqref{eq:linsyst1} in the form
$$
A \cdot \big([\bar{i}\bar{j}k], [i\bar{j}\bar{k}], [\bar{i}j\bar{k}], [ijk]\big)^t = 0 \,\, .
$$
Then, the matrix $A$ factorizes as the product $A = P\cdot A'$, where $P$ is the invertible matrix
$$
P \coloneqq \left(\!\!\begin{array}{cccccc}
\tfrac12 & -\tfrac12 & \tfrac12 & \tfrac12 & 0 & 0 \\
0 & 0 & 0 & 0 & 1 & 0 \\
0 & 0 & 0 & 0 & 0 & 1 \\
1 & 1 & 0 & 0 & 0 & -1 \\
-\tfrac12 & \tfrac12 & \tfrac12 & \tfrac12 & 0 & 0 \\
0 & 0 & 1 & -1 & -1 & 0
\end{array}\!\!\right)
$$
and $A'$ is given by
$$
A' \coloneqq \left(\!\!\begin{array}{cccc}
\mu_{p,i}^2 -\mu_{p,j}^2 +\mu_{p,k}^2 & -\mu_{p,i}^2 +\mu_{p,j}^2 -\mu_{p,k}^2 & 0 & 0 \\[2pt]
0 & 0 & \mu_{p,i}^2 -\mu_{p,j}^2 +\mu_{p,k}^2 & -\mu_{p,i}^2 +\mu_{p,j}^2 -\mu_{p,k}^2 \\[2pt]
\mu_{p,i}^2 +\mu_{p,j}^2 -\mu_{p,k}^2 & 0 & 0 & -\mu_{p,i}^2 -\mu_{p,j}^2 +\mu_{p,k}^2 \\[2pt]
0 & \mu_{p,i}^2 +\mu_{p,j}^2 -\mu_{p,k}^2 & -\mu_{p,i}^2 -\mu_{p,j}^2 +\mu_{p,k}^2 & 0 \\[2pt]
\mu_{p,j}^2 & -\mu_{p,j}^2 & \mu_{p,i}^2 -\mu_{p,k}^2 & -\mu_{p,i}^2 +\mu_{p,k}^2 \\[2pt]
\mu_{p,k}^2 & -\mu_{p,i}^2 +\mu_{p,j}^2 & \mu_{p,i}^2 -\mu_{p,j}^2 & -\mu_{p,k}^2 
\end{array}\!\!\right) \,\, .
$$ \smallskip

Notice that, since $\mu_{p,i} \neq 0$ by assumption, the quantities
$$
\mu_{p,i}^2 -\mu_{p,j}^2 +\mu_{p,k}^2 \,\, , \quad \mu_{p,i}^2 +\mu_{p,j}^2 -\mu_{p,k}^2
$$
cannot vanish simultaneously. Then, by explicit computation, the row-echelon form of $A'$ is as described in the following table.

\begin{table}[h!]
\centering
\label{tab:example}
\begin{tabular}{|c|c|c|c|}
\hline
\text{If $\mu_{p,i}^2 -\mu_{p,j}^2 +\mu_{p,k}^2 = 0$} & \text{If $\mu_{p,i}^2 +\mu_{p,j}^2 -\mu_{p,k}^2 = 0$} & \text{If $\mu_{p,i}^2 -\mu_{p,j}^2 +\mu_{p,k}^2 \neq 0$} \\
& & \text{and $\mu_{p,i}^2 +\mu_{p,j}^2 -\mu_{p,k}^2 \neq 0$} \\
\hline
  & & \\
$\left(\!\!\begin{array}{cccc}
1 & 0 & 0 & -1 \\
0 & 1 & -1 & 0 \\
0 & 0 & \mu_{p,i}^2 -\mu_{p,j}^2 & -\mu_{p,i}^2 +\mu_{p,j}^2 \\
0 & 0 & 0 & 0 \\
0 & 0 & 0 & 0 \\
0 & 0 & 0 & 0
\end{array}\!\!\right)$
&
$\left(\!\!\begin{array}{cccc}
1 & -1 & 0 & 0 \\
0 & \mu_{p,i}^2 -\mu_{p,k}^2 & -\mu_{p,i}^2 +\mu_{p,k}^2 & 0 \\
0 & 0 & 1 & -1 \\
0 & 0 & 0 & 0 \\
0 & 0 & 0 & 0 \\
0 & 0 & 0 & 0
\end{array}\!\!\right)$
&
$\left(\!\!\begin{array}{cccc}
1 & 0 & 0 & -1 \\
0 & 1 & 0 & -1 \\
0 & 0 & 1 & -1 \\
0 & 0 & 0 & 0 \\
0 & 0 & 0 & 0 \\
0 & 0 & 0 & 0
\end{array}\!\!\right)$ \\
& & \\
\hline
\end{tabular}
\end{table}

Therefore, we obtain the following characterization: \begin{itemize}
\item[$(1)$] if $\mu_{p,j} = 0$ and $\mu_{p,k}^2 = \mu_{p,i}^2$, then ${\rm rank}(A) = 2$ and the solutions to \eqref{eq:linsyst1} are
$$
[\bar{i}\bar{j}k] = [i\bar{j}\bar{k}] \,\, , \quad [\bar{i}j\bar{k}] = [ijk] \,\, ;
$$
\item[$(2)$] if $\mu_{p,k} = 0$ and $\mu_{p,j}^2 = \mu_{p,i}^2$, then ${\rm rank}(A) = 2$ and the solutions to \eqref{eq:linsyst1} are
$$
[\bar{i}\bar{j}k] = [ijk] \,\, , \quad [i\bar{j}\bar{k}] = [\bar{i}j\bar{k}] \,\, ;
$$
\item[$(3)$] in all the other cases, ${\rm rank}(A) = 3$ and the solutions to \eqref{eq:linsyst1} are
$$
[\bar{i}\bar{j}k] = [i\bar{j}\bar{k}] = [\bar{i}j\bar{k}] = [ijk] \,\, .
$$
\end{itemize}
Based on this characterization, we proceed as follows.
\begin{itemize}
\item[$a)$] If $j \in I^+_p$, then $\mu_{p,j}^2 > 0$, and we are either in case $(2)$ or $(3)$. Hence $[\bar{i}\bar{j}k] = [ijk]$, that is \eqref{eq:strconst1}.
\item[$b)$] If $j,k \in I^0_p$, then $\mu_{p,j} = \mu_{p,k} = 0$, so we are in case $(3)$. Moreover, since $j=\bar{j}$ and $k=\bar{k}$, we obtain $[\bar{i}jk] = [ijk]$, that is \eqref{eq:strconst2}.
\item[$c)$] Let $1 \leq q \leq \ell_0$ be such that $[V_p,V_q]=0$ and set $j \coloneqq \phi_q(i)$, $k \coloneqq q$. Then $\mu_{p,k} = 0$, and we are again in case $(2)$ or $(3)$. Hence $[\bar{i}\bar{j}k] = [ijk]$. By \eqref{eq:commphip}, this gives \eqref{eq:strconst3}.
\end{itemize}
This concludes the proof.
\end{proof}

\subsection{Diagonalizing the Ricci tensor} \label{subsect:diagRicci} \hfill \par

Let $\varphi = (\mathfrak{m}_1,{\dots},\mathfrak{m}_{\ell}) \in \EuScript{F}^{\mathsf{G}}$ be a decomposition of $\mathfrak{m}$ and $g \in \EuScript{M}^{\mathsf{G}}$ be $\varphi$-diagonal. Its Ricci tensor ${\rm Ric}(g)$ generally has off-diagonal terms. More precisely, for any $X \in \mathfrak{m}_i$ and $Y \in \mathfrak{m}_j$, with $i \neq j$, the following hold true: \begin{itemize}
\item[$\bcdot$] if $\mathfrak{m}_i \not\simeq \mathfrak{m}_j$, then ${\rm Ric}(g)(X,Y) = 0$ by Schur's Lemma;
\item[$\bcdot$] if $\mathfrak{m}_i \simeq \mathfrak{m}_j$, then
\begin{equation} \label{eq:offdiag-Ricci}
{\rm Ric}(g)(X,Y) = \sum_{1 \leq r,k \leq \ell} \frac{x_ix_j -2x_r^2 +2x_rx_k}{4x_rx_k} \sum_{\substack{e_{\alpha} \in \mathfrak{m}_r \\ e_{\beta} \in \mathfrak{m}_k}}Q([e_{\alpha},e_{\beta}],X)Q([e_{\alpha},e_{\beta}],Y) \,\, ,
\end{equation}
where $\{e_{\alpha}\}$ is any $Q$-orthonormal, $\varphi$-adapted basis for $\mathfrak{m}$ (see, e.g., \cite[Section 3]{K21}).
\end{itemize}
For this reason, we recall the following definition introduced in \cite{K21}. A basis $\EuScript{B}$ adapted to $\varphi$ is said to be {\it stably Ricci-diagonal} if any $\EuScript{B}$-diagonal, invariant metric has $\EuScript{B}$-diagonal Ricci tensor. In fact, it can be seen that this is actually a property of the decomposition $\varphi$, so we may equivalently define it as follows.

\begin{definition}
A decomposition $\varphi \in \EuScript{F}^{\mathsf{G}}$ is said to be {\it stably Ricci-diagonal} if any $\varphi$-diagonal, invariant metric $g \in \EuScript{M}^{\mathsf{G}}$ has $\varphi$-diagonal Ricci tensor.
\end{definition}

The role of such decompositions in the study of the Ricci flow is given by the following observation.

\begin{proposition} \label{prop:diagRF}
Let $g(t)$ be a homogeneous Ricci flow solution, $t \in I$, on a compact manifold $M$ and assume that $M$ admits a stably Ricci-diagonal decomposition $\varphi$ for its isotropy representation. If there exists $t^* \in I$ such that $g(t^*)$ is $\varphi$-diagonal, then $g(t)$ is $\varphi$-diagonal for any $t \in I$.
\end{proposition}

\begin{proof}
Let $\varphi$ be stably Ricci-diagonal and consider the subset
$$
D \coloneqq \{t \in I : \text{$g(t)$ is $\varphi$-diagonal} \} \,\, .
$$
Then, $D$ is non-empty and closed. We show below that $D$ is also open, which leads to $D=I$.

Fix $t_1 \in D$ and consider the following system of ODEs:
\begin{equation} \label{eq:diagsyst}
\dot{x}_i(t) = -2{\rm ric}_i\big(x_1(t),{\dots},x_{\ell}(t)\big) \,\, , \quad 1 \leq i \leq \ell
\end{equation}
where the functions ${\rm ric}_i$ are defined in \eqref{eq:ric}. By the existence theorem for ODEs, there exists a solution $\big(x_1(t),{\dots},x_{\ell}(t)\big)$ to \eqref{eq:diagsyst} defined for $t \in (t_1-\epsilon,t_1+\epsilon)$, for some $\epsilon >0$, such that $\big(x_1(t_1),{\dots},x_{\ell}(t_1)\big)$ is the tuple of eigenvalues of $g(t_1)$. Since $g(t_1)$ is positive definite, up to choosing a smaller $\epsilon$, it follows that
$$
h(t) \coloneqq x_1(t)\,Q|_{\mathfrak{m}_1 \otimes \mathfrak{m}_1} + {\dots} + x_{\ell}(t)\,Q|_{\mathfrak{m}_{\ell} \otimes \mathfrak{m}_{\ell}}
$$
is a $\mathsf{G}$-invariant Riemannian metric on $M$ and $(t_1-\epsilon,t_1+\epsilon) \subset I$. Since $\varphi$ is stably Ricci-diagonal, it follows that $h(t)$ is a Ricci flow solution. Therefore, by uniqueness of solutions to the Ricci flow on compact manifolds, $h(t) = g(t)$ for every $t \in (t_1-\epsilon,t_1+\epsilon)$. In particular, $(t_1-\epsilon,t_1+\epsilon) \subset D$. 
\end{proof}

Stably Ricci-diagonal bases for Lie algebras were introduced and named in \cite{Pay10}, where they were used to study the Ricci flow on nilmanifolds. Later, in \cite{LauW13}, the authors introduced the notion of {\it nice bases} for nilpotent Lie algebras in terms of a condition on the structure constants and they proved that a basis is nice if and only if is stably Ricci diagonal. Later, these two notions were extended and generalized to compact semisimple Lie algebras and compact homogeneous spaces in \cite{K21}, where the following theorem was proved.

\begin{theorem}[\cite{K21} Theorem A] \label{thm:K21}
A $Q$-orthonormal, $\varphi$-adapted basis $\EuScript{B} = \{e_{\alpha}\}$ for $\mathfrak{m}$ is stably Ricci-diagonal if and only if the following property holds: if $\mathfrak{m}_i \simeq \mathfrak{m}_j$ and $i \neq j$, then
\begin{equation} \label{eq:nicebasis}
\sum_{\substack{e_{\alpha} \in \mathfrak{m}_k \\ e_{\beta} \in \mathfrak{m}_r}} Q([e_{\alpha},e_{\beta}], v)\,Q([e_{\alpha},e_{\beta}], w) = 0 \quad \text{for any $1 \leq k, r \leq \ell$, $v \in \EuScript{B} \cap \mathfrak{m}_i$, $w \in \EuScript{B} \cap \mathfrak{m}_j$} \,\, .
\end{equation}
\end{theorem}

In the special case where the homogeneous space is a compact Lie group $\mathsf{G}$, the condition \eqref{eq:nicebasis} is equivalent to the following statement: the bracket of any two basis elements is a multiple of another basis element. In this setting, this is equivalent to the nice basis condition defined by Lauret and Will, see \cite[Remark 3]{K21}. Hence, following \cite{K21}, a basis $\EuScript{B}$, adapted to a decomposition $\varphi \in \EuScript{F}^{\mathsf{G}}$, satisfying \eqref{eq:nicebasis} is called a {\it nice basis for $\mathfrak{m}$}. Notice that we can also restate \eqref{eq:nicebasis} in a basis-free manner as follows: if $\mathfrak{m}_i, \mathfrak{m}_j \in \varphi$ are equivalent and distinct, then
$$
{\rm Tr}\big({\rm ad}(X) \circ {\rm ad}(Y) |_{\mathfrak{m}_k}\big) = 0 \quad \text{for any $1 \leq k \leq \ell$, $X \in \mathfrak{m}_i$, $Y \in \mathfrak{m}_j$} \,\, .
$$

Finally, for the sake of shortness, we introduce the following nomenclature.

\begin{definition} \label{def:stronglyRd}
An element $\varphi \in \EuScript{F}^{\mathsf{G}}$ is said to be an {\it NR-decomposition of $\mathfrak{m}$} if it is both normalizer-adapted and stably Ricci-diagonal.
\end{definition}

We are ready to prove Proposition \ref{prop:main-uniq.torus}.

\begin{proof}[Proof of Proposition \ref{prop:main-uniq.torus}]
By hypotheses and Proposition \ref{prop:diagRF} the metric $g(t)$ is diagonal, i.e., we can write
$$
g(t) = x_1(t)\,Q|_{\mathfrak{m}_1 \otimes \mathfrak{m}_1} + {\dots} + x_{\ell}(t)\,Q|_{\mathfrak{m}_{\ell} \otimes \mathfrak{m}_{\ell}}
$$
for some NR-decomposition, where $x_i: (-\infty,0) \to (0,+\infty)$ are smooth functions. Since $N_{\mathsf{G}}(\mathsf{H})/\mathsf{H}$ has rank $1$, and its Lie algebra is identified with $\mathfrak{m}_0$ (see Section \ref{sect:prelhomsp}), either $\mathfrak{m}_0 = \mathfrak{m}_1 \simeq \mathbb{R}$, or $\mathfrak{m}_0 = \mathfrak{m}_1 +\mathfrak{m}_2 +\mathfrak{m}_3 \simeq \mathfrak{so}(3)$. If $\mathfrak{m}_0 = \mathfrak{m}_1 \simeq \mathbb{R}$, then the collapsing torus is necessarily unique, and its Lie algebra coincides with all of $\mathfrak{m}_0$.

Assume now that $\mathfrak{m}_0 = \mathfrak{m}_1 +\mathfrak{m}_2 +\mathfrak{m}_3 \simeq \mathfrak{so}(3)$ and consider the continuous function
$$
\lambda(t) \coloneqq \min\{x_1(t), x_2(t), x_3(t)\} \,\, .
$$
Notice that $\tfrac{\lambda(t)}{|t|} \to 0$ as $t \to -\infty$. Indeed, using Theorem \ref{thm:limittorus}, we see that for any sequence of times $t^{(n)} \to -\infty$, there exist a subsequence $\{t^{(n_k)}\} \subset \{t^{(n)}\}$ such that
$$
\frac{\lambda(t^{(n_k)})}{|t^{(n_k)}|} \to 0
$$
as $k \to +\infty$. Consider also the set of all times where the minimum value $\lambda(t)$ is achieved by more than one eigenvalues, that is
$$
\EuScript{X} \coloneqq \big\{t \in (-\infty,-1] : \text{ there exist $1 \leq i < j \leq 3$ such that $\lambda(t) = x_i(t) = x_j(t)$}\big\} \,\, .
$$
We claim that there exists $T>1$ such that $\EuScript{X} \subset [-T,-1]$, i.e., there exists an eigenvalue which is strictly less than the other two for all times $t < -T$.

In fact, assume by contradiction that there is a sequence of times $\{t^{(n)}\} \subset \EuScript{X}$ such that $t^{(n)} \to -\infty$. Then, up to passing to a subsequence, there exist $1 \leq i < j \leq 3$ such that
$$
\lambda(t^{(n)}) = x_i(t^{(n)}) = x_j(t^{(n)}) \quad \text{ for any $n \in \mathbb{N}$} \,\, ,
$$
which in turn implies that
$$
\lim_{n \to +\infty}\frac{x_i(t^{(n)})}{|t^{(n)}|} = \lim_{n \to +\infty}\frac{x_j(t^{(n)})}{|t^{(n)}|} = 0 \,\, .
$$
By claim $iii)$ in Theorem \ref{thm:limittorus}, it follows that ${\rm span}_{\mathbb{R}}\{V_i, V_j\}$ is an abelian Lie subalgebra of $\mathfrak{m}_0$, which contradicts the hypothesis $\mathfrak{m}_0 \simeq \mathfrak{so}(3)$.

As a consequence, up to a permutation of the indices $\{1,2,3\}$, it follows that $\frac1{|t|}x_1(t) \to 0$ for $t \to -\infty$. By applying claim $iii)$ in Theorem \ref{thm:limittorus} as above, since the rank of $\mathfrak{m}_0$ is 1, it follows that $\tfrac1{|t|}x_2(t)$ and $\tfrac1{|t|}x_3(t)$ are bounded away from zero as $t \to -\infty$. Hence, the collapsing torus is unique, and its Lie algebra coincides with $\mathfrak{m}_1$.
\end{proof}

\begin{remark} \label{rem:unique-rigid}
The proof above shows that, under the assumptions of Proposition \ref{prop:main-uniq.torus}, the collapsing torus $\mathsf{T}$ of $g(t)$ is not only unique, but is also rigid. Indeed, by the proof of Proposition \ref{prop:main-uniq.torus}, the Lie algebra of $\mathsf{T}$ is the one dimensional trivial submodule $\mathfrak{m}_1$. Since $g(t)$ is diagonal with respect to an NR-decomposition $\varphi$ such that $\mathfrak{m}_1 \in \varphi$, condition $ii)$ of Definition \ref{def:sstorus} follows. Moreover, the proof also shows that $\frac{1}{|t|}x_1(t) \to 0$ as $t \to -\infty$, and that $\frac{1}{|t|}x_i(t)$ are uniformly bounded away from $0$ as $t \to -\infty$, for $2 \leq i \leq \ell$. Hence, conditions $i)$ and $iii)$ of Definition \ref{def:sstorus} follow.
\end{remark}

\section{Existence of toral symmetries}
\label{sect:mainproof} \setcounter{equation} 0

This section is devoted to the proof of Theorem \ref{thm:main-extra.symm} and consists of two subsections. In Subsection \ref{subsect:defsubtens}, we introduce the so-called {\it submersion tensor}, which describes quantitatively how far the action of a subgroup of the gauge group is from being isometric. In Subsection \ref{subsect:StensRF}, we study submersion tensors associated to collapsing tori of homogeneous ancient solutions. In particular, we prove a monotonicity formula, which is the key ingredient in proving Theorem \ref{thm:main-extra.symm}.

\subsection{The submersion tensor and a formula for the Ricci eigenvalues} \label{subsect:defsubtens} \hfill \par

We begin this subsection by introducing the following tensor that is of central importance for proving Theorem \ref{thm:main-extra.symm}. We recall that $(N_{\mathsf{G}}(\mathsf{H})/\mathsf{H})_0$ is identified with a compact subgroup of $\mathsf{G}$ (see Section \ref{sect:prelhomsp}). 

\begin{definition}
Let $\mathsf{L} \subset (N_{\mathsf{G}}(\mathsf{H})/\mathsf{H})_0$ be a connected Lie subgroup, with ${\rm Lie}(\mathsf{L}) = \mathfrak{l} \subset \mathfrak{m}_0$, and $g \in \EuScript{M}^{\mathsf{G}}$. We define the {\it submersion tensor} associated to the pair $(\mathsf{L},g)$ as the $g$-symmetric part of the adjoint action of $\mathfrak{l}$ on $\mathfrak{m}$, i.e., the multilinear form $S = S(\mathsf{L},g) : \mathfrak{l} \otimes \mathfrak{m} \otimes \mathfrak{m} \to \mathbb{R}$ defined by
\begin{equation} \label{eq:S-tensor}
2\,S(V,X,Y) \coloneqq g([V,X],Y) + g([V,Y],X) \,\, .
\end{equation}
\end{definition}

Notice that the nomenclature is due to the fact that $S(\mathsf{L},g) = 0$ if and only if the Lie subgroup $\mathsf{L}$ acts $g$-isometrically on the right (see Remark \ref{rem:gauge_action}). The latter is equivalent to the existence of a $\mathsf{G}$-homogeneous metric $\check{g}$ on $\mathsf{G}/\mathsf{H}\mathsf{L}$ for which the natural projection $(\mathsf{G}/\mathsf{H}, g) \to (\mathsf{G}/\mathsf{H}\mathsf{L}, \check{g})$ is a Riemannian submersion with totally geodesic fibers. \smallskip

Assume now that $\varphi = (\mathfrak{m}_1, {\dots}, \mathfrak{m}_{\ell})$ is an NR-decomposition, let $1 \leq s \leq \ell_0$ be such that, up to reordering, the subspace
\begin{equation} \label{eq:t-proofB}
\mathfrak{t} \coloneqq \mathfrak{m}_1 + \dots + \mathfrak{m}_s
\end{equation}
is an abelian subalgebra of $\mathfrak{m}_0$ and denote by $\mathsf{T}$ the corresponding connected Lie subgroup. Let also $g \in \EuScript{M}^{\mathsf{G}}$ be $\varphi$-diagonal as in \eqref{eq:g-diag}. Fix $Q$-unitary generators $V_p$ in each trivial submodule $\mathfrak{m}_p \in \varphi$, with $1 \leq p \leq s$, as discussed in Section \ref{subsect:norm-diag}. We will now describe the corresponding bilinear forms
\begin{equation} \label{eq:subtens_p}
S_p \coloneqq S(\mathsf{T},g)(V_p, \cdot, \cdot) \,\, ,
\end{equation}
which will play a crucial role in proving our main theorem. \smallskip

Pick a $Q$-orthonormal, $\varphi$-adapted basis $\EuScript{B}$ for $\mathfrak{m}$. It is worth noting that if $i \in I_p$, $e_1 \in \mathfrak{m}_i \cap \EuScript{B}$ and $e_2 \in \EuScript{B}$, then $S_p(e_1,e_2) \neq 0$ only if $e_2 \in \mathfrak{m}_{\phi_p(i)}$. Here, the index set $I_p$ and the map $\phi_p$ are defined in Section \ref{subsect:norm-diag}, after the proof of Lemma \ref{lem:pij}. Therefore, for any $i \in I_p$, we define
\begin{equation} \label{eq:deftheta}
\theta_{p,i} \coloneqq \frac1{[pi\phi_p(i)]} \sum_{\substack{e_{\alpha} \in \mathfrak{m}_i \\ e_{\bar{\alpha}} \in \mathfrak{m}_{\phi_p(i)}}} \frac{S_p(e_{\alpha},e_{\bar{\alpha}})^2}{g(e_{\alpha},e_{\alpha})g(e_{\bar{\alpha}},e_{\bar{\alpha}})}
\end{equation}
so that 
\begin{equation} \label{eq:normS^2} 
\big(|S_p|_{g}\big)^2 = \sum_{i \in I_p} [pi\phi_p(i)]\, \theta_{p,i} \,\, .
\end{equation}
Moreover, a direct computation shows that, if $g$ is of the form \eqref{eq:g-diag}, then
\begin{equation} \label{eq:theta_x}
\theta_{p,i} = \frac14\Big(\frac{x_{\phi_p(i)}}{x_i} +\frac{x_i}{x_{\phi_p(i)}}\Big) -\frac12 \,\, .
\end{equation}
Observe that the quantity $|S_p|_{g}$ is scale invariant. \smallskip

We conclude this subsection by proving a preparatory formula.

\begin{proposition} \label{prop:ricciNR}
Let $g \in \EuScript{M}^{\mathsf{G}}$ be of the form \eqref{eq:g-diag} with respect to an NR-decomposition $\varphi$ and ${\rm ric}_i(g)$ its Ricci eigenvalues. For any $1 \leq p \leq s$, with $s$ as in \eqref{eq:t-proofB}, and for any $i \in I_p$, the following holds:
\begin{equation} \label{eq:diffricci} \begin{aligned}
&{\rm ric}_{\phi_{p}(i)}(g) -{\rm ric}_{i}(g) = \\
&\qquad = \mu_{p,i}^2\frac{x_{\phi_p(i)} -x_i}{x_p x_i x_{\phi_p(i)}}\bigg(x_{\phi_p(i)} +x_i -\frac{b_i}{2\mu_{p,i}^2}x_p\bigg) +\frac12\sum_{\substack{1 \leq q \leq s \\ q \neq p}}\mu_{q,i}^2\bigg(\frac{x_q}{x_i x_{\phi_q(i)}} -\frac{x_q}{x_{\phi_p(i)} x_{\phi_q(\phi_p(i))}}\bigg) \\
&\qquad +\frac12\sum_{\substack{1 \leq q \leq s \\ q \neq p}}\frac{\mu_{q,i}^2}{x_q}\bigg(\frac{x_{\phi_q(i)}}{x_i} -\frac{x_{\phi_q(\phi_p(i))}}{x_{\phi_p(i)}} +\frac{x_{\phi_p(i)}}{x_{\phi_q(\phi_p(i))}} -\frac{x_i}{x_{\phi_q(i)}}\bigg) \\
&\qquad -\frac1{2d_i}\sum_{\substack{s < j, k \leq \ell \\ j \in I_p^0 ,\ k \in I_p^0}}[ijk]\bigg(\frac{x_k}{x_{\phi_p(i)} x_j} -\frac{x_k}{x_i x_j}\bigg) +\frac1{4d_i}\sum_{\substack{s < j, k \leq \ell \\ j \in I_p^0 ,\ k \in I_p^0}}[ijk]\bigg(\frac{x_{\phi_p(i)}}{x_j x_k} -\frac{x_i}{x_j x_k}\bigg) \\
&\qquad -\frac1{2d_i}\sum_{\substack{s < j,k \leq \ell \\ j \in I_p^0 ,\ k \in I_p^+}}[ijk]\bigg(\frac{x_{\phi_p(k)}}{x_{\phi_p(i)} x_j} -\frac{x_k}{x_i x_j}\bigg) +\frac1{4d_i}\sum_{\substack{s < j,k \leq \ell \\ j \in I_p^0 ,\ k \in I_p^+}}[ijk]\bigg(\frac{x_{\phi_p(i)}}{x_j x_{\phi_p(k)}} -\frac{x_i}{x_j x_k}\bigg) \\
&\qquad -\frac1{2d_i}\sum_{\substack{s < j,k \leq \ell \\ j \in I_p^+}}[ijk]\bigg(\frac{x_k}{x_{\phi_p(i)} x_{\phi_p(j)}} -\frac{x_k}{x_i x_j}\bigg) +\frac1{4d_i}\sum_{\substack{s < j,k \leq \ell \\ j \in I_p^+}}[ijk]\bigg(\frac{x_{\phi_p(i)}}{x_{\phi_p(j)} x_k} -\frac{x_i}{x_j x_k}\bigg) \,\, .
\end{aligned} \end{equation}
\end{proposition}

\begin{proof}
An explicit computation shows that
\begin{align*}
&{\rm ric}_{\phi_{p}(i)}(g) -{\rm ric}_{i}(g) = \\
&\quad \overset{\eqref{eq:ric}}{=} \frac{b_{\phi_p(i)}}{2x_{\phi_p(i)}} -\frac{b_i}{2x_i} -\frac1{2d_{\phi_p(i)}}\sum_{1 \leq j,k \leq \ell}[\phi_p(i)jk]\frac{x_k}{x_{\phi_p(i)} x_j} +\frac1{2d_i}\sum_{1 \leq j,k \leq \ell}[ijk]\frac{x_k}{x_i x_j} \\
&\quad\qquad +\frac1{4d_{\phi_p(i)}}\sum_{1 \leq j, k \leq \ell}[\phi_p(i)jk]\frac{x_{\phi_p(i)}}{x_j x_k} -\frac1{4d_i}\sum_{1 \leq j, k \leq \ell}[ijk]\frac{x_i}{x_j x_k} \\
&\quad \overset{\eqref{eq:coef1}}{=} -\frac{b_i}{2}\bigg(\frac{1}{x_i} -\frac{1}{x_{\phi_p(i)}}\bigg) \\
&\quad\qquad -\frac1{2d_i}\sum_{1 \leq q \leq s}[q\phi_p(i)\phi_q(\phi_p(i))]\frac{x_{\phi_q(\phi_p(i))}}{x_{\phi_p(i)} x_q}
+\frac1{2d_i}\sum_{1 \leq q \leq s}[iq\phi_q(i)]\frac{x_{\phi_q(i)}}{x_i x_q} \\
&\quad\qquad +\frac1{4d_i}\sum_{1 \leq q \leq s}[\phi_p(i)q\phi_q(\phi_p(i))]\frac{x_{\phi_p(i)}}{x_q x_{\phi_q(\phi_p(i))}} -\frac1{4d_i}\sum_{1 \leq q \leq s}[iq\phi_q(i)]\frac{x_i}{x_q x_{\phi_q(i)}} \\
&\quad\qquad -\frac1{2d_i}\sum_{1 \leq q \leq s}[\phi_p(i)\phi_q(\phi_p(i))q]\frac{x_q}{x_{\phi_p(i)} x_{\phi_q(\phi_p(i))}}
+\frac1{2d_i}\sum_{1 \leq q \leq s}[i\phi_q(i)q]\frac{x_q}{x_i x_{\phi_q(i)}} \\
&\quad\qquad +\frac1{4d_i}\sum_{1 \leq q \leq s}[\phi_p(i)\phi_q(\phi_p(i))q]\frac{x_{\phi_p(i)}}{x_{\phi_q(\phi_p(i))} x_q} -\frac1{4d_i}\sum_{1 \leq q \leq s}[i\phi_q(i)q]\frac{x_i}{x_{\phi_q(i)} x_q} \\
&\quad\qquad -\frac1{2d_i}\sum_{\substack{s < j, k \leq \ell \\ j \in I_p^0 ,\ k \in I_p^0}}[\phi_p(i)jk]\frac{x_k}{x_{\phi_p(i)} x_j} +\frac1{2d_i}\sum_{\substack{s < j, k \leq \ell \\ j \in I_p^0 ,\ k \in I_p^0}}[ijk]\frac{x_k}{x_i x_j} \\
&\quad\qquad -\frac1{2d_i}\sum_{\substack{s < j, k \leq \ell \\ j \in I_p^+ ,\ k \in I_p^0}}[\phi_p(i)\phi_p(j)k]\frac{x_k}{x_{\phi_p(i)} x_{\phi_p(j)}} +\frac1{2d_i}\sum_{\substack{s < j, k \leq \ell \\ j \in I_p^+ ,\ k \in I_p^0}}[ijk]\frac{x_k}{x_i x_j} \\
&\quad\qquad -\frac1{2d_i}\sum_{\substack{s < j, k \leq \ell \\ j \in I_p^0 ,\ k \in I_p^+}}[\phi_p(i)j\phi_p(k)]\frac{x_{\phi_p(k)}}{x_{\phi_p(i)} x_j} +\frac1{2d_i}\sum_{\substack{s < j, k \leq \ell \\ j \in I_p^0 ,\ k \in I_p^+}}[ijk]\frac{x_k}{x_i x_j} \\
&\quad\qquad -\frac1{2d_i}\sum_{\substack{s < j, k \leq \ell \\ j \in I_p^+ ,\ k \in I_p^+}}[\phi_p(i)\phi_p(j)k]\frac{x_k}{x_{\phi_p(i)} x_{\phi_p(j)}} +\frac1{2d_i}\sum_{\substack{s < j, k \leq \ell \\ j \in I_p^+ ,\ k \in I_p^+}}[ijk]\frac{x_k}{x_i x_j} \\
&\quad\qquad +\frac1{4d_i}\sum_{\substack{s < j, k \leq \ell \\ j \in I_p^0 ,\ k \in I_p^0}}[\phi_p(i)jk]\frac{x_{\phi_p(i)}}{x_j x_k} -\frac1{4d_i}\sum_{\substack{s < j, k \leq \ell \\ j \in I_p^0 ,\ k \in I_p^0}}[ijk]\frac{x_i}{x_j x_k} \\
&\quad\qquad +\frac1{4d_i}\sum_{\substack{s < j, k \leq \ell \\ j \in I_p^+ ,\ k \in I_p^0}}[\phi_p(i)\phi_p(j)k]\frac{x_{\phi_p(i)}}{x_{\phi_p(j)} x_k} -\frac1{4d_i}\sum_{\substack{s < j, k \leq \ell \\ j \in I_p^+ ,\ k \in I_p^0}}[ijk]\frac{x_i}{x_j x_k} \\
&\quad\qquad +\frac1{4d_i}\sum_{\substack{s < j, k \leq \ell \\ j \in I_p^0 ,\ k \in I_p^+}}[\phi_p(i)j\phi_p(k)]\frac{x_{\phi_p(i)}}{x_j x_{\phi_p(k)}} -\frac1{4d_i}\sum_{\substack{s < j, k \leq \ell \\ j \in I_p^0 ,\ k \in I_p^+}}[ijk]\frac{x_i}{x_j x_k} \\
&\quad\qquad +\frac1{4d_i}\sum_{\substack{s < j, k \leq \ell \\ j \in I_p^+ ,\ k \in I_p^+}}[\phi_p(i)\phi_p(j)k]\frac{x_{\phi_p(i)}}{x_{\phi_p(j)} x_k} -\frac1{4d_i}\sum_{\substack{s < j, k \leq \ell \\ j \in I_p^+ ,\ k \in I_p^+}}[ijk]\frac{x_i}{x_j x_k} \,\, .
\end{align*}
Therefore, by using \eqref{eq:phimu}, \eqref{eq:commphip}, \eqref{eq:strconst1}, \eqref{eq:strconst2} and \eqref{eq:strconst3}, we obtain
\begin{align*}
&{\rm ric}_{\phi_{p}(i)}(g) -{\rm ric}_{i}(g) = \\
&\qquad = -\frac{b_i}{2}\bigg(\frac{1}{x_i} -\frac{1}{x_{\phi_p(i)}}\bigg) \\
&\quad\qquad -\frac1{2d_i}\sum_{1 \leq q \leq s}[qi\phi_q(i)]\bigg(\frac{x_{\phi_q(\phi_p(i))}}{x_{\phi_p(i)} x_q} -\frac{x_{\phi_p(i)}}{x_q x_{\phi_q(\phi_p(i))}}\bigg) +\frac1{2d_i}\sum_{1 \leq q \leq s}[qi\phi_q(i)]\bigg(\frac{x_{\phi_q(i)}}{x_i x_q} -\frac{x_i}{x_{\phi_q(i)} x_q}\bigg) \\
&\quad\qquad -\frac1{2d_i}\sum_{1 \leq q \leq s}[qi\phi_q(i)]\bigg(\frac{x_q}{x_{\phi_p(i)} x_{\phi_q(\phi_p(i))}} -\frac{x_q}{x_i x_{\phi_q(i)}}\bigg) \\
&\quad\qquad -\frac1{2d_i}\sum_{\substack{s < j,k \leq \ell \\ j \in I_p^0 ,\ k \in I_p^0}}[ijk]\bigg(\frac{x_k}{x_{\phi_p(i)} x_j} -\frac{x_k}{x_i x_j}\bigg) +\frac1{4d_i}\sum_{\substack{s < j, k \leq \ell \\ j \in I_p^0 ,\ k \in I_p^0}}[ijk]\bigg(\frac{x_{\phi_p(i)}}{x_j x_k} -\frac{x_i}{x_j x_k}\bigg) \\
&\quad\qquad -\frac1{2d_i}\sum_{\substack{s < j, k \leq \ell \\ j \in I_p^+ ,\ k \in I_p^0}}[ijk]\bigg(\frac{x_k}{x_{\phi_p(i)} x_{\phi_p(j)}} -\frac{x_k}{x_i x_j}\bigg) +\frac1{4d_i}\sum_{\substack{s < j, k \leq \ell \\ j \in I_p^+ ,\ k \in I_p^0}}[ijk]\bigg(\frac{x_{\phi_p(i)}}{x_{\phi_p(j)} x_k} -\frac{x_i}{x_j x_k}\bigg) \\
&\quad\qquad -\frac1{2d_i}\sum_{\substack{s < j, k \leq \ell \\ j \in I_p^0 ,\ k \in I_p^+}}[ijk]\bigg(\frac{x_{\phi_p(k)}}{x_{\phi_p(i)} x_j} -\frac{x_k}{x_i x_j}\bigg) +\frac1{4d_i}\sum_{\substack{s < j, k \leq \ell \\ j \in I_p^0 ,\ k \in I_p^+}}[ijk]\bigg(\frac{x_{\phi_p(i)}}{x_j x_{\phi_p(k)}} -\frac{x_i}{x_j x_k}\bigg) \\
&\quad\qquad -\frac1{2d_i}\sum_{\substack{s < j, k \leq \ell \\ j \in I_p^+ ,\ k \in I_p^+}}[ijk]\bigg(\frac{x_k}{x_{\phi_p(i)} x_{\phi_p(j)}} -\frac{x_k}{x_i x_j}\bigg) +\frac1{4d_i}\sum_{\substack{s < j, k \leq \ell \\ j \in I_p^+ ,\ k \in I_p^+}}[ijk]\bigg(\frac{x_{\phi_p(i)}}{x_{\phi_p(j)} x_k} -\frac{x_i}{x_j x_k}\bigg) \\
&\qquad = -\frac{b_i}{2}\bigg(\frac{1}{x_i} -\frac{1}{x_{\phi_p(i)}}\bigg) +\mu_{p,i}^2\bigg(\frac{x_{\phi_p(i)}}{x_i x_p} -\frac{x_i}{x_{\phi_p(i)} x_p}\bigg) \\
&\quad\qquad +\frac12\sum_{\substack{1 \leq q \leq k \\ q \neq p}}\frac{\mu_{q,i}^2}{x_q}\bigg(\frac{x_{\phi_q(i)}}{x_i} -\frac{x_{\phi_q(\phi_p(i))}}{x_{\phi_p(i)}} +\frac{x_{\phi_p(i)}}{x_{\phi_q(\phi_p(i))}} -\frac{x_i}{x_{\phi_q(i)}}\bigg) \\
&\quad\qquad -\frac12\sum_{\substack{1 \leq q \leq s \\ q \neq p}}\mu_{q,i}^2\bigg(\frac{x_q}{x_{\phi_p(i)} x_{\phi_q(\phi_p(i))}} -\frac{x_q}{x_i x_{\phi_q(i)}}\bigg) \\
&\quad\qquad -\frac1{2d_i}\sum_{\substack{s < j, k \leq \ell \\ j \in I_p^0 ,\ k \in I_p^0}}[ijk]\bigg(\frac{x_k}{x_{\phi_p(i)} x_j} -\frac{x_k}{x_i x_j}\bigg) +\frac1{4d_i}\sum_{\substack{s < j, k \leq \ell \\ j \in I_p^0 ,\ k \in I_p^0}}[ijk]\bigg(\frac{x_{\phi_p(i)}}{x_j x_k} -\frac{x_i}{x_j x_k}\bigg) \\
&\quad\qquad -\frac1{2d_i}\sum_{\substack{s < j, k \leq \ell \\ j \in I_p^+ ,\ k \in I_p^0}}[ijk]\bigg(\frac{x_k}{x_{\phi_p(i)} x_{\phi_p(j)}} -\frac{x_k}{x_i x_j}\bigg) +\frac1{4d_i}\sum_{\substack{s < j, k \leq \ell \\ j \in I_p^+ ,\ k \in I_p^0}}[ijk]\bigg(\frac{x_{\phi_p(i)}}{x_{\phi_p(j)} x_k} -\frac{x_i}{x_j x_k}\bigg) \\
&\quad\qquad -\frac1{2d_i}\sum_{\substack{s < j, k \leq \ell \\ j \in I_p^0 ,\ k \in I_p^+}}[ijk]\bigg(\frac{x_{\phi_p(k)}}{x_{\phi_p(i)} x_j} -\frac{x_k}{x_i x_j}\bigg) +\frac1{4d_i}\sum_{\substack{s < j, k \leq \ell \\ j \in I_p^0 ,\ k \in I_p^+}}[ijk]\bigg(\frac{x_{\phi_p(i)}}{x_j x_{\phi_p(k)}} -\frac{x_i}{x_j x_k}\bigg) \\
&\quad\qquad -\frac1{2d_i}\sum_{\substack{s < j, k \leq \ell \\ j \in I_p^+ ,\ k \in I_p^+}}[ijk]\bigg(\frac{x_k}{x_{\phi_p(i)} x_{\phi_p(j)}} -\frac{x_k}{x_i x_j}\bigg) +\frac1{4d_i}\sum_{\substack{s < j, k \leq \ell \\ j \in I_p^+ ,\ k \in I_p^+}}[ijk]\bigg(\frac{x_{\phi_p(i)}}{x_{\phi_p(j)} x_k} -\frac{x_i}{x_j x_k}\bigg) \,\, ,
\end{align*}
which concludes the proof.
\end{proof}

We remark that formulas \eqref{eq:deftheta}, \eqref{eq:normS^2}, \eqref{eq:theta_x} and Proposition \ref{prop:ricciNR} are true even if we replace NR with normalizer-adapted.

\subsection{Vanishing of the submersion tensor for ancient solutions}
\label{subsect:StensRF} \hfill \par

Let $g(\tau)$ be a solution to the $\mathsf{G}$-homogeneous, backward Ricci flow on $M$ that is defined for any positive time $\tau > 0$. We also make the following assumptions. \begin{enumerate}[label=(D), leftmargin=30pt]
\item\label{(D)} The starting metric $g(1)$ can be diagonalized by an NR-decomposition $\varphi = (\mathfrak{m}_1,{\dots},\mathfrak{m}_{\ell})$. Therefore, by Proposition \ref{prop:diagRF}, $g(\tau)$ is diagonal with respect to $\varphi$ for any $\tau > 0$, i.e., 
\begin{equation} \label{eq:backhomRF2}
g(\tau) = x_1(\tau)\,Q|_{\mathfrak{m}_1 \otimes \mathfrak{m}_1} + {\dots} + x_{\ell}(\tau)\,Q|_{\mathfrak{m}_{\ell} \otimes \mathfrak{m}_{\ell}}
\end{equation}
for some smooth functions $x_i: (0,+\infty) \to (0,+\infty)$.
\end{enumerate} \begin{enumerate}[label=(R), leftmargin=30pt]
\item\label{(R)} The collapsing torus $\mathsf{T}$ of $g(\tau)$ is rigid. Therefore, up to reordering, there exists $1 \leq s \leq \ell_0$ such that
\begin{equation} \label{eq:stronglyshrinking} \begin{array}{ll}
\displaystyle{\lim_{\tau \to +\infty}\, \frac{x_p(\tau)}{\tau} = 0} \quad &\text{ for any $1 \leq p \leq s$} \,\, , \\
\displaystyle{\liminf_{\tau \to +\infty}\, \frac{x_i(\tau)}{\tau}  > 0} \quad &\text{ for any $s < i \leq \ell$} \,\, .
\end{array} \end{equation}
\end{enumerate}

Following \eqref{eq:subtens_p}, we can associate to $g(\tau)$ the $1$-parameter families of bilinear forms $S_p(\tau)$ for any $1 \leq p \leq s$ and for any $\tau > 0$. For convenience, in the proofs below, we refer to the indices $1 \leq p \leq s$ as {\it toral} and to the indices $s < i \leq \ell$ as {\it non-toral}. However, despite this terminology, we emphasize that we are not assuming that $\mathfrak{t}$ is maximal in $\mathfrak{m}_0$.

As a consequence of Theorem \ref{thm:limittorus}, the submersion tensors along the directions tangent to the collapsing torus of $g(\tau)$ vanish asymptotically, i.e., the following is true.

\begin{proposition}
Let $g(\tau)$ be as in \eqref{eq:backhomRF}, satisfying \ref{(D)} and \ref{(R)}. Let $S_p(\tau)$ be the $1$-parameter families of bilinear forms defined as in \eqref{eq:subtens_p}. Then,
\begin{equation} \label{eq:Sgoesto0}
\lim_{\tau \to +\infty} |S_p(\tau)|_{g(\tau)} = 0 \quad \text{for any $1 \leq p \leq s$} .
\end{equation}
\end{proposition}

\begin{proof}
From \eqref{eq:AGAG2} and \eqref{eq:stronglyshrinking}, it follows that
\begin{equation} \label{eq:AGAG3}
\lim_{\tau \to +\infty} \frac{x_{\phi_p(i)}(\tau)}{x_i(\tau)} = 1 \quad \text{ for any $1 \leq p \leq s$, for any $i \in I_p$. }
\end{equation}
Therefore, \eqref{eq:Sgoesto0} follows by \eqref{eq:normS^2}, \eqref{eq:theta_x} and \eqref{eq:AGAG3}.
\end{proof}

In the following, we will prove that these tensors actually vanish identically, i.e., $S_p(\tau) \equiv 0$ for any $1 \leq p \leq s$, from which Theorem \ref{thm:main-extra.symm} follows. For any $1 \leq q \leq s$ and for any $j \in I_q$, let $\theta_{q,j}(\tau)$ be the functions defined in \eqref{eq:deftheta}. By \eqref{eq:normS^2}, it follows that \eqref{eq:Sgoesto0} is equivalent to 
\begin{equation} \label{eq:thetato0}
\lim_{\tau \to +\infty} \theta_{q,j}(\tau) = 0 \quad \text {for any $1 \leq q \leq s$ and for any $j \in I_q$} \, .
\end{equation}
Moreover, from \eqref{eq:theta_x} and the Ricci flow equation, we obtain the following.

\begin{lemma}
For any $1 \leq q \leq s$ and for any $j \in I_q$, we have
\begin{equation} \label{eq:theta'}
\frac{{\rm d}}{{\rm d} \tau}\theta_{q,j}(\tau) = \frac12\bigg(\frac{x_{\phi_q(j)}(\tau)}{x_j(\tau)} -\frac{x_j(\tau)}{x_{\phi_q(j)}(\tau)}\bigg)\big({\rm ric}_{\phi_q(j)}(g(\tau)) -{\rm ric}_{j}(g(\tau))\big) \quad \text{ for any $\tau > 0$ .}
\end{equation}
\end{lemma}

Fix a toral index $1 \leq p \leq s$. For any $\tau > 0$, let $\iota = \iota(p,\tau) \in I_p$ be the index that satisfies the following property:
\begin{equation} \label{eq:special-p-i}
\theta_{p,\iota(p,\tau)}(\tau) = \max\big\{\theta_{p,j}(\tau) : j \in I_p \big\} \,\, .
\end{equation}
In order to prove Theorem \ref{thm:main-extra.symm}, we need to estimate the derivative
\begin{equation} \label{eq:dertoest}
\left.\frac{{\rm d}}{{\rm d} \tau}\theta_{p,\iota(p,\tau^*)}(\tau)\right|_{\tau = \tau^*}
\end{equation}
for any $\tau^* > 0$ large enough. To achieve this, we need some preparatory work. For the sake of notation, we set
\begin{gather}
b_{\mathsf{G}/\mathsf{H}} \coloneqq {\rm Tr}_Q(-\mathcal{B}_{\mathfrak{g}}) = \sum_{1 \leq j \leq \ell}d_jb_j \,\, , \label{eq:b} \\
\alpha_q \coloneqq \max\Big\{ \max\{|\mu_{q,j}| : j \in I_q\}, \big(\min\{|\mu_{q,j}| : j \in I_q\}\big)^{-1} \Big\} \quad \text{for any $1 \leq q \leq s$} \,\, , \label{eq:alpha_p} \\
\varepsilon_{q,j}(\tau) \coloneqq \bigg|\frac{x_{\phi_q(j)}(\tau)}{x_j(\tau)} -1\bigg| \quad \text{for any $1 \leq q \leq s$, for any $j \in I_q$} \,\, , \label{def:epsilon} \\
\underline{x}(\tau) \coloneqq \min\{x_j(\tau) : s < j \leq \ell \} \,\, . \label{eq:xmin}
\end{gather}
Here, the coefficients $d_j$ and $b_j$ were defined in Section \ref{sect:prelhomsp}, while the coefficients $\mu_{q,j}$ were defined in Section \ref{subsect:norm-diag}. We also remark that, by assumption \ref{(R)}, the quantity ${\rm scal}(g(\tau))\underline{x}(\tau)$ is bounded away from zero (see Section \ref{subsect:StensRF}).

By \eqref{eq:stronglyshrinking} and \eqref{eq:AGAG3}, up to scaling, we can assume that the following inequalities hold for any $\tau \geq 1$:
\begin{equation} \label{eq:time1}
1-\frac{1}{10} < \frac{x_{\phi_q(j)}(\tau)}{x_j(\tau)},\frac{x_j(\tau)}{x_{\phi_q(j)}(\tau)} < 1+\frac{1}{10} \quad \text{for any $1 \leq q \leq s$ and for any $j \in I_q$} \, ,
\end{equation}
\begin{equation} \label{eq:time2}
x_q(\tau) < \frac{3\alpha_q}{b_{\mathsf{G}/\mathsf{H}}}\underline{x}(\tau) \quad \text{ for any $1 \leq q \leq s$} \, .
\end{equation}
We start now by proving the following two lemmas. 
 
\begin{lemma}
Fix three non-toral indices $s < j, k, r \leq \ell$ and let $\tilde{j} \in \{j,\phi_p(j)\}$, $\tilde{k} \in \{k,\phi_p(k)\}$, $\tilde{r} \in \{r,\phi_p(r)\}$. Then, it follows that
\begin{equation} \label{eq:estimatealmostsub}
\bigg|\frac{x_{\tilde{j}}(\tau)}{x_{\tilde{k}}(\tau) x_{\tilde{r}}(\tau)} -\frac{x_j(\tau)}{x_k(\tau) x_r(\tau)}\bigg| \leq 6\frac{x_j(\tau)}{x_k(\tau) x_r(\tau)}\big(\varepsilon_{p,j}(\tau) +\varepsilon_{p,k}(\tau) +\varepsilon_{p,r}(\tau)\big) \quad \text{ for any $\tau \geq 1$ .}
\end{equation}
\end{lemma}

\begin{proof}
Notice that, for any $\tau \geq 1$,
$$
x_j(\tau)\big(1 -\varepsilon_{p,j}(\tau)\big) \leq x_{\tilde{j}}(\tau) \leq x_j(\tau)\big(1 +\varepsilon_{p,j}(\tau)\big)
$$
and so
$$
\frac{x_{\tilde{j}}(\tau)}{x_{\tilde{k}}(\tau) x_{\tilde{r}}(\tau)} \leq \frac{x_j(\tau)}{x_k(\tau) x_r(\tau)} \frac{1 +\varepsilon_{p,j}(\tau)}{\big(1 -\varepsilon_{p,k}(\tau)\big)\big(1 -\varepsilon_{p,r}(\tau)\big)} \,\, .
$$
By \eqref{eq:time1}, it follows that
$$
0 \leq \varepsilon_{p,j}(\tau) \leq \frac{1}{10} \quad \text{for any $j \in I_p$}
$$
and so we have
$$\begin{aligned}
\frac{1 +\varepsilon_{p,j}(\tau)}{(1 -\varepsilon_{p,k}(\tau))(1 -\varepsilon_{p,r}(\tau))} &\leq \big(1 +\varepsilon_{p,j}(\tau)\big)\big(1 +2\varepsilon_{p,r}(\tau)\big)\big(1 +2\varepsilon_{p,k}(\tau)\big) \\
&\leq \big(1 +3\varepsilon_{p,j}(\tau) +3\varepsilon_{p,k}(\tau)\big)\big(1 +2\varepsilon_{p,r}(\tau)\big) \\
&\leq 1 +6\big(\varepsilon_{p,j}(\tau) +\varepsilon_{p,k}(\tau) +\varepsilon_{p,r}(\tau)\big)
\end{aligned}$$
and then \eqref{eq:estimatealmostsub} follows.
\end{proof}

\begin{lemma}
For any non-toral indices $s < j, k, r \leq \ell$, the following inequality holds:
\begin{equation} \label{eq:nontoralterms}
[jkr]\frac{x_j(\tau)}{x_k(\tau) x_r(\tau)} \leq \frac{2b_{\mathsf{G}/\mathsf{H}}}{\underline{x}(\tau)} \quad \text{for any $\tau \geq 1$} \, .
\end{equation}
\end{lemma}

\begin{proof}
As we mentioned in Section \ref{subsect:homRF}, it follows by hypothesis that ${\rm scal}(g(\tau))>0$ for any $\tau > 0$. Moreover, by \cite[Lemma 1.5]{WaZ86} and the fact that $\mathfrak{t}$ is an abelian subalgebra of $\mathfrak{m}_0$, we have
\begin{equation} \label{eq:WZ}
d_qb_q = \sum_{s < j, k \leq \ell}[qjk] \quad \text{for any $1 \leq q \leq s$} \, .
\end{equation}
Therefore, by \eqref{eq:scal} and \eqref{eq:WZ}, we have
$$\begin{aligned}
{\rm scal}(g(\tau)) &= \frac12\sum_{1 \leq j \leq \ell}\frac{d_j b_j}{x_j(\tau)} -\frac14\sum_{1 \leq j, k, r \leq \ell}[jkr] \frac{x_j(\tau)}{x_k(\tau) x_r(\tau)} \\
&= \frac12\sum_{1 \leq q \leq s}\frac{d_qb_q}{x_q(\tau)} -\frac12\sum_{\substack{1 \leq q \leq s \\ s < j, k \leq \ell}}[qjk]\frac{x_k(\tau)}{x_q(\tau) x_j(\tau)} -\frac14\sum_{\substack{1 \leq q \leq s \\ s < j, k \leq \ell}}[qjk]\frac{x_q(\tau)}{x_j(\tau) x_k(\tau)} \\
&\quad +\frac12\sum_{s < j \leq \ell}\frac{d_jb_j}{x_j(\tau)} -\frac14\sum_{s < j, k, r \leq \ell}[jkr]\frac{x_j(\tau)}{x_k(\tau) x_r(\tau)} \\
&= \frac12\sum_{\substack{1 \leq q \leq s \\ s < j, k \leq \ell}}[qjk]\bigg(\frac{1}{x_q(\tau)} -\frac{x_k(\tau)}{x_q(\tau) x_j(\tau)}\bigg) -\frac14\sum_{\substack{1 \leq q \leq s \\ s < j, k \leq \ell}}[qjk]\frac{x_q(\tau)}{x_j(\tau) x_k(\tau)} \\
&\quad +\frac12\sum_{s < j \leq \ell}\frac{d_jb_j}{x_j(\tau)} -\frac14\sum_{s < j, k, r \leq \ell}[jkr]\frac{x_j(\tau)}{x_k(\tau) x_r(\tau)} \\
&= -\frac12\sum_{\substack{1 \leq q \leq s \\ j \in I_q}}[qj\phi_q(j)]\frac{(x_{\phi_q(j)}(\tau) -x_j(\tau))^2}{x_q(\tau) x_j(\tau) x_{\phi_q(j)}(\tau)} -\frac14\sum_{\substack{1 \leq q \leq s \\ j \in I^+_q}}[qj\phi_q(j)]\frac{x_q(\tau)}{x_j(\tau) x_{\phi_q(j)}(\tau)} \\
&\quad +\frac12\sum_{s < j \leq \ell}\frac{d_jb_j}{x_j(\tau)} -\frac14\sum_{s < j, k, r \leq \ell}[jkr]\frac{x_j(\tau)}{x_k(\tau) x_r(\tau)} \\
&\leq \frac12\sum_{s < j \leq \ell}\frac{d_jb_j}{x_j(\tau)} -\frac14\sum_{s < j, k, r \leq \ell}[jkr]\frac{x_j(\tau)}{x_k(\tau) x_r(\tau)} \,\, ,
\end{aligned}$$
which in turn implies that
$$
\sum_{s < j, k, r \leq \ell}[jkr]\frac{x_j(\tau)}{x_k(\tau) x_r(\tau)} < 2\sum_{s < j \leq \ell}\frac{d_jb_j}{x_j(\tau)} \leq \frac{2b_{\mathsf{G}/\mathsf{H}}}{\underline{x}(\tau)} \,\, ,
$$
where $b_{\mathsf{G}/\mathsf{H}}$ and $\underline{x}(\tau)$ were defined in \eqref{eq:b} and \eqref{eq:xmin}, respectively. Therefore, \eqref{eq:nontoralterms} follows.
\end{proof}

To estimate the terms in \eqref{eq:dertoest} without toral indices, we prove the following intermediate result.

\begin{proposition}
Fix three non-toral indices $s < j, k, r \leq \ell$ and let $\tilde{j} \in \{j,\phi_p(j)\}$, $\tilde{k} \in \{k,\phi_p(k)\}$, $\tilde{r} \in \{r,\phi_p(r)\}$. Then, there exists a constant $c(m)>0$, that depends only on the dimension $m$, such that
\begin{equation} \label{eq:almsubnorm}
[jkr]\bigg|\frac{x_{\tilde{j}}(\tau)}{x_{\tilde{k}}(\tau) x_{\tilde{r}}(\tau)} -\frac{x_j(\tau)}{x_k(\tau) x_r(\tau)}\bigg| \leq 12 \,b_{\mathsf{G}/\mathsf{H}}\, \alpha_p \, c(m) \frac{|S_p(\tau)|_{g(\tau)}}{\underline{x}(\tau)} \quad \text{for any $\tau \geq 1$} \, .
\end{equation}
\end{proposition}

\begin{proof}
Let us observe that, from \eqref{eq:theta_x} and \eqref{def:epsilon}, it follows that for any toral index $1 \leq q \leq s$ and for any $j \in I_q$
\begin{equation} \label{def:thetaepsilon}
\theta_{q,j}(\tau) = \frac{1}{4} \frac{x_j(\tau)}{x_{\phi_q(j)}(\tau)}\,\varepsilon_{q,j}(\tau)^2 \,\, .
\end{equation}
By \eqref{eq:phimu}, \eqref{eq:normS^2}, \eqref{def:epsilon}, \eqref{eq:time1} and \eqref{def:thetaepsilon}, there exists $c(m)>1$ such that
\begin{equation} \label{eq:normSepsilon}
\frac1{\alpha_p c(m)}\sum_{j \in I_p} \varepsilon_{p,j}(\tau) \leq |S_p(\tau)|_{g(\tau)} \leq \alpha_p c(m)\sum_{j \in I_p} \varepsilon_{p,j}(\tau) \quad \text{ for any $\tau \geq 1$} \,\, .
\end{equation}
Therefore, \eqref{eq:almsubnorm} follows from \eqref{eq:estimatealmostsub}, \eqref{eq:nontoralterms} and \eqref{eq:normSepsilon}.
\end{proof}

To estimate the terms in \eqref{eq:dertoest} in which toral indices different from $p$ appear, we prove the following intermediate result.

\begin{lemma}
Given any $\tau^* \geq 1$, the following inequalities hold for any $1 \leq q \leq s$ and for $\iota = \iota(p,\tau^*)$:
\begin{gather}
\bigg(\frac{x_{\phi_p(\iota)}(\tau^*)}{x_{\iota}(\tau^*)} -\frac{x_{\iota}(\tau^*)}{x_{\phi_p(\iota)}(\tau^*)}\bigg)\bigg(\frac{x_{\phi_q(\iota)}(\tau^*)}{x_{\iota}(\tau^*)} -\frac{x_{\phi_q(\phi_p(\iota))}(\tau^*)}{x_{\phi_p(\iota)}(\tau^*)}\bigg) \geq 0 \,\, , \label{eq:toralsign1} \\
\bigg(\frac{x_{\phi_p(\iota)}(\tau^*)}{x_{\iota}(\tau^*)} -\frac{x_{\iota}(\tau^*)}{x_{\phi_p(\iota)}(\tau^*)}\bigg)\bigg(\frac{x_{\phi_p(\iota)}(\tau^*)}{x_{\phi_q(\phi_p(\iota))}(\tau^*)} -\frac{x_{\iota}(\tau^*)}{x_{\phi_q(\iota)}(\tau^*)}\bigg) \geq 0 \,\, , \label{eq:toralsign2} \\
\bigg(\frac{x_{\phi_p(\iota)}(\tau^*)}{x_{\iota}(\tau^*)} -\frac{x_{\iota}(\tau^*)}{x_{\phi_p(\iota)}(\tau^*)}\bigg)\bigg(\frac{1}{x_{\iota}(\tau^*) x_{\phi_q(\iota)}(\tau^*)} -\frac{1}{x_{\phi_p(\iota)}(\tau^*) x_{\phi_q(\phi_p(\iota))}(\tau^*)}\bigg) \geq 0 \,\, . \label{eq:toralsign3}
\end{gather}
\end{lemma}

\begin{proof}
If $x_{\phi_{p}(\iota)}(\tau^*) = x_{\iota}(\tau^*)$, the statement is trivially true. Assume $x_{\phi_{p}(\iota)}(\tau^*) > x_{\iota}(\tau^*)$ and notice that \eqref{eq:special-p-i} implies that
\begin{equation} \label{eq:max-p-i}
\frac{x_{\phi_p{(\iota)}}(\tau^*)}{x_i(\tau^*)} \geq \max_{\substack{1 \leq q \leq s \\ 1 \leq j \leq \ell}}\bigg\{\frac{x_{\phi_q(j)}(t^*)}{x_j(\tau^*)} \, , \,\, \frac{x_j(\tau^*)}{x_{\phi_q(j)}(\tau^*)}\bigg\} \,\, . \end{equation}
Indeed, observe that
$$
\theta_{p,\iota}(\tau^*) = f\left(\frac{x_{\phi_p(\iota)}(\tau^*)}{x_{\iota}(\tau^*)}\right) \,\, , \quad \text{ where $f(y) \coloneqq \frac{1}{4}\left( y + \frac{1}{y} - 2 \right)$} \,\, ,
$$
so \eqref{eq:max-p-i} follows from the fact that $f$ is a non-decreasing function of $y \in \mathbb{R}$ for $y \ge 1$.

Fix $1 \leq q \leq s$. Then, by \eqref{eq:commphip} and \eqref{eq:max-p-i}, we get
$$\begin{aligned}
0 &\leq \frac{x_{\phi_p{(\iota)}}(\tau^*)}{x_{\iota}(\tau^*)} -\frac{x_{\phi_p(\phi_q(\iota))}(\tau^*)}{x_{\phi_q(\iota)}(\tau^*)} = \frac{x_{\phi_p{(\iota)}}(\tau^*)}{x_{\phi_q(\iota)}(\tau^*)}\bigg(\frac{x_{\phi_q(\iota)}(\tau^*)}{x_{\iota}(\tau^*)} -\frac{x_{\phi_q(\phi_p(\iota))}(\tau^*)}{x_{\phi_p{(\iota)}}(\tau^*)}\bigg) \,\, , \\
0 &\leq \frac{x_{\phi_p{(\iota)}}(\tau^*)}{x_{\iota}(\tau^*)} -\frac{x_{\phi_p(\phi_q(\iota))}(\tau^*)}{x_{\phi_q(\iota)}(\tau^*)} = \frac{x_{\phi_p(\phi_q(\iota))}(\tau^*)}{x_{\iota}(\tau^*)}\bigg(\frac{x_{\phi_p{(\iota)}}(\tau^*)}{x_{\phi_q(\phi_p(\iota))}(\tau^*)} -\frac{x_{\iota}(\tau^*)}{x_{\phi_q(\iota)}(\tau^*)}\bigg) \,\, , \\
0 &\leq \frac{x_{\phi_p{(\iota)}}(\tau^*)}{x_{\iota}(\tau^*)} -\frac{x_{\phi_q(\iota)}(\tau^*)}{x_{\phi_p(\phi_q(\iota))}(\tau^*)} = x_{\phi_p(\iota)}(\tau^*) x_{\phi_q(\iota)}(\tau^*) \bigg(\frac{1}{x_{\iota}(\tau^*) x_{\phi_q(\iota)}(\tau^*)} -\frac{1}{x_{\phi_p(\iota)}(\tau^*) x_{\phi_q(\phi_p(\iota))}(\tau^*)}\bigg) \,\, .
\end{aligned}$$
Therefore, \eqref{eq:toralsign1}, \eqref{eq:toralsign2} and \eqref{eq:toralsign3} are proved. If $x_{\phi_{p}(\iota)}(\tau^*) < x_{\iota}(\tau^*)$, the proof goes the same way.
\end{proof}

We are ready to prove the main estimate that allows us to obtain Theorem \ref{thm:main-extra.symm}, that is, the following monotonicity formula.

\begin{theorem} \label{thm:mainestimate}
Let $g(\tau)$ be as in \eqref{eq:backhomRF}, satisfying \ref{(D)} and \ref{(R)}. Let also $S_p(\tau)$ be the $1$-parameter families of bilinear forms defined as in \eqref{eq:subtens_p} and $\theta_{q,j}(\tau)$ the functions defined in \eqref{eq:deftheta}. Then, there exist a constant $\delta(m,p)>0$, that depends only on Lie algebraic data, and a time $T_p>0$ such that
\begin{equation} \label{eq:mainestimate}
\left.\frac{{\rm d}}{{\rm d} \tau}\theta_{p,{\iota}(p,\tau^*)}(\tau)\right|_{\tau = \tau^*} \geq \delta(m,p) \frac{\big(|S_p(\tau^*)|_{g(\tau^*)}\big)^2}{x_p(\tau^*)} \quad \text{for any $\tau^* \geq T_p$} \, .
\end{equation}
\end{theorem}

\begin{proof}
By \eqref{eq:diffricci} and \eqref{eq:theta'}, it follows that
\begin{align*}
& \left.\frac{{\rm d}}{{\rm d} \tau}\theta_{p,{\iota}(p,\tau^*)}(\tau)\right|_{\tau = \tau^*} = \\
&= \frac12 \bigg(\frac{x_{\phi_p(\iota)}(\tau^*)}{x_{\iota}(\tau^*)} -\frac{x_{\iota}(\tau^*)}{x_{\phi_p(\iota)}(\tau^*)}\bigg) \Bigg\{\mu_{p,\iota}^2\frac{x_{\phi_p(\iota)}(\tau^*) -x_{\iota}(\tau^*)}{x_p(\tau^*) x_{\iota}(\tau^*) x_{\phi_p(\iota)}(\tau^*)}\bigg(x_{\phi_p(\iota)}(\tau^*) +x_{\iota}(\tau^*) -\frac{b_{\iota}}{2\mu_{p,\iota}^2}x_p(\tau^*)\bigg) \\
&\qquad +\frac12\sum_{\substack{1 \leq q \leq s \\ q \neq p}}\mu_{q,\iota}^2x_q(\tau^*) \bigg(\frac{1}{x_{\iota}(\tau^*) x_{\phi_q(\iota)}(\tau^*)} -\frac{1}{x_{\phi_p(\iota)}(\tau^*) x_{\phi_q(\phi_p(\iota))}(\tau^*)}\bigg) \\
&\qquad +\frac12\sum_{\substack{1 \leq q \leq s \\ q \neq p}}\frac{\mu_{q,\iota}^2}{x_q(\tau^*)}\bigg(\frac{x_{\phi_q(\iota)}(\tau^*)}{x_{\iota}(\tau^*)} -\frac{x_{\phi_q(\phi_p(\iota))}(\tau^*)}{x_{\phi_p(\iota)}(\tau^*)} +\frac{x_{\phi_p(\iota)}(\tau^*)}{x_{\phi_q(\phi_p(\iota))}(\tau^*)} -\frac{x_{\iota}(\tau^*)}{x_{\phi_q(\iota)}(\tau^*)}\bigg) \\
&\qquad -\frac1{2d_{\iota}}\sum_{\substack{s < j, k \leq \ell \\ j \in I_p^0 ,\ k \in I_p^0}}[\iota jk]\bigg(\frac{x_k(\tau^*)}{x_{\phi_p(\iota)}(\tau^*) x_j(\tau^*)} -\frac{x_k(\tau^*)}{x_{\iota}(\tau^*) x_j(\tau^*)}\bigg) \\
&\qquad +\frac1{4d_{\iota}}\sum_{\substack{s < j, k \leq \ell \\ j \in I_p^0 ,\ k \in I_p^0}}[\iota jk]\bigg(\frac{x_{\phi_p(\iota)}(\tau^*)}{x_j(\tau^*) x_k(\tau^*)} -\frac{x_{\iota}(\tau^*)}{x_j(\tau^*) x_k(\tau^*)}\bigg) \\
&\qquad -\frac1{2d_{\iota}}\sum_{\substack{s < j, k \leq \ell \\ j \in I_p^0 ,\ k \in I_p^+}}[\iota jk]\bigg(\frac{x_{\phi_p(k)}(\tau^*)}{x_{\phi_p(\iota)}(\tau^*) x_j(\tau^*)} -\frac{x_k(\tau^*)}{x_{\iota}(\tau^*) x_j(\tau^*)}\bigg) \\
&\qquad +\frac1{4d_{\iota}}\sum_{\substack{s < j, k \leq \ell \\ j \in I_p^0 ,\ k \in I_p^+}}[\iota jk]\bigg(\frac{x_{\phi_p(\iota)}(\tau^*)}{x_j(\tau^*) x_{\phi_p(k)}(\tau^*)} -\frac{x_{\iota}(\tau^*)}{x_j(\tau^*) x_k(\tau^*)}\bigg) \\
&\qquad -\frac1{2d_{\iota}}\sum_{\substack{s < j, k \leq \ell \\ j \in I_p^+}}[\iota jk]\bigg(\frac{x_k(\tau^*)}{x_{\phi_p(\iota)}(\tau^*) x_{\phi_p(j)}(\tau^*)} -\frac{x_k(\tau^*)}{x_{\iota}(\tau^*) x_j(\tau^*)}\bigg) \\
&\qquad +\frac1{4d_{\iota}}\sum_{\substack{s < j, k \leq \ell \\ j \in I_p^+}}[\iota jk]\bigg(\frac{x_{\phi_p(\iota)}(\tau^*)}{x_{\phi_p(j)}(\tau^*) x_k(\tau^*)} -\frac{x_{\iota}(\tau^*)}{x_j(\tau^*) x_k(\tau^*)}\bigg) \Bigg\} \,\, .
\end{align*}
Then, by applying \eqref{eq:toralsign1}, \eqref{eq:toralsign2} and \eqref{eq:toralsign3}, we get
\begin{align*}
&\left.\frac{{\rm d}}{{\rm d} \tau}\theta_{p,{\iota}(p,\tau^*)}(\tau)\right|_{\tau = \tau^*} \geq \\
&\geq \frac12 \mu_{p,\iota}^2 \frac{\big(x_{\phi_p(\iota)}(\tau^*) -x_{\iota}(\tau^*)\big)^2\big(x_{\phi_p(\iota)}(\tau^*) +x_{\iota}(\tau^*)\big)}{x_{\iota}(\tau^*)^2 x_{\phi_p(\iota)}(\tau^*)^2 x_p(\tau^*)}\bigg(x_{\phi_p(\iota)}(\tau^*) +x_{\iota}(\tau^*) -\frac{b_{\iota}}{2\mu_{p,{\iota}}^2}x_p(\tau^*)\bigg) \\
&\qquad -\frac1{4d_{\iota}}\bigg|\frac{x_{\phi_p(\iota)}(\tau^*)}{x_{\iota}(\tau^*)} -\frac{x_{\iota}(\tau^*)}{x_{\phi_p(\iota)}(\tau^*)}\bigg|\Bigg\{\sum_{\substack{s < j, k \leq \ell \\ j \in I_p^0 ,\ k \in I_p^0}}[\iota jk]\bigg|\frac{x_k(\tau^*)}{x_{\phi_p(\iota)}(\tau^*) x_j(\tau^*)} -\frac{x_k(\tau^*)}{x_{\iota}(\tau^*) x_j(\tau^*)}\bigg| \\
&\qquad\qquad +\sum_{\substack{s < j, k \leq \ell \\ j \in I_p^0 ,\ k \in I_p^0}}[\iota jk]\bigg|\frac{x_{\phi_p(\iota)}(\tau^*)}{x_j(\tau^*) x_k(\tau^*)} -\frac{x_{\iota}(\tau^*)}{x_j(\tau^*) x_k(\tau^*)}\bigg| \\
&\qquad\qquad +\sum_{\substack{s < j, k \leq \ell \\ j \in I_p^0 ,\ k \in I_p^+}}[\iota jk]\bigg|\frac{x_{\phi_p(k)}(\tau^*)}{x_{\phi_p(\iota)}(\tau^*) x_j(\tau^*)} -\frac{x_k(\tau^*)}{x_{\iota}(\tau^*) x_j(\tau^*)}\bigg| \\
&\qquad\qquad +\sum_{\substack{s < j, k \leq \ell \\ j \in I_p^0 ,\ k \in I_p^+}}[\iota jk]\bigg|\frac{x_{\phi_p(\iota)}(\tau^*)}{x_j(\tau^*) x_{\phi_p(k)}(\tau^*)} -\frac{x_{\iota}(\tau^*)}{x_j(\tau^*) x_k(\tau^*)}\bigg| \\
&\qquad\qquad +\sum_{\substack{s < j, k \leq \ell \\ j \in I_p^+}}[\iota jk]\bigg|\frac{x_k(\tau^*)}{x_{\phi_p(\iota)}(\tau^*) x_{\phi_p(j)}(\tau^*)} -\frac{x_k(\tau^*)}{x_{\iota}(\tau^*) x_j(\tau^*)}\bigg| \\
&\qquad\qquad +\sum_{\substack{s < j, k \leq \ell \\ j \in I_p^+}}[\iota jk]\bigg|\frac{x_{\phi_p(\iota)}(\tau^*)}{x_{\phi_p(j)}(\tau^*) x_k(\tau^*)} -\frac{x_{\iota}(\tau^*)}{x_j(\tau^*) x_k(\tau^*)}\bigg| \Bigg\} \,\, .
\end{align*}
By \eqref{eq:theta_x}, \eqref{eq:alpha_p} and \eqref{eq:almsubnorm}, we get
\begin{multline*}
\left.\frac{{\rm d}}{{\rm d} \tau}\theta_{p,{\iota}(p,\tau^*)}(\tau)\right|_{\tau = \tau^*} \geq \frac{2}{\alpha_p} \frac{\theta_{p,\iota}(\tau^*)}{x_p(\tau^*)} \bigg(\frac{x_{\phi_p(\iota)}(\tau^*) +x_{\iota}(\tau^*)}{x_{\phi_p(\iota)}(\tau^*)}\bigg) \Bigg(\frac{x_{\phi_p(i)}(\tau^*) +x_{\iota}(\tau^*) -\frac{b_i}{2\mu_{p,\iota}^2}x_p(\tau^*)}{x_{\iota}(\tau^*)}\Bigg) \\
-\frac1{2d_{\iota}}\bigg|\frac{x_{\phi_p(\iota)}(\tau^*)}{x_{\iota}(\tau^*)} -\frac{x_{\iota}(\tau^*)}{x_{\phi_p(\iota)}(\tau^*)}\bigg|\Bigg\{ 6(\ell-k)^2\cdot12\,b_{\mathsf{G}/\mathsf{H}}\, \alpha_p \, c(m) \frac{|S_p(\tau^*)|_{g(\tau^*)}}{\underline{x}(\tau^*)} \Bigg\} \,\, .
\end{multline*}
Notice that, by \eqref{eq:time1}, we obtain 
\begin{equation} \label{eq:obvious}
\frac53 < \frac{x_{\phi_p(\iota)}(\tau^*) +x_{\iota}(\tau^*)}{x_{\phi_p(\iota)}(\tau^*)} < \frac{20}{9}
\end{equation}
and that, by recalling definitions \eqref{eq:b}, \eqref{eq:alpha_p}, \eqref{eq:xmin} and by applying \eqref{eq:time1}, \eqref{eq:time2}, we get
\begin{equation} \label{eq:positiveaddend}
\frac{x_{\phi_p(\iota)}(\tau^*) +x_{\iota}(\tau^*) -\frac{b_{\iota}}{2\mu_{p,\iota}^2}x_p(\tau^*)}{x_{\iota}(\tau^*)} \geq \frac32 -\frac{b_{\mathsf{G}/\mathsf{H}}\alpha_p^2}{2}\frac{x_p(\tau^*)}{\underline{x}(\tau^*)} > 0 \,\, .
\end{equation}
Therefore, by \eqref{def:epsilon}, \eqref{eq:obvious} and \eqref{eq:positiveaddend}, we get
$$
\left.\frac{{\rm d}}{{\rm d} \tau}\theta_{p,{\iota}(p,\tau^*)}(\tau)\right|_{\tau = \tau^*} \geq \frac{5}{\alpha_p} \frac{\theta_{p,i}(\tau^*)}{x_p(\tau^*)} \left(1 -\frac{b_{\mathsf{G}/\mathsf{H}}\alpha_p^2}{3}\frac{x_p(\tau^*)}{\underline{x}(\tau^*)}\right) -\frac{80 (\ell-k)^2\,b_{\mathsf{G}/\mathsf{H}}\, \alpha_p \, c(m)}{d_{\iota}} \varepsilon_{p,\iota}(\tau^*) \frac{|S_p(\tau^*)|_{g(\tau^*)}}{\underline{x}(\tau^*)} \,\, .
$$
By applying \eqref{eq:phimu}, \eqref{eq:normS^2}, \eqref{eq:alpha_p}, \eqref{eq:special-p-i} and \eqref{eq:normSepsilon}, we get
\begin{equation} \label{eq:final}
\left.\frac{{\rm d}}{{\rm d} \tau}\theta_{p,{\iota}(p,\tau^*)}(\tau)\right|_{\tau = \tau^*} \geq \frac{5}{d_{\iota} m \alpha_p^3} \Bigg(1 -\frac{b_{\mathsf{G}/\mathsf{H}}\alpha_p^2\big(1 +48 (\ell-k)^2\,m\alpha_p^3 \, c(m)^2\big)}{3} \frac{x_p(\tau^*)}{\underline{x}(\tau^*)} \Bigg) \frac{\big(|S_p(\tau^*)|_{g(\tau^*)}\big)^2}{x_p(\tau^*)} \,\, .
\end{equation}
By \eqref{eq:stronglyshrinking} and \eqref{eq:xmin}, we can chose now a time $T_p>0$ such that
$$
1 -\frac{b_{\mathsf{G}/\mathsf{H}}\alpha_p^2\big(1 +48 (\ell-k)^2\,m\alpha_p^3 \, c(m)^2\big)}{3} \frac{x_p(\tau^*)}{\underline{x}(\tau^*)} \geq \frac15 \quad \text{for any $\tau \geq T_p$}
$$
and so, if we set
$$
\delta(m,p) \coloneqq \frac{1}{m^2 \alpha_p^3} \,\, ,
$$
\eqref{eq:mainestimate} immediately follows from \eqref{eq:final}.
\end{proof}

We are now ready to complete the proof of Theorem \ref{thm:main-extra.symm}.

\begin{proof}[Proof of Theorem \ref{thm:main-extra.symm}]
Let $g(t)$ be a collapsed ancient solution to the homogeneous Ricci flow on a compact manifold. By abuse of notation, we set $\tau = -t$ and denote by $g(\tau)$ the corresponding backward solution. Assume that $g(1)$, and hence $g(\tau)$ for any $\tau > 0$, is diagonal with respect to an NR-decomposition $\varphi$. Assume the collapsing torus $\mathsf{T}$ of $g(\tau)$ is unique. Then we observe that, since the solution is $\varphi$-diagonal, $\mathsf{T}$ is rigid (see Remark \ref{rem:unique-rigid}). Therefore, both the assumptions \ref{(D)} and \ref{(R)} are satisfied.

Now, fix a toral index $1 \leq p \leq s$ and recall that, by \eqref{eq:normS^2} and \eqref{eq:Sgoesto0},
\begin{equation} \label{eq:finalfinal}
\big(|S_p(\tau)|_{g(\tau)}\big)^2 = \sum_{j \in I_p} [pj\phi_p(j)]\, \theta_{p,j}(\tau) \,\rightarrow\, 0 \quad \text{as $\tau \to +\infty$} \,\, .
\end{equation}
If there exists $\tau^* \geq 1$ such that $|S_p(\tau^*)|_{g(\tau^*)}=0$, then $|S_p(\tau)|_{g(\tau)}=0$ for any $\tau > 0$ since isometries are preserved under the Ricci flow. Assume then by contradiction that $|S_p(\tau)|_{g(\tau)}>0$ for any $\tau > 0$. Then, by \eqref{eq:special-p-i} and \eqref{eq:mainestimate}, it follows that the largest summand $\theta_{p,\iota(p,\tau)}(\tau)$ is strictly increasing as $\tau \to +\infty$, which contradicts \eqref{eq:finalfinal}.
\end{proof}

\section{Examples}
\label{sect:example-sRd} \setcounter{equation} 0

In this section, by adapting the techniques of \cite{PedSb22}, we obtain a general existence result for diagonal, collapsed ancient solutions. Moreover, we sharpen the nice basis condition from \cite{K21} to provide an algebraic criterion for NR-decompositions. Consequently, we obtain a large class of compact homogeneous spaces where Theorem \ref{thm:main-extra.symm} and Theorem \ref{thm:main-coll.solitons} can be applied. 

\subsection{Existence of diagonal ancient solutions}
\label{subsect:example-RF} \hfill \par

Assume that $M^m = \mathsf{G}/\mathsf{H}$ admits an NR-decomposition $\varphi = (\mathfrak{m}_1,{\dots},\mathfrak{m}_{\ell})$ for its isotropy representation. Let also $\mathsf{T} \subset N_{\mathsf{G}}(\mathsf{H})/\mathsf{H}$ be an $s$-dimensional connected, abelian Lie subgroup and assume that its Lie algebra $\mathfrak{t} \coloneqq {\rm Lie}(\mathsf{T}) \subset \mathfrak{m}_0$ is given by $\mathfrak{t} = \mathfrak{m}_1 + {\dots} + \mathfrak{m}_s$. Consider the (locally) homogeneous torus bundle $\mathsf{T} \to M \to B = M/\mathsf{T}$ and the corresponding $Q$-orthogonal splitting at the Lie algebra level \eqref{eq:decomp}. Then, the decomposition $\varphi$ for the ${\rm Ad}(\mathsf{H})$-action on $\mathfrak{m}$ uniquely induces an ${\rm Ad}(\mathsf{H}\mathsf{T})$-invariant, irreducible decomposition $\check{\varphi} = (\mathfrak{b}_1,{\dots},\mathfrak{b}_r)$ for $\mathfrak{b}$ by gathering together the modules in $\varphi$ that are interchanged by the ${\rm Ad}(\mathsf{T})$-action. It can be directly checked that $\check{\varphi}$ is also an NR-decomposition. Then, the argument used in \cite{PedSb22} allows us to prove the following.

\begin{theorem}[c.f. \cite{PedSb22}] \label{thm:existence-curvnorm}
Let $M = \mathsf{G}/\mathsf{H}$ be a compact homogeneous space and $\mathsf{T} \subset N_{\mathsf{G}}(\mathsf{H})/\mathsf{H}$ a torus. Assume that $M$ admits an NR-decomposition $\varphi$ for its isotropy representation and let $\check{\varphi}$ be the induced decomposition for $B = M/\mathsf{T}$. If $B$ admits a $\check{\varphi}$-diagonal homogeneous Einstein metric, then $M$ admits a $\varphi$-diagonal, collapsed ancient solution $g(t)$ to the homogeneous Ricci flow, with rigid collapsing torus $\mathsf{T}$.
\end{theorem}

\begin{proof}
Consider the space of submersion metrics on $M$ that are $\varphi$-diagonal and possibly not positive-definite on the toral fibers, i.e.,
$$
\EuScript{N} \coloneqq \big\{ g = y_1\,Q|_{\mathfrak{m}_1{\otimes}\mathfrak{m}_1} + {\dots} + y_s\,Q|_{\mathfrak{m}_s{\otimes}\mathfrak{m}_s} + x_1\,Q|_{\mathfrak{b}_1{\otimes}\mathfrak{b}_1} + {\dots} + x_r\,Q|_{\mathfrak{b}_r{\otimes}\mathfrak{b}_r} : y_i \in \mathbb{R} \, , \,\, x_j>0 \big\} \,\, .
$$
By \cite[Proposition 3.1]{PedSb22}, the Ricci curvature can be extended analytically to the space $\EuScript{N}$. Moreover, since $\varphi$ is an NR-decomposition, the Ricci curvature is diagonal, so we obtain an analytic function ${\rm Ric} : \EuScript{N} \to \mathbb{R}^{s+r}$. Now let
$$
\check{g} = \check{x}_1\,Q|_{\mathfrak{b}_1{\otimes}\mathfrak{b}_1} + {\dots} + \check{x}_r\,Q|_{\mathfrak{b}_r{\otimes}\mathfrak{b}_r}
$$
be a $\mathsf{G}$-invariant, unit volume, $\check{\varphi}$-diagonal Einstein metric on $B = \mathsf{G}/\mathsf{H}\mathsf{T}$ and consider the Euclidean scalar product on $\mathbb{R}^{s+r}$ defined by
\begin{equation} \label{def:prod}
\langle\!\langle (y_1,{\dots},y_s,x_1,{\dots},x_r)^t, (y'_1,{\dots},y'_s,x'_1,{\dots},x'_r)^t\rangle\!\rangle^{\mathsmaller{(\check{g})}} \coloneqq \sum_{q=1}^s \frac{1}{\dim(\mathfrak{b})}y_qy'_q + \sum_{k=1}^r \frac{\dim(\mathfrak{b}_k)}{\dim(\mathfrak{b})}\frac{x_k x'_k}{\check{x}_k^2} \,\, .
\end{equation}
We define then the {\it $\check{g}$-projected Ricci curvature} to be the tensor
\begin{equation} \label{def:R}
R^{\mathsmaller{(\check{g})}}: \EuScript{N} \to \mathbb{R}^{s+r} \,\, , \quad R^{\mathsmaller{(\check{g})}}(g) \coloneqq {\rm Ric}(g) -\frac{\langle\!\langle {\rm Ric}(g), g\rangle\!\rangle^{\mathsmaller{(\check{g})}}}{\langle\!\langle g, g\rangle\!\rangle^{\mathsmaller{(\check{g})}}}g
\end{equation}
obtained by projecting the Ricci curvature to the unit sphere
$$
\Sigma^{\mathsmaller{(\check{g})}} \coloneqq \{g \in \EuScript{N} : \langle\!\langle g,g \rangle\!\rangle^{\mathsmaller{(\check{g})}} = 1\} \,\, .
$$
Therefore, by applying the Center Manifold Theorem \cite[p.\ 116]{Per01} as in the proof of \cite[Theorem 4.3]{PedSb22}, it is possible to find an ancient solution to the {\it $\check{g}$-projected Ricci flow}
$$
\partial_t h(t) = -2R^{\mathsmaller{(\check{g})}}(h(t))
$$
that lives in $\Sigma^{\mathsmaller{(\check{g})}} \cap (\mathbb{R}_{>0})^{s+r}$ and converges to $0 \oplus \check{g}$ as $t \to -\infty$. Notice that the $\check{g}$-projected Ricci flow is equivalent to the Ricci flow, and thus it yields a $\varphi$-diagonal Ricci flow solution $g(t)$ on $M$. Arguing as in the proof of \cite[Theorem 4.3]{PedSb22}, $g(t)$ is ancient as well. Moreover, the corresponding curvature-normalized solution $\bar{g}(t)$ converges, up to scaling, to $\check{g}$ in the Gromov-Hausdorff topology as $t \to -\infty$, and so $g(t)$ is collapsed. Finally, by construction, $\mathsf{T}$ is the collapsing torus of $g(t)$ and is rigid.
\end{proof}

We point out that the presence of an NR-decomposition allowed us to drop the requirement of maximality of the torus $\mathsf{T}$ that was needed in the proof of \cite[Theorem A]{PedSb22}. As a by-product, this provides a method to construct solutions that shrink a non-maximal torus. Up until now, the only examples of this type have been products with flat tori. Finally, we remark that the proof of the above theorem goes through even if we replace NR with stably Ricci-diagonal.

\subsection{Examples of NR-decompositions}
\label{subsect:stronglynice} \hfill \par

We introduce the following definition in light of \eqref{eq:nicebasis} and of Definition \ref{def:stronglyRd}.

\begin{definition} \label{def:stronglynice}
A $Q$-orthonormal basis $\EuScript{B}$ for $\mathfrak{m}$ is said to be {\it strongly nice} if there exists a decomposition $\varphi = (\mathfrak{m}_1, {\dots}, \mathfrak{m}_{\ell}) \in \EuScript{F}^{\mathsf{G}}$ with respect to which $\EuScript{B}$ is $\varphi$-adapted and the following property holds true: if $\mathfrak{m}_i \simeq \mathfrak{m}_j$ and $i \neq j$, then
\begin{equation} \label{eq:stronglynice}
Q([e_1,e_2], v)\,Q([e_1,e_2], w) = 0 \quad \text{ for any $e_1, e_2 \in \EuScript{B}$, $v \in \EuScript{B} \cap \mathfrak{m}_i$, $w \in \EuScript{B} \cap \mathfrak{m}_j$} \,\, .
\end{equation}
\end{definition}

Clearly, every strongly nice basis is also nice. Moreover, if $M = \mathsf{G}$ is a Lie group, i.e., $\mathsf{H}=\{e\}$, then a $Q$-orthonormal basis $\EuScript{B}$ for the Lie algebra $\mathfrak{g}$ is strongly nice if and only if it is nice (see Section \ref{subsect:diagRicci}). The significance of Definition \ref{def:stronglynice} is illustrated by the following observation.

\begin{proposition}
Let $\varphi \in \EuScript{F}^{\mathsf{G}}$ be a decomposition and assume that there exists a strongly nice basis $\EuScript{B}$ that is $\varphi$-adapted. Then, $\varphi$ is an NR-decomposition.
\end{proposition}

\begin{proof}
Let $\varphi = (\mathfrak{m}_1, {\dots}, \mathfrak{m}_{\ell})$ and $\EuScript{B}$ as in the hypotheses. Notice that, by Theorem \ref{thm:K21}, the decomposition $\varphi$ is necessarily stably Ricci-diagonal. We must prove it is also normalizer-adapted (see Definition \ref{def:norm-diag}). Fix $1 \leq p \leq \ell_0$, $1 \leq i \leq \ell$ and assume that $[\mathfrak{m}_p, \mathfrak{m}_i] \neq \{0\}$. Since the action of ${\rm ad}(V_p)$ is by intertwining maps, the image $[\mathfrak{m}_p, \mathfrak{m}_i]$ is an irreducible ${\rm Ad}(\mathsf{H})$-module equivalent to $\mathfrak{m}_i$. Moreover, since $[\mathfrak{m}_p, \mathfrak{m}_i] \neq \{0\}$, there exist $e_1 \in \EuScript{B} \cap \mathfrak{m}_i$ and $e_2 \in \EuScript{B} \cap \mathfrak{m}_j$ such that $Q([V_p,e_1],e_2) \neq 0$, for some $\mathfrak{m}_j \simeq \mathfrak{m}_i$. By \eqref{eq:stronglynice}, for any other $\mathfrak{m}_{j'} \simeq \mathfrak{m}_i$, it follows that $Q([V_p,e_1], w) = 0$ for any $w \in \EuScript{B} \cap \mathfrak{m}_{j'}$, and so we get $[\mathfrak{m}_p, \mathfrak{m}_i] = \mathfrak{m}_j$. \end{proof}

As far as we know, all examples of stably Ricci-diagonal decompositions admit a strongly nice basis. In particular, all known stably Ricci-diagonal decompositions are NR-decompositions. In the remainder of this section, we give examples of strongly nice bases by using well-known facts from the theory of root spaces. We refer to Appendix \ref{subsect:roots} for more details and notations.

\begin{example} \label{example:sniceso(n)}
Consider $M = \mathsf{SO}(n)$ and define $E_{ij}$ to be the $n{\times}n$ matrix with $1$ in the $(i,j)$-entry, $-1$ in the $(j,i)$-entry, and $0$ in all other entries. The collection $\EuScript{B}_{\rm st} = \{E_{ij}\}_{1\leq i<j \leq n}$ is called the {\it standard basis for $\mathfrak{so}(n)$}. By explicit computations, one can check that $\EuScript{B}_{\rm st}$ is nice, hence also strongly nice.
\end{example}

\begin{example} \label{example:sniceso(n)2}
Consider the homogenous space $M = \mathsf{G}/\mathsf{H}$, with $\mathsf{G} = \mathsf{SO}(n)$ and
$$
\mathsf{H} = \mathsf{SO}(p_1) \times {\dots} \times \mathsf{SO}(p_k) \subset \mathsf{G} \,\text{ in block embedding} \,\, , \quad \text{with} \quad q \coloneqq n -\sum_{i=1}^k p_i \geq 1 \,\, .
$$
Then, $\mathfrak{m}_0 \simeq \mathfrak{so}(q)$ and $\mathfrak{m}$ admits a decomposition $\varphi$ whose irreducible modules are subspaces of $\mathfrak{so}(n)$ defined by the off-diagonal blocks. Since $\EuScript{B}_{\rm st} = \{E_{ij}\}_{1\leq i < j \leq n}$ is a nice basis for $\mathfrak{g}$ which respects both the splitting $\mathfrak{g} = \mathfrak{h} + \mathfrak{m}$ and the decomposition $\varphi$, it follows that $\EuScript{B}_{\rm st} \cap \mathfrak{m}$ is a strongly nice basis for $\mathfrak{m}$.
\end{example}

\begin{example} \label{example:blockSp}
Consider the homogenous space $M = \mathsf{G}/\mathsf{H}$, with $\mathsf{G} = \mathsf{Sp}(n)$ and
$$
\mathsf{H} = \mathsf{Sp}(p_1) \times {\dots} \times \mathsf{Sp}(p_k) \subset \mathsf{G} \,\text{ in block embedding} \,\, , \quad \text{with} \quad n -\sum_{i=1}^k p_i = 1 \,\, .
$$
Then, $\mathfrak{m}_0 \simeq \mathfrak{sp}(1)$ and $\mathfrak{m}$ admits a decomposition $\varphi$ whose irreducible modules are subspaces of $\mathfrak{sp}(n)$ defined by the off-diagonal blocks. Since $\mathsf{H}$ is regular in $\mathsf{G}$, one can consider the real Cartan-Weyl basis $\EuScript{B}$ for $\mathfrak{m}$ (see Appendix \ref{subsect:roots}). One can check that each module of $\varphi$ is either ${\rm Ad}(\mathsf{H})$-trivial, or a sum of real root spaces of $\mathfrak{g}$, so $\EuScript{B}$ is $\varphi$-adapted. Moreover, the $Q$-orthogonal complement of $\mathfrak{m}_0$ inside $\mathfrak{m}$ does not contain any pair of equivalent ${\rm Ad}(\mathsf{H})$-modules. Therefore, using \eqref{eq:structcoef1}, one can check that $\EuScript{B}$ is strongly nice.
\end{example}

Notice that any collapsed ancient solution on $\mathsf{Sp}(n)/\Pi_i \mathsf{Sp}(p_i)$ is gauge-equivalent to a diagonal solution. More generally, the following is true.

\begin{proposition} \label{prop:rank1diagonal}
Let $g(t)$ be a collapsed, ancient solution to the homogeneous Ricci flow on $M=\mathsf{G}/\mathsf{H}$. Assume that $\mathfrak{m}_0 \simeq \mathfrak{so}(3)$ and the $Q$-orthogonal complement of $\mathfrak{m}_0$ inside $\mathfrak{m}$ does not contain any pair of equivalent ${\rm Ad}(\mathsf{H})$-modules. Then, there exists $f \in N_{\mathsf{G}}(\mathsf{H})/\mathsf{H}$ such that the solution $f^*g(t)$ is diagonal.
\end{proposition}

\begin{proof}
Notice that, by Schur's Lemma, $\mathfrak{m}_0$ is always $g(t)$-orthogonal to its $Q$-orthogonal complement $\mathfrak{m}_0^{\perp}$ inside $\mathfrak{m}$ for any $t \leq 0$, i.e., the metric $g(t)$ splits as
$$
g(t) = g_0(t) + g_0^{\perp}(t) \quad \text{ for any $t \leq 0$} \,\, ,
$$
with
$$
g_0(t) \coloneqq g(t)|_{\mathfrak{m}_0 \otimes \mathfrak{m}_0} \in S^2_+(\mathfrak{m}_0^*) \quad \text{and} \quad g_0^{\perp}(t) \coloneqq g(t)|_{\mathfrak{m}_0^{\perp} \otimes \mathfrak{m}_0^{\perp}} \in S^2_+((\mathfrak{m}_0^{\perp})^*)^{{\rm Ad}(\mathsf{H})} \,\, .
$$
Then, the identity component of the gauge group $N_{\mathsf{G}}(\mathsf{H})/\mathsf{H}$ is isomorphic to $\mathsf{SO}(3)$ or $\mathsf{Sp}(1)$ and, by Remark \ref{rem:gauge_action}, it acts transitively on the space of $Q$-orthonormal bases for $\mathfrak{m}_0$. Thus, up to isometry, we may assume that $g_0(-1)$ is diagonal with respect to the standard basis $\EuScript{B}_{\rm st}$ of $\mathfrak{m}_0$, that is $\{E_{12}, E_{23}, E_{13}\}$ for $\mathfrak{so}(3)$, or $\{\mathtt{i}, \mathtt{j}, \mathtt{k}\}$ for $\mathfrak{sp}(1)$.

Suppose $(N_{\mathsf{G}}(\mathsf{H})/\mathsf{H})_0 \simeq \mathsf{Sp}(1)$. The imaginary unit quaternions $\mathtt{i}, \mathtt{j}, \mathtt{k} \in \mathsf{Sp}(1)$ give rise, via the conjugation action, to diffeomorphisms $\Psi_{\mathtt{i}}, \Psi_{\mathtt{j}}, \Psi_{\mathtt{k}}: M \to M$, with derivatives $\psi_{\mathtt{i}}, \psi_{\mathtt{j}}, \psi_{\mathtt{k}} : \mathfrak{m} \to \mathfrak{m}$, whose restrictions to $\mathfrak{m}_0$ switch the signs two at a time, i.e., 
$$\begin{array}{ccc}
\psi_{\mathtt{i}}(\mathtt{i}) = +\mathtt{i} \, , & \psi_{\mathtt{i}}(\mathtt{j}) = -\mathtt{j} \, , & \psi_{\mathtt{i}}(\mathtt{k}) = -\mathtt{k} \,\, , \\
\psi_{\mathtt{j}}(\mathtt{i}) = -\mathtt{i} \, , & \psi_{\mathtt{j}}(\mathtt{j}) = +\mathtt{j} \, , & \psi_{\mathtt{j}}(\mathtt{k}) = -\mathtt{k} \,\, , \\
\psi_{\mathtt{k}}(\mathtt{i}) = -\mathtt{i} \, , & \psi_{\mathtt{k}}(\mathtt{j}) = -\mathtt{j} \, , & \psi_{\mathtt{k}}(\mathtt{k}) = +\mathtt{k} \,\, .
\end{array}$$
Assume $\mathfrak{m}_0^{\perp}$ does not contain any pair of equivalent ${\rm Ad}(\mathsf{H})$-modules. By Remark \ref{rem:gauge_action}, the linear maps $\psi_{\mathtt{i}}, \psi_{\mathtt{j}}, \psi_{\mathtt{k}}$ preserve the ${\rm Ad}(\mathsf{H})$-invariant submodules of $\mathfrak{m}_0^{\perp}$. As a consequence, the diffeomorphisms $\Psi_{\mathtt{i}}, \Psi_{\mathtt{j}}, \Psi_{\mathtt{k}}$ will in fact be isometries if $g_0(t)$ is diagonal in the basis $\EuScript{B}_{\rm st}$. Since isometries are preserved under the flow, $\Psi_{\mathtt{i}}, \Psi_{\mathtt{j}}, \Psi_{\mathtt{k}}$ are isometries for $g(t)$ at any $t < 0$. Finally, a metric for which these maps are isometries must be diagonal in the basis $\EuScript{B}_{\rm st}$. This shows that $g(t)$ remains diagonal along the flow.

If $(N_{\mathsf{G}}(\mathsf{H})/\mathsf{H})_0 \simeq \mathsf{SO}(3)$, the same argument goes through with the discrete isometries coming from the elements ${\rm diag}(1,-1,-1), {\rm diag}(-1,1,-1), {\rm diag}(-1,-1,1) \in \mathsf{SO}(3)$.
\end{proof}

We remark that, by Proposition \ref{prop:main-uniq.torus}, the collapsing torus of any solution as in Proposition \ref{prop:rank1diagonal} (in this case, a circle) is unique, hence rigid (see Remark \ref{rem:unique-rigid}). \smallskip

Next, we describe a general criterion that guarantees the existence of strongly nice bases when the gauge group $N_{\mathsf{G}}(\mathsf{H})/\mathsf{H}$ has rank 1. This can be seen as a direct generalization of Example \ref{example:blockSp}. We refer the reader to Appendix \ref{subsect:roots} for the necessary notation and terminology.

\begin{proposition} \label{prop:strnice_CW}
Let $\mathsf{G}$ be a compact, connected Lie group and let $\mathsf{H}$ be a regular subgroup such that $N_{\mathsf{G}}(\mathsf{H})/\mathsf{H}$ has rank $1$. Let $\mathcal{R}^+ = \mathcal{R}^+_{\mathfrak{h}} \cup \mathcal{R}^+_{\mathfrak{m}}$ be the decomposition of positive roots according to the splitting $\mathfrak{g} = \mathfrak{h} \oplus \mathfrak{m}$ and assume that for every pair $\alpha, \beta \in \mathcal{R}^+_{\mathfrak{m}}$, at most one of $\alpha + \beta, \alpha - \beta, \beta - \alpha$ lies in $\mathcal{R}^+_{\mathfrak{m}}$. Suppose further that there exists $\varphi = (\mathfrak{m}_1, \dots, \mathfrak{m}_\ell) \in \EuScript{F}^{\mathsf{G}}$ such that each non-trivial $\mathfrak{m}_i \in \varphi$ is a sum of real root spaces of $\mathfrak{g}$. Then, the real Cartan-Weyl basis of $\mathfrak{m}$ is strongly nice.
\end{proposition}

\begin{proof}
We can assume $\mathfrak{m}_0 \simeq \mathfrak{so}(3)$, and the case $\mathfrak{m}_0 \simeq \mathbb{R}$ can be treated similarly. By assumption, there is a unique $\alpha_0 \in \mathcal{R}^+_{\mathfrak{m}}$ such that $\mathfrak{m}_0 = {\rm span}_{\mathbb{R}}\{H_{\alpha_0}\} + \mathfrak{r}_{\alpha_0}$. Then, the real Cartan-Weyl basis of $\mathfrak{m}$ is
\begin{equation} \label{eq:CWbasis}
\EuScript{B} = \{H_{\alpha_0}\} \cup \{X_{\alpha}, Y_{\alpha} : \alpha \in \mathcal{R}^+_{\mathfrak{m}}\} \,\, .
\end{equation}
By hypothesis, there exists a decomposition $\varphi = (\mathfrak{m}_1, \dots, \mathfrak{m}_\ell)$ such that each non-trivial $\mathfrak{m}_i \in \varphi$ is a sum of real root spaces of $\mathfrak{g}$. Then, one can assume that
$$
\mathfrak{m}_1 = {\rm span}_{\mathbb{R}}\{H_{\alpha_0}\} \,\, , \quad \mathfrak{m}_2 = {\rm span}_{\mathbb{R}}\{X_{\alpha_0}\} \,\, , \quad \mathfrak{m}_3 = {\rm span}_{\mathbb{R}}\{Y_{\alpha_0}\}
$$
and therefore $\EuScript{B}$ is $\varphi$-adapted. Fix $2 \leq i \leq j \leq \ell$ and take two vectors $e_1 \in \EuScript{B} \cap \mathfrak{m}_i$, $e_2 \in \EuScript{B} \cap \mathfrak{m}_j$. By \eqref{eq:CWbasis}, there exist $\alpha, \beta \in \mathcal{R}^+_{\mathfrak{m}}$ such that $e_1 \in \mathfrak{r}_{\alpha}$ and $e_2 \in \mathfrak{r}_{\beta}$. By assumption, at most one of $\alpha + \beta, \alpha - \beta, \beta - \alpha$ lies in $\mathcal{R}^+_{\mathfrak{m}}$ and so, by \eqref{eq:structcoef1} and \eqref{eq:structcoef2}, either $[e_1, e_2] \in \mathfrak{h}$ or there exists $e_3 \in \EuScript{B} \cap \mathfrak{m}_k$ such that $[e_1,e_2]_{\mathfrak{m}} = e_3$, for some $1 \leq k \leq \ell$. In particular, for any $1 \leq r, s \leq \ell$ with $\mathfrak{m}_r \simeq \mathfrak{m}_s$ and $r \neq s$, we get
$$
Q([e_1,e_2],v)\,Q([e_1,e_2],w) = 0 \quad \text{for any $v \in \EuScript{B} \cap \mathfrak{m}_r$, $w \in \EuScript{B} \cap \mathfrak{m}_s$} \,\, .
$$
Finally, by the properties of the root spaces, it follows that $[H_{\alpha_0},e_1] = \pm \alpha(H_{\alpha_0}) \tilde{e}_1$, where $\tilde{e}_1$ is the unique vector contained in $(\EuScript{B} \cap \mathfrak{r}_{\alpha}) \setminus \{e_1\}$. Therefore, we obtain that for any $1 \leq r, s \leq \ell$ with $\mathfrak{m}_r \simeq \mathfrak{m}_s$ and $r \neq s$, the following condition holds:
$$
Q([H_{\alpha_0},e_1],v)\,Q([H_{\alpha_0},e_1],w) = 0 \quad \text{for any $v \in \EuScript{B} \cap \mathfrak{m}_r$, $w \in \EuScript{B} \cap \mathfrak{m}_s$} \,\, .
$$
This concludes the proof.
\end{proof}

\begin{corollary} \label{cor:SUnice}
Let $\mathsf{G} = \mathsf{U}(n)$, $\mathsf{G} = \mathsf{SU}(n+1)$ or $\mathsf{G} = \mathsf{SO}(2n)$ and let $\mathsf{H}$ be a regular subgroup such that $N_{\mathsf{G}}(\mathsf{H})/\mathsf{H}$ has rank $1$. If there exists $\varphi = (\mathfrak{m}_1, \dots, \mathfrak{m}_\ell) \in \EuScript{F}^{\mathsf{G}}$ such that each non-trivial $\mathfrak{m}_i \in \varphi$ is a sum of real root spaces of $\mathfrak{g}$, then the real Cartan-Weyl basis of $\mathfrak{m}$ is strongly nice.
\end{corollary}

\begin{proof}
In both cases, by Remark \ref{rem:classicalroots}, if $\alpha, \beta \in \mathcal{R}^+$, then at most one of $\alpha + \beta$, $\alpha - \beta$, $\beta - \alpha$ lies in $\mathcal{R}^+$. Therefore, the result follows from Proposition \ref{prop:strnice_CW}.
\end{proof}

We now list some examples that can be derived from Proposition \ref{prop:strnice_CW} and Corollary \ref{cor:SUnice}, using the notation introduced in Remark \ref{rem:classicalroots}.

\begin{example} \label{example:AWspaces}
Let us consider the Aloff-Wallach spaces, that is $W^7_{(p,q)} = \mathsf{G}/\mathsf{H}_{(p,q)} = \mathsf{SU}(3)/S^1_{(p,q)}$, with $p, q \in \mathbb{Z}$ coprime, where $S^1_{(p,q)}$ is the circle
$$
S^1_{(p,q)} = \Big\{{\rm diag}\big(e^{\mathtt{i}pu}, e^{\mathtt{i}qu}, e^{-\mathtt{i}(p+q)u}\big) : u \in \mathbb{R} \Big\} \,\, .
$$
By explicit computations, one can check that the isotropy representation of $W^7_{(p,q)}$ is multiplicity-free for any choice of $(p,q) \notin \big\{(0,1),(1,1)\big\}$.
\begin{itemize}
\item[$\bcdot$] If $(p,q) = (1,1)$, then
$$
\mathfrak{h}_{(1,1)} = {\rm span}_{\mathbb{R}}\{H_{\theta_1 - \theta_3} +H_{\theta_2 - \theta_3}\} \,\, , \quad \mathfrak{m}_0 = {\rm span}_{\mathbb{R}}\{H_{\theta_1 - \theta_2}\} + \mathfrak{r}_{\theta_1 - \theta_2} \simeq \mathfrak{so}(3)
$$
and $\mathfrak{m}_0^{\perp}$ splits into two equivalent, irreducible ${\rm Ad}(\mathsf{H}_{(1,1)})$-modules, that is
$$
\mathfrak{m}_0^{\perp} = \mathfrak{m}_4 + \mathfrak{m}_5 \,\, , \quad \text{with} \quad \mathfrak{m}_4 = \mathfrak{r}_{\theta_1 - \theta_3} \, , \,\, \mathfrak{m}_5 = \mathfrak{r}_{\theta_2 - \theta_3} \,\, .
$$
\item[$\bcdot$] If $(p,q) = (0,1)$, then
$$
\mathfrak{h}_{(0,1)} = {\rm span}_{\mathbb{R}}\{H_{\theta_2 - \theta_3}\} \,\, , \quad \mathfrak{m}_0 = {\rm span}_{\mathbb{R}}\{H_{\theta_1 - \theta_2} +H_{\theta_1 - \theta_3}\} \simeq \mathbb{R}
$$
and $\mathfrak{m}_0^{\perp}$ splits into three irreducible ${\rm Ad}(\mathsf{H}_{(0,1)})$-modules, that is
$$
\mathfrak{m}_0^{\perp} = \mathfrak{m}_2 + \mathfrak{m}_3 + \mathfrak{m}_4 \,\, , \quad \text{with} \quad \mathfrak{m}_2 = \mathfrak{r}_{\theta_1 - \theta_2} \, , \,\, \mathfrak{m}_3 = \mathfrak{r}_{\theta_1 - \theta_3} \, , \,\, \mathfrak{m}_4 = \mathfrak{r}_{\theta_2 - \theta_3} \,\, .
$$
In this case, $\mathfrak{m}_2 \simeq \mathfrak{m}_3$ as ${\rm Ad}(\mathsf{H}_{(0,1)})$-modules.
\end{itemize}
Therefore, by means of Corollary \ref{cor:SUnice}, the real Cartan-Weyl basis for the isotropy representation of $W^7_{(p,q)}$ is always strongly nice.
\end{example}

\begin{example} \label{example:U(n)U(1)U(n-2)}
Consider $M = \mathsf{G}/\mathsf{H}$, where $\mathsf{G} = \mathsf{U}(n)$ and $\mathsf{H}$ is the connected subgroup of $\mathsf{G}$ with Lie algebra
$$
\mathfrak{h} = {\rm span}_{\mathbb{R}}\{H_{\theta_i - \theta_{i+1}}: 3 \leq i \leq n-1\} +{\rm span}_{\mathbb{R}}\{H_{\theta_1 + \theta_2}, H_{\theta_3 + \theta_4}\} +\sum_{3\leq i < j \leq n} \mathfrak{r}_{\theta_i - \theta_j} \simeq \mathfrak{u}(1) \oplus \mathfrak{u}(n-2) \,\, .
$$
In this case, the trivial submodule is
$$
\mathfrak{m}_0 = {\rm span}_{\mathbb{R}}\{H_{\theta_1 -\theta_2}\} + \mathfrak{r}_{\theta_1 -\theta_2} \simeq \mathfrak{su}(2)
$$
and $\mathfrak{m}_0^\perp$ splits into two ${\rm Ad}(\mathsf{H})$-equivalent submodules
$$
\mathfrak{m}_4 = \sum_{k \geq 3} \mathfrak{r}_{\theta_1 \pm \theta_k} \,\, , \quad \mathfrak{m}_5 = \sum_{k \geq 3} \mathfrak{r}_{\theta_2 \pm \theta_k} \,\, .
$$
By Corollary \ref{cor:SUnice}, the real Cartan-Weyl basis for the isotropy representation of $\mathsf{U}(n)/\mathsf{H}$ is strongly nice.
\end{example}

\begin{example} \label{example:SOeven}
Consider $M = \mathsf{G}/\mathsf{H}$, where $\mathsf{G} = \mathsf{SO}(2n)$ and $\mathsf{H}$ is the connected subgroup of $\mathsf{G}$ with Lie algebra
$$
\mathfrak{h} = {\rm span}_{\mathbb{R}}\{H_{\theta_1 - \theta_2}, H_{\theta_3}, \dots, H_{\theta_n}\} + \mathfrak{r}_{\theta_1 - \theta_2} + \sum_{3\leq i < j \leq n} \mathfrak{r}_{\theta_i \pm \theta_j} \simeq \mathfrak{so}(3) \oplus \mathfrak{so}(2n -4)  \,\, .
$$
In this case
$$
\mathfrak{m}_0 = {\rm span}_{\mathbb{R}}\{H_{\theta_1 +\theta_2}\} + \mathfrak{r}_{\theta_1 +\theta_2} \simeq \mathfrak{so}(3) \,\, , \quad \mathfrak{m}_0^\perp = \sum_{k \geq 3} \mathfrak{r}_{\theta_1 \pm \theta_k} + \sum_{k \geq 3} \mathfrak{r}_{\theta_2 \pm \theta_k} \,\, .
$$
Therefore, since $\mathfrak{m}_0^\perp$ is irreducible, by Corollary \ref{cor:SUnice}, the real Cartan-Weyl basis for the isotropy representation of $\mathsf{SO}(2n)/\mathsf{H}$ is strongly nice.
\end{example}

\begin{example} \label{example:SOodd}
Consider $M = \mathsf{G}/\mathsf{H}$, where $\mathsf{G} = \mathsf{SO}(2n+1)$ and $\mathsf{H}$ is the connected subgroup of $\mathsf{G}$ with Lie algebra
$$
\mathfrak{h} = {\rm span}_{\mathbb{R}}\{H_{\theta_1 - \theta_2}, H_{\theta_3}, \dots, H_{\theta_n}\} + \mathfrak{r}_{\theta_1 - \theta_2} + \sum_{k \geq 3} \mathfrak{r}_{\theta_k} + \sum_{3\leq i < j \leq n} \mathfrak{r}_{\theta_i \pm \theta_j} \simeq \mathfrak{so}(3) \oplus \mathfrak{so}(2n -3) \,\, .
$$
In this case
$$
\mathfrak{m}_0 = {\rm span}_{\mathbb{R}}\{H_{\theta_1 + \theta_2}\} + \mathfrak{r}_{\theta_1 + \theta_2} \simeq \mathfrak{so}(3) \,\, , \quad \mathfrak{m}_0^\perp = \mathfrak{r}_{\theta_1} + \mathfrak{r}_{\theta_2} + \sum_{k \geq 3} \mathfrak{r}_{\theta_1 \pm \theta_k} + \sum_{k \geq 3} \mathfrak{r}_{\theta_2 \pm \theta_k} \,\, .
$$
Therefore, since $\mathfrak{m}_0^\perp$ is irreducible, by Proposition \ref{prop:strnice_CW}, the real Cartan-Weyl basis for the isotropy representation of $\mathsf{SO}(2n+1)/\mathsf{H}$ is strongly nice.
\end{example}

We remark that, by Proposition \ref{prop:main-uniq.torus}, in the cases listed in Example \ref{example:sniceso(n)2} with $q=3$, Example \ref{example:AWspaces}, Example \ref{example:U(n)U(1)U(n-2)}, Example \ref{example:SOeven} and Example \ref{example:SOodd}, the collapsing torus of any diagonal ancient solution is unique, hence rigid (see Remark \ref{rem:unique-rigid}). \smallskip

Up until now, the only examples of compact homogeneous spaces $\mathsf{G}/\mathsf{H}$, with ${\rm rank}\big(N_{\mathsf{G}}(\mathsf{H})/\mathsf{H}\big) >1$, admitting an NR-decomposition are those listed in Example \ref{example:sniceso(n)} and in Example \ref{example:sniceso(n)2} with $q>3$. Using the theory from Appendix \ref{subsect:roots}, we present below a method to construct examples whose gauge group has a higher rank. In the following, given a pair of groups $\mathsf{H} \subset \mathsf{G}$, we denote by $\Delta \mathsf{H}$ the diagonal
$$
\Delta \mathsf{H} \coloneqq \big\{(h,h) : h \in \mathsf{H} \big\} \subset \mathsf{G}{\times}\mathsf{G} \,\, .
$$

\begin{proposition} \label{prop:higherrank}
Let $\mathsf{G}$ be a compact, connected, semisimple Lie group and $\mathsf{H} \subset \mathsf{G}$ a regular subgroup. Assume also that the isotropy representation $\mathfrak{m}$ of $\mathsf{G}/\mathsf{H}$ satisfies the following assumptions: \begin{itemize}
\item[i)] its real Cartan-Weyl basis is strongly nice;
\item[ii)] $\mathfrak{m}$ does not contain any ${\rm Ad}(\mathsf{H})$-module that is equivalent to an ${\rm Ad}(\mathsf{H})$-module in its Lie algebra $\mathfrak{h}$.
\end{itemize}
Then, the homogeneous space $\mathsf{G} {\times} \mathsf{G} / \Delta \mathsf{H}$ admits a strongly nice basis for its isotropy representation.
\end{proposition}

\begin{proof}
Denote by $\mathcal{R}^+ = \mathcal{R}^+_{\mathfrak{h}} \cup \mathcal{R}^+_{\mathfrak{m}}$ the set of positive roots of $\mathsf{G}/\mathsf{H}$ according to the splitting $\mathfrak{g} = \mathfrak{h} \oplus \mathfrak{m}$, by $\EuScript{B}_{\mathfrak{h}}$ the real Cartan-Weyl basis for the Lie algebra $\mathfrak{h}$, by $\EuScript{B}_{\mathfrak{m}}$ the real Cartan-Weyl basis for $\mathfrak{m}$ and let $\varphi = (\mathfrak{m}_1, \dots, \mathfrak{m}_\ell)$ be a decomposition with respect to which $\EuScript{B}$ is $\varphi$-adapted. Then, $\mathfrak{g} \oplus \mathfrak{g}$ splits as
\begin{equation} \label{eq:decGxG}
\mathfrak{g} \oplus \mathfrak{g} = \Delta \mathfrak{h} + \tilde{\Delta}\mathfrak{h} + \sum_{i = 1}^{\ell}\mathfrak{m}_i^{(1)} + \sum_{j = 1}^{\ell}\mathfrak{m}_j^{(2)} \,\, ,
\end{equation}
with
$$\begin{gathered}
\Delta \mathfrak{h} \coloneqq {\rm Lie}(\Delta \mathsf{H}) = \{(H,H) : H \in \mathfrak{h}\} \,\, , \quad \tilde{\Delta}\mathfrak{h} \coloneqq (\Delta \mathfrak{h})^{\perp} \cap (\mathfrak{h} \oplus \mathfrak{h}) = \{(H,-H) : H \in \mathfrak{h}\} \,\, , \\
\mathfrak{m}_i^{(1)} \coloneqq \{(X,0) : X \in \mathfrak{m}_i\} \,\, , \quad \mathfrak{m}_j^{(2)} \coloneqq \{(0,X) : X \in \mathfrak{m}_j\} \,\, .
\end{gathered}$$
Notice that the ${\rm Ad}(\Delta\mathsf{H})$-modules $\mathfrak{m}_i^{(1)}$ and $\mathfrak{m}_j^{(2)}$ are irreducible, while the ${\rm Ad}(\Delta\mathsf{H})$-action on $\tilde{\Delta}\mathfrak{h}$ is equivalent to the ${\rm Ad}(\mathsf{H})$-action on $\mathfrak{h}$.
Let us consider now the following sets of vectors:
$$
\EuScript{B}_{\tilde{\Delta}\mathfrak{h}} \coloneqq \{(H,-H) : H \in \EuScript{B}_{\mathfrak{h}}\} \,\, , \quad \EuScript{B}_{\mathfrak{m}}^{(1)} \coloneqq \{(X,0) : X \in \EuScript{B}_{\mathfrak{m}}\} \,\, , \quad \EuScript{B}_{\mathfrak{m}}^{(2)} \coloneqq \{(0,X) : X \in \EuScript{B}_{\mathfrak{m}}\} \,\, .
$$
By hypothesis, there is no ${\rm Ad}(\Delta\mathsf{H})$-module in $\tilde{\Delta}\mathfrak{h}$ that is equivalent to some $\mathfrak{m}_i^{(1)}$ or $\mathfrak{m}_j^{(2)}$. Then, one can check that $\EuScript{B} \coloneqq \EuScript{B}_{\tilde{\Delta}\mathfrak{h}} \cup \EuScript{B}_{\mathfrak{m}}^{(1)} \cup \EuScript{B}_{\mathfrak{m}}^{(2)}$ is a strongly nice basis for the isotropy representation of $\mathsf{G} {\times} \mathsf{G} / \Delta \mathsf{H}$.
\end{proof}

Notice that, generally, $\Delta \mathsf{H}$ is not regular in $\mathsf{G} \times \mathsf{G}$. In fact, under the hypotheses of Proposition \ref{prop:higherrank}, by \eqref{eq:decGxG} we obtain
$$
{\rm rank}\big(N_{\mathsf{G}{\times}\mathsf{G}}(\Delta\mathsf{H})\big) = {\rm rank}\big(\mathsf{G}{\times}\mathsf{G}\big) -{\rm rank}\big([\mathsf{H},\mathsf{H}]\big) \,\, ,
$$
where $[\mathsf{H},\mathsf{H}]$ denotes the commutator of $\mathsf{H}$. Moreover, again by \eqref{eq:decGxG}, notice that
$$
{\rm rank}\big(N_{\mathsf{G}{\times}\mathsf{G}}(\Delta\mathsf{H})/\Delta\mathsf{H}\big) = 2{\rm rank}\big(N_{\mathsf{G}}(\mathsf{H})/\mathsf{H}\big) +\dim\!\big(Z(\mathsf{H})\big) \,\, ,
$$
where $Z(\mathsf{H})$ denotes the center of $\mathsf{H}$. 

Therefore, this allows us to construct compact homogeneous spaces whose gauge group has rank at least $2$ and whose isotropy representation admits a strongly nice basis, e.g., $\mathsf{Sp}(n+1){\times}\mathsf{Sp}(n+1) / \Delta\mathsf{Sp}(n)$.

\appendix

\section{Regular subalgebras and real Cartan-Weyl bases of compact Lie groups}
\label{subsect:roots} \setcounter{equation} 0

In this appendix, we collect some basic facts about root spaces and regular subgroups of compact Lie groups, that get used in Section \ref{sect:example-sRd}. \smallskip

Let $\mathsf{G}$ be a compact, connected, semisimple Lie group with Lie algebra $\mathfrak{g} = {\rm Lie}(\mathsf{G})$ and $\mathsf{T}$ a maximal torus of $\mathsf{G}$ with Lie algebra $\mathfrak{t} = {\rm Lie}(\mathsf{T}) \subset \mathfrak{g}$. In this case, we choose as background ${\rm Ad}(\mathsf{G})$-invariant metric $Q$ on $\mathfrak{g}$ the negative Cartan-Killing form of $\mathsf{G}$. Let $\mathcal{R} \subset \mathfrak{t}^*$ be a root system for $\mathfrak{g}$ and denote by
$$
\mathfrak{g}_{\alpha} \coloneqq \{ E \in \mathfrak{g} \otimes_{\mathbb{R}} \mathbb{C} : [H,E] = \mathtt{i}\alpha(H)E \text{ for any $H \in \mathfrak{t}$}\}
$$
the complex $1$-dimensional $\alpha$-root space, for any $\alpha \in \mathcal{R}$. Fix now a choice of positive roots $\mathcal{R}^+ \subset \mathcal{R}$ and consider the corresponding real root spaces decomposition for $\mathfrak{g}$, that is
\begin{equation} \label{eq:realroots}
\mathfrak{g} = \mathfrak{t} + \sum_{\alpha \in \mathcal{R}^+} \mathfrak{r}_{\alpha} \,\, , \quad \text{with $\mathfrak{r}_{\alpha} \coloneqq (\mathfrak{g}_{\alpha} +\mathfrak{g}_{-\alpha}) \cap \mathfrak{g} $} \, .
\end{equation}
For any $\vartheta \in \mathfrak{t}^*$, we denote by $H_{\vartheta} \in \mathfrak{t}$ the $Q$-dual vector of $\vartheta$.

Following \cite[Sections III.4 - III.5 - III.6]{Hel78} (see also \cite[Section 4.3]{AlBe15}), for any $\alpha \in \mathcal{R}^+$ there exists a $Q$-orthonormal basis $\{X_{\alpha},Y_{\alpha}\}$ for $\mathfrak{r}_{\alpha}$ that satisfies the following properties:
\begin{equation} \begin{gathered} \label{eq:structcoef1} 
[X_{\alpha},X_{\beta}] = \eta_{\alpha, \beta}X_{\alpha +\beta} -{\rm sgn}(\alpha -\beta)\eta_{\alpha, -\beta}X_{|\alpha -\beta|} \,\, ,\\
[X_{\alpha},Y_{\beta}] = \eta_{\alpha, \beta}Y_{\alpha +\beta} +{\rm sgn}(\alpha -\beta)\eta_{\alpha, -\beta}Y_{|\alpha -\beta|} \,\, ,\\
[Y_{\alpha},Y_{\beta}] = -\eta_{\alpha, \beta}X_{\alpha +\beta} -{\rm sgn}(\alpha -\beta)\eta_{\alpha, -\beta}X_{|\alpha -\beta|} \,\, ,
\end{gathered} \end{equation}
where
$$
{\rm sgn}(\alpha -\beta) \coloneqq \left\{\!\! \begin{array}{ll}
+1 & \text{if $\alpha -\beta \in \mathcal{R}^+$} \\
-1 & \text{if $\beta -\alpha \in \mathcal{R}^+$} \\
0 & \text{otherwise}
\end{array} \right.
$$
and $|\alpha -\beta| \coloneqq {\rm sgn}(\alpha -\beta) (\alpha -\beta)$. Here, the coefficients $n_{\alpha, \beta} \in \mathbb{R}$ satisfy
\begin{equation} \begin{gathered} \label{eq:structcoef2} 
n_{\alpha, \beta} = 0 \quad \text{if $\alpha +\beta \notin \mathcal{R}$} \, , \\
n_{\alpha, \beta} = -n_{\beta, \alpha} \, , \,\, n_{-\alpha, -\beta} = -n_{\alpha, \beta} \, , \\
n_{\alpha, \beta} = n_{\beta, -(\alpha +\beta)} = n_{-(\alpha +\beta), \alpha} \, .
\end{gathered} \end{equation}
We call it a {\it real Cartan-Weyl basis for $\mathfrak{g}$}.

\begin{remark} \label{rem:classicalroots}
We point out that canonical choices exist for positive roots of compact, simple Lie groups of classical type. In the description of positive roots below, we follow the notation of \cite{BtD95}:
$$\begin{array}{c||c}
\mathfrak{g} & \mathcal{R}^+ \\
\hline
\hline
\mathfrak{su}(n+1) & \{\theta_i - \theta_j : 1\leq i < j \leq n+1\} \\
\hline
\mathfrak{so}(2n+1) & \{\theta_i \pm \theta_j : 1\leq i < j \leq n\} \cup \{\theta_k : 1\leq k \leq n\} \\
\hline
\mathfrak{sp}(n) & \{\theta_i \pm \theta_j : 1\leq i < j \leq n\} \cup \{2\theta_k : 1\leq k \leq n\} \\
\hline
\mathfrak{so}(2n) & \{\theta_i \pm \theta_j : 1\leq i < j \leq n\}
\end{array}$$
We refer to \cite[Section V.6]{BtD95} for the details.
\end{remark}

Let $\mathsf{H}$ be a compact and connected subgroup of $\mathsf{G}$. We recall that $\mathsf{H}$ is said to be {\it regular} if its normalizer $N_{\mathsf{G}}(\mathsf{H})$ contains a maximal torus $\mathsf{T}$ of $\mathsf{G}$. By \cite[p. 150]{OnVin90}, it follows that $\mathsf{H}$ is regular if and only if its Lie algebra $\mathfrak{h} = {\rm Lie}(\mathsf{H})$ splits as
\begin{equation} \label{eq:regularh}
 \mathfrak{h} = \mathfrak{t}' + \sum_{\alpha \in \mathcal{R}^+_{\mathfrak{h}}} \mathfrak{r}_{\alpha} \,\, .
\end{equation}
Here, $\mathfrak{t}'$ is a torus contained in the maximal torus $\mathfrak{t}$, $\mathfrak{r}_{\alpha}$ denotes the real $\alpha$-root space defined in \eqref{eq:realroots}, and $\mathcal{R}^+_{\mathfrak{h}} \subset \mathcal{R}^+$ is a {\it closed} subset of positive roots, meaning that for any $\alpha, \beta \in \mathcal{R}^+_{\mathfrak{h}}$, the condition $\alpha \pm \beta \in \mathcal{R}^+$ implies $\alpha \pm \beta \in \mathcal{R}^+_{\mathfrak{h}}$. We denote by $\mathcal{R}^+_{\mathfrak{m}} \coloneqq \mathcal{R}^+ \setminus \mathcal{R}^+_{\mathfrak{h}}$ the complement of this subset of positive roots.

As a consequence of \eqref{eq:regularh}, when $\mathsf{H} \subset \mathsf{G}$ is regular, a real Cartan-Weyl basis for $\mathfrak{g}$ as in \eqref{eq:structcoef1} restricts to a basis for the isotropy representation $\mathfrak{m}$ of the homogeneous space $\mathsf{G}/\mathsf{H}$, which we refer to as a {\it real Cartan-Weyl basis for $\mathfrak{m}$}. This will be used in the construction of examples in Section \ref{sect:example-sRd}. Notice, however, that in general a real Cartan-Weyl basis is not adapted to any ${\rm Ad}(\mathsf{H})$-invariant, irreducible decomposition $\varphi$ of $\mathfrak{m}$. This occurs, for instance, for the homogeneous space $M = \mathsf{SO}(n+2)/\mathsf{SO}(n)$, with $n \geq 2$ and $\mathsf{SO}(n) \subset \mathsf{SO}(n+2)$ given by the block embedding.

\section{Proof of Theorem \ref{thm:limittorus}}
\label{sect:proofsAGAG} \setcounter{equation} 0

This appendix consists of the proof of Theorem \ref{thm:limittorus}, from which Theorem \ref{thm:main-existence.torus} follows. We use the same notation as Section \ref{sect:limitbehavior}. In the interest of simplifying the notation, we omit to differentiate between the original sequence $\xi = \{\tau^{(n)}\}$ and its subsequences, and to specify the subsequence extractions. \smallskip

First, notice that claim $i)$ follows from the compactness of $\EuScript{F}^{\mathsf{G}}$. Claim $ii)$ can be deduced from the fact that $g(t)$ is collapsed. Indeed, by \eqref{eq:estscal}, there exists a constant $c >0$ such that
$$
\tfrac1{|t|}g(t) \leq c\cdot{\rm scal}\big({\rm vol}(g(t))^{-\frac2m}g(t)\big) \cdot {\rm vol}(g(t))^{-\frac2m}g(t) \,\, .
$$
Moreover, since $g(t)$ is collapsed, we have
$$
{\rm scal}\big({\rm vol}(g(t))^{-\frac2m}g(t)\big) \to 0
$$
as $t \to -\infty$ (see Section \ref{subsect:homRF}). Because ${\rm vol}(g(t))^{-\frac2m}g(t)$ has unit volume, it then follows that the volume of $\tfrac{1}{|t|}g(t)$ tends to $0$ as $t \to -\infty$. Therefore, for any sequence of times $t^{(n)} \to -\infty$, the smallest eigenvalue of $\tfrac{1}{|t^{(n)}|}g(t^{(n)})$ goes to zero.
\smallskip

From now on, we will always assume that the decompositions $\varphi^{(n)}$ are ordered in such a way that
$$
x_1^{(n)} \leq x_2^{(n)} \leq {\dots} \leq x_{\ell}^{(n)} \,\, .
$$
We denote by $b_i^{(n)}$ and $b_i^{(\infty)}$ the coefficients defined in \eqref{def:b_i} with respect to $\varphi^{(n)}$ and $\varphi^{(\infty)}$, respectively. For the sake of shortness, we set $[ijk]^{(n)} \coloneqq [ijk]_{\varphi^{(n)}}$, $[ijk]^{(\infty)} \coloneqq [ijk]_{\varphi^{(\infty)}}$ and
$$
\bar{g}^{(n)} \coloneqq \frac1{\tau^{(n)}}g(-\tau^{(n)}) \,\, , \qquad \bar{x}^{(n)}_i \coloneqq \frac{x^{(n)}_i}{\tau^{(n)}} \quad \text{ for any $1 \leq i \leq \ell$ .}
$$

In order to prove claims $iii)$ and $iv)$, we partition the index set $\{1,{\dots},s_{\xi}\}$ into equivalence classes according to the relative asymptotic growth of the corresponding eigenvalues:
$$
\{1,{\dots},s_{\xi}\} = J_1 \sqcup {\dots} \sqcup J_u \,\, .
$$
More precisely, the index sets $J_i$, for $1 \leq i \leq u$, are defined recursively as follows: for every $p \in \{1,{\dots},s_{\xi}\}$,
$$
\text{$p \in J_{i+1}$ if and only if, for every $q \in \{1,{\dots},s_{\xi}\} \setminus \big(J_1 \sqcup {\dots} \sqcup J_i\big)$, we have $\displaystyle{\lim_{n\to +\infty}\frac{\bar{x}_p^{(n)}}{\bar{x}_q^{(n)}}} \in [0,+\infty)$.}
$$
From the definition of the sets $J_i$, it can be deduced that the following characterization holds: \begin{itemize}
\item[$\bcdot$] if $1 \leq i < j \leq u$ and $p \in J_i$, $q \in J_j$, then $\displaystyle{\lim_{n\to +\infty}\frac{\bar{x}_p^{(n)}}{\bar{x}_q^{(n)}}=0}$;
\item[$\bcdot$]  if $1 \leq i \leq u$ and $p, q \in J_i$, then $\displaystyle{\lim_{n\to +\infty}\frac{\bar{x}_p^{(n)}}{\bar{x}_q^{(n)}}} = z_{p,q} \in (0,+\infty)$.
\end{itemize}
For the sake of notation, we set
$$
J_{\leq h} \coloneqq J_1 \sqcup {\dots} \sqcup J_h \,\, , \quad J_{> h} \coloneqq \{1,{\dots},\ell\} \setminus \big(J_1 \sqcup {\dots} \sqcup J_h\big)
$$
and we define
$$
\mathfrak{a}_h^{(n)} \coloneqq \sum_{p \in J_{\leq h}} \mathfrak{m}_p^{(n)} \,\, , \quad \mathfrak{a}_h^{(\infty)} \coloneqq \sum_{p \in J_{\leq h}} \mathfrak{m}_p^{(\infty)} \,\, .
$$
For any integer $1 \leq h \leq u$, we consider the claim $P(h)$ defined as follows:
\begin{equation} \tag{$P(h)$} \begin{gathered}
\text{$\mathfrak{a}_h^{(\infty)}$ is an abelian subalgebra of $\mathfrak{m}_0$ and both \eqref{eq:AGAG1} and \eqref{eq:AGAG2} are true for any $p \in J_{\leq h}$.}
\end{gathered} \end{equation}
We are going to prove $P(h)$ by induction on $h \in \{1,{\dots},u\}$.

\begin{proof}[Proof of claim $P(1)$]

In order to prove claim $P(1)$, by \eqref{eq:scal} we have
$$
{\rm scal}\big(\bar{g}^{(n)}\big) \leq \frac1{4\bar{x}^{(n)}_1}\Bigg(2\,{\rm Tr}_Q(-\mathcal{B}_{\mathfrak{g}}) -\sum_{1 \leq i,j,k \leq \ell}[ijk]^{(n)}\frac{\bar{x}^{(n)}_1 \bar{x}^{(n)}_i}{\bar{x}^{(n)}_j \bar{x}^{(n)}_k} \Bigg) \,\, .
$$
Therefore, since the scalar curvature of $\bar{g}^{(n)}$ is uniformly bounded (see \eqref{eq:estscal}) and $\bar{x}^{(n)}_1 \to 0$ as established before, we have
$$
4\,\bar{x}^{(n)}_1{\rm scal}\big(\bar{g}^{(n)}\big) \to 0
$$
and so there exists $C>0$ such that
\begin{equation} \label{eq:ijkest-CB0}
\sum_{1 \leq i,j,k \leq \ell}[ijk]^{(n)}\frac{\bar{x}^{(n)}_i}{\bar{x}^{(n)}_j \bar{x}^{(n)}_k} \leq \frac{C}{\bar{x}^{(n)}_1} \quad \text{ for any $n \in \mathbb{N}$} \,\, .
\end{equation}
As a direct consequence of \eqref{eq:ijkest-CB0}:
\begin{equation} \label{eq:claim-subalg0}
\text{$[ijk]^{(\infty)}>0$ \,\,with\,\, $i, j \in J_1$,\, $k \in \{1,{\dots},\ell\}$ \quad $\Longrightarrow$ \quad $k \in J_1$ .}
\end{equation}
Indeed, assume by contradiction that there exist $i, j \in J_1$ and $k \in J_{> 1}$ such that $[ijk]^{(\infty)}>0$. Then, by \eqref{eq:ijkest-CB0}, up to changing the constant, we obtain
$$
\frac{\bar{x}^{(n)}_k}{\bar{x}^{(n)}_i} \leq C \,\, ,
$$
which is not possible. Therefore, \eqref{eq:claim-subalg0} implies that $\mathfrak{a}_1^{(\infty)}$ is a subalgebra of $\mathfrak{g}$.

Let $\{e_{\alpha}^{(n)}\}$ be a $\varphi^{(n)}$-adapted, $Q$-orthonormal basis for $\mathfrak{m}$ and consider the quantities
$$
{\rm sec}^{(n)}_{ij} \coloneqq \sum_{e_{\alpha}^{(n)} \in \mathfrak{m}_i^{(n)}} \sum_{e_{\beta}^{(n)} \in \mathfrak{m}_j^{(n)}} {\rm sec}_{\bar{g}^{(n)}}\big(e_{\alpha}^{(n)},e_{\beta}^{(n)}\big) \,\, ,
$$
where ${\rm sec}$ denotes the sectional curvature. By \cite[Formula (4.6) and Formula (4.7)]{Ped19}, they satisfy the following relations for any $1 \leq i < j \leq \ell$:
\begin{align}
4\,\bar{x}_i^{(n)}{\rm sec}^{(n)}_{ij} &= \sum_{1 \leq k \leq \ell} [ijk]^{(n)} \frac{\bar{x}_i^{(n)}}{\bar{x}_j^{(n)}}\frac{\bar{x}_i^{(n)}}{\bar{x}_k^{(n)}} +\sum_{1 \leq k \leq \ell} [ijk]^{(n)} \bigg(\frac{\bar{x}_j^{(n)}}{\bar{x}_k^{(n)}} -1\bigg)\bigg(1 -2\frac{\bar{x}_i^{(n)}}{\bar{x}_j^{(n)}} +3\frac{\bar{x}_k^{(n)}}{\bar{x}_j^{(n)}}\bigg) \,\, , \label{eq:sec1} \\
4\,\bar{x}_i^{(n)}{\rm sec}^{(n)}_{ii} &= 4d_ic_i^{(n)} +[iii]^{(n)} +4\sum_{\substack{1 \leq k \leq \ell \\ k \neq i}}[iik]^{(n)} -3 \sum_{\substack{1 \leq k \leq \ell \\ k \neq i}}[iik]^{(n)}\frac{\bar{x}_k^{(n)}}{\bar{x}_i^{(n)}} \,\, , \label{eq:sec2}
\end{align}
where $c_i^{(n)}$ are the diagonal components of the {\it Casimir operator} (see \cite{WaZ86}). The rest of the proof is divided into three steps. \smallskip

\noindent{\it Step 1:} Since the sectional curvature of $\bar{g}^{(n)}$ is uniformly bounded, we have
$$
\lim_{n \to +\infty} \big(4\,\bar{x}_p^{(n)}{\rm sec}^{(n)}_{pi}\big) = 0 \quad \text{for any $p \in J_1$, $i \in J_{>1}$ .}
$$
Noticing that
$$
\lim_{n \to +\infty} [pik]^{(n)} \frac{\bar{x}_p^{(n)}}{\bar{x}_i^{(n)}}\frac{\bar{x}_p^{(n)}}{\bar{x}_k^{(n)}} = 0 \quad \text{for any $p \in J_1$, $i \in J_{>1}$, $k \in \{1,{\dots},\ell\}$ ,}
$$
equation \eqref{eq:sec1} implies that
\begin{equation} \label{eq:proofAGAG01'}
\lim_{n \to +\infty}\Bigg\{\sum_{1 \leq k \leq \ell} [pik]^{(n)} \bigg(\frac{\bar{x}_i^{(n)}}{\bar{x}_k^{(n)}} -1\bigg)\bigg(1 -2\frac{\bar{x}_p^{(n)}}{\bar{x}_i^{(n)}} +3\frac{\bar{x}_k^{(n)}}{\bar{x}_i^{(n)}}\bigg)\Bigg\} = 0 \quad \text{for any $p \in J_1$, $i \in J_{>1}$ .}
\end{equation}
Arguing by induction as in \cite[p. 537]{Ped19}, we obtain that each summand in \eqref{eq:proofAGAG01'} is infinitesimal, i.e.,
\begin{equation} \label{eq:proofAGAG01}
\lim_{n \to +\infty}\, [pik]^{(n)} \bigg(\frac{\bar{x}_i^{(n)}}{\bar{x}_k^{(n)}} -1\bigg)\bigg(1 -2\frac{\bar{x}_p^{(n)}}{\bar{x}_i^{(n)}} +3\frac{\bar{x}_k^{(n)}}{\bar{x}_i^{(n)}}\bigg) = 0 \quad \text{for any $p \in J_1$\,,\,\,$i \in J_{>1}$\,,\,\,$k \in \{1,{\dots},\ell\}$\,,}
\end{equation}
and so the following claims hold:
\begin{align}
\text{$p \in J_1$\,,\,\,\,$i \in J_{>1}$\,,\,\,\,$k \in \{1,{\dots},\ell\}$\,,\,\,\,$[pik]^{(\infty)} = 0$} \quad&\Longrightarrow\quad \lim_{n \to +\infty} [pik]^{(n)} \frac{\bar{x}_i^{(n)}}{\bar{x}_k^{(n)}} = 0 \,\, , \label{eq:profAGAG11} \\
\text{$p \in J_1$\,,\,\,\,$i, j \in J_{>1}$\,,\,\,\,$[pij]^{(\infty)} > 0$} \quad&\Longrightarrow\quad \lim_{n \to +\infty} \frac{\bar{x}_i^{(n)}}{\bar{x}_j^{(n)}} = 1 \,\, . \label{eq:profAGAG12}
\end{align}

\noindent{\it Step 2:} Since the sectional curvature of $\bar{g}^{(n)}$ is uniformly bounded, we have
$$
\lim_{n \to +\infty} \big(4\,\bar{x}_p^{(n)}{\rm sec}^{(n)}_{pq}\big) = 0 \quad \text{for any $p, q \in J_1$\,,\,\,$p \neq q$\,.}
$$
Using \eqref{eq:profAGAG11}, we have
\begin{equation} \label{eq:proofAGAG21'}
\lim_{n \to +\infty}\, [pqi]^{(n)} \bigg(\frac{\bar{x}_q^{(n)}}{\bar{x}_i^{(n)}} -1\bigg)\bigg(1 -2\frac{\bar{x}_p^{(n)}}{\bar{x}_q^{(n)}} +3\frac{\bar{x}_i^{(n)}}{\bar{x}_q^{(n)}}\bigg) = 0 \quad \text{for any $p, q \in J_1$\,,\,\,\,$i \in J_{>1}$\,.}
\end{equation}
Moreover
\begin{equation} \label{eq:proofAGAG21''}
\lim_{n \to +\infty}\, [pqi]^{(n)} \frac{\bar{x}_p^{(n)}}{\bar{x}_q^{(n)}}\frac{\bar{x}_p^{(n)}}{\bar{x}_i^{(n)}} = 0 \quad \text{for any $p, q \in J_1$\,,\,\,\,$i \in J_{>1}$\,.}
\end{equation}
Therefore, \eqref{eq:sec1}, \eqref{eq:proofAGAG21'} and \eqref{eq:proofAGAG21''} imply that
\begin{equation} \label{eq:proofAGAG21'''}
\lim_{n \to +\infty} \Bigg\{\sum_{q' \in J_1} [pqq']^{(n)} \Bigg( \frac{\bar{x}_p^{(n)}}{\bar{x}_q^{(n)}}\frac{\bar{x}_p^{(n)}}{\bar{x}_{q'}^{(n)}} +\bigg(\frac{\bar{x}_q^{(n)}}{\bar{x}_{q'}^{(n)}} -1\bigg)\bigg(1 -2\frac{\bar{x}_p^{(n)}}{\bar{x}_q^{(n)}} +3\frac{\bar{x}_{q'}^{(n)}}{\bar{x}_q^{(n)}}\bigg) \Bigg) \Bigg\} = 0 \quad \text{for any $p, q \in J_1$\,,\,\,$p \neq q$\,.}
\end{equation}
Arguing by induction as in \cite[p. 538]{Ped19}, we obtain that each summand in \eqref{eq:proofAGAG21'''} is infinitesimal, i.e.,
$$
\lim_{n \to +\infty} [pqq']^{(n)}\Bigg( \frac{\bar{x}_p^{(n)}}{\bar{x}_q^{(n)}}\frac{\bar{x}_p^{(n)}}{\bar{x}_{q'}^{(n)}} +\bigg(\frac{\bar{x}_q^{(n)}}{\bar{x}_{q'}^{(n)}} -1\bigg)\bigg(1 -2\frac{\bar{x}_p^{(n)}}{\bar{x}_q^{(n)}} +3\frac{\bar{x}_{q'}^{(n)}}{\bar{x}_q^{(n)}}\bigg)\Bigg) = 0 \quad \text{for any $p \in J_1$\,,\,\,$q,q' \in J_1$\,,\,\,$q,q'>p$}
$$
and so the following claim holds:
\begin{equation} \label{eq:profAGAG21}
\text{$p, q, q' \in J_1$\,,\,\,\,$p \leq q < q'$} \quad\Longrightarrow\quad \lim_{n \to +\infty} [pqq']^{(n)} \frac{\bar{x}_{q'}^{(n)}}{\bar{x}_q^{(n)}} = 0 \,\, .
\end{equation}

\noindent{\it Step 3:} Since the sectional curvature of $\bar{g}^{(n)}$ is uniformly bounded, we have
$$
\lim_{n \to +\infty} \big(4\,\bar{x}_p^{(n)}{\rm sec}^{(n)}_{pp}\big) = 0 \quad \text{for any $p \in J_1$\,.}
$$
By \eqref{eq:proofAGAG01} and \eqref{eq:profAGAG11}, we have
$$
\lim_{n \to +\infty} [ppi]^{(n)} \frac{\bar{x}_i^{(n)}}{\bar{x}_p^{(n)}} = 0 \quad \text{for any $p\in J_1$\,,\,\,\,$i \in J_{>1}$}
$$
and so equation \eqref{eq:sec2} implies that
$$
\lim_{n \to +\infty}\Bigg\{\sum_{q \in J_1 \setminus \{p\}}[ppq]^{(n)}\bigg(3\frac{\bar{x}_q^{(n)}}{\bar{x}_p^{(n)}} -4\bigg)\Bigg\} = 4d_p\,c_p^{(\infty)} +[ppp]^{(\infty)} \,\, .
$$
By \eqref{eq:profAGAG21}, the limit on the left-hand side is non-positive, and so we obtain
\begin{equation} \label{eq:profAGAG31}
\text{$p, q \in J_1$} \quad\Longrightarrow\quad c_p^{(\infty)} = 0\,,\,\,[ppq]^{(\infty)} = 0 \,\, .
\end{equation}
Therefore, claim $P(1)$ follows from \eqref{eq:profAGAG11}, \eqref{eq:profAGAG12}, \eqref{eq:profAGAG21}, \eqref{eq:profAGAG31}.
\end{proof}

Fix now $h \in \{1,{\dots},u-1\}$ and assume that $P(h')$ holds true for each $1 \leq h' \leq h$. By \cite[Lemma 5.55]{Boe04}, the scalar curvature can be estimated by
\begin{equation} \label{eq:CBestimate}
{\rm scal}\big(\bar{g}^{(n)}\big) \leq \frac12\sum_{i \in J_{> h}}\frac{d_ib_i^{(n)}}{\bar{x}^{(n)}_i} -\frac14\sum_{i,j,k \in J_{> h}}[ijk]^{(n)}\frac{\bar{x}^{(n)}_i}{\bar{x}^{(n)}_j \bar{x}^{(n)}_k} \,\, .
\end{equation}
By \eqref{eq:CBestimate} and arguing as above, there exists $C'>0$ such that
\begin{equation} \label{eq:ijkest-CB}
\sum_{i,j,k \in J_{> h}}[ijk]^{(n)}\frac{\bar{x}^{(n)}_i}{\bar{x}^{(n)}_j \bar{x}^{(n)}_k} \leq \frac{C'}{\bar{x}_p^{(n)}} \quad \text{ for any $p \in J_{h+1}$, for any $n \in \mathbb{N}$ }.
\end{equation}
As a consequence of \eqref{eq:ijkest-CB} and the inductive hypothesis:
\begin{equation} \label{eq:claim-subalg}
\text{$[ijk]^{(\infty)}>0$ \,\,with\,\, $i,j \in J_{\leq h+1}$,\, $k \in \{1,{\dots},\ell\}$ \quad $\Longrightarrow$ \quad $k \in J_{\leq h+1}$ .}
\end{equation}
Indeed, we distinguish the following cases. \begin{itemize}
\item[$\bcdot$] If $i,j \in J_{\leq h}$, then \eqref{eq:claim-subalg} follows by inductive hypothesis.
\item[$\bcdot$] Assume by contradiction that there exist $i \in J_{\leq h+1}$, $j \in J_{h+1}$ and $k \in J_{> h+1}$ such that $[ijk]^{(\infty)}>0$. Then, by \eqref{eq:ijkest-CB}, up to changing the constant, we obtain
$$
\frac{\bar{x}^{(n)}_k}{\bar{x}^{(n)}_i} \leq C' \,\, ,
$$
which is not possible.
\end{itemize}
Therefore, \eqref{eq:claim-subalg} implies that $\mathfrak{a}_{h+1}^{(\infty)}$ is a subalgebra of $\mathfrak{g}$. Hence, to complete the proof of $P(h+1)$, one can use the same steps as in the proof of $P(1)$. This concludes the proof of claims $iii)$ and $iv)$.

Finally, to prove claim $v)$, we assume by contradiction that $\bar{x}_{s_{\xi}+1}^{(n)} \to +\infty$ as $n \to +\infty$. By \eqref{eq:estscal}, it follows that ${\rm scal}\big(\bar{g}^{(n)}\big) \geq c^{-1} >0$. Therefore, by \eqref{eq:CBestimate}, we have
$$
\frac{4}{c} +\sum_{i,j,k > s_{\xi}}[ijk]^{(n)} \frac{\bar{x}^{(n)}_i}{\bar{x}^{(n)}_j \bar{x}^{(n)}_k} \leq \frac{2\,{\rm Tr}_Q(-\mathcal{B}_{\mathfrak{g}})}{\bar{x}_{s_{\xi}+1}^{(n)}} \,\, ,
$$
which is not possible since the left-hand side is positive and bounded away from zero.

\end{document}